%% file: July2025_waterfall.tex
\documentclass[letterpaper,11pt,oneside,reqno]{article}

\usepackage{amsmath,amssymb,amsthm,amsfonts, mathtools}
\usepackage[final,pdftex,backref=page,colorlinks=true,linkcolor=blue,citecolor=red]{hyperref}
\usepackage[alphabetic,nobysame]{amsrefs}

\usepackage{graphicx,color}
\usepackage{upgreek}
\usepackage[mathscr]{euscript}

\allowdisplaybreaks
\numberwithin{equation}{section}

\usepackage{tikz}
\usetikzlibrary{shapes,arrows,positioning,decorations.markings}

\usepackage{array}
\usepackage{adjustbox}
\usepackage{cleveref}
\usepackage{enumerate}
\usepackage{caption}
\usepackage{esint}

\usepackage[final]{listings}
\usepackage[DIV=12]{typearea}
\newcommand{\ssp}{\hspace{1pt}}

\newtheorem{proposition}{Proposition}[section]
\newtheorem{lemma}[proposition]{Lemma}
\newtheorem{corollary}[proposition]{Corollary}
\newtheorem{theorem}[proposition]{Theorem}
\newtheorem{conjecture}[proposition]{Conjecture}
\crefname{conjecture}{conjecture}{conjectures}
\Crefname{conjecture}{Conjecture}{Conjectures}
\crefname{lstlisting}{listing}{listings}
\Crefname{lstlisting}{Listing}{Listings}
\theoremstyle{definition}
\newtheorem{definition}[proposition]{Definition}
\newtheorem{remark}[proposition]{Remark}

\begin{document}
\title{Random Lozenge Waterfall:\\Dimensional Collapse of Gibbs Measures}

\author{Alisa Knizel and Leonid Petrov}

\date{}

\maketitle

\begin{abstract}
We investigate the asymptotic behavior of the $q$-Racah
probability measure on lozenge tilings of a hexagon whose
side lengths scale linearly with a parameter $L\to\infty$,
while the parameters $q\in(0,1)$ and $\kappa\in
\mathbf{i}\mathbb{R}$ remain fixed.
This regime differs fundamentally
from the traditional case $q\sim e^{-c/L}\to1$, in which
random tilings are locally governed by two-dimensional
translation-invariant ergodic Gibbs measures.
In
the fixed-$q$ regime we uncover a new macroscopic phase, the
\emph{waterfall} (previously only observed experimentally),
where the two-dimensional Gibbs structure collapses into a
one-dimensional random stepped interface that we call a
\emph{barcode}.

We prove a law of large numbers and
exponential concentration, showing that the random tilings
converge to a deterministic waterfall profile.
We further conjecture an explicit
correlation kernel of the one-dimensional barcode process
arising in the limit.
Remarkably, the limit is invariant under shifts by
$2\mathbb{Z}$ but not by $\mathbb{Z}$, exhibiting an
emergent period-two structure absent from the original
weights.
Our conjectures are supported by extensive numerical
evidence and perfect sampling simulations.
The kernel is built from a family of
functions orthogonal in both spaces
$\ell^{2}(\mathbb{Z})$
and $\ell^{2}(\mathbb{Z}+\frac12)$,
that may be of independent interest.

Our proofs adapt the spectral projection method of
Borodin--Gorin--Rains (2009) to the regime with fixed~$q$.
The resulting asymptotic analysis is substantially more
involved, and leads to non-self-adjoint operators. We
overcome these challenges in the exponential concentration
result by a separate argument based on sharp bounds for the
ratios of probabilities under the $q$-Racah orthogonal polynomial ensemble.
\end{abstract}

\section{Introduction}
\label{sec:intro}

\subsection{Overview}
\label{sub:overview_intro}

Random dimer coverings of graphs --- most notably, random
tilings of planar regions ---
form an exactly solvable
setup
for understanding phase transitions and emergent large-scale
behavior in two-dimensional statistical mechanics.
A prototypical example is uniformly random lozenge tilings of
the hexagon (see \Cref{fig:lozenge_hexagon_intro}).
When the side lengths of the hexagonal region grow
proportionally to a large parameter~$L$, several universal
regimes emerge:
\textbf{(i)}~a deterministic limit shape with Gaussian Free Field fluctuations
\cite{CohnKenyonPropp2000},
\cite{Kenyon2004Height},
\cite{Petrov2012GFF},
\cite{bufetov2016fluctuations},
\cite{berggren2024perfect_hexagon};
\textbf{(ii)}~edge statistics governed by Random-Matrix and Pearcey-type kernels
\cite{OkounkovReshetikhin2006RandomMatr},
\cite{GorinPanova2012_full},
\cite{Okounkov2005},
\cite{aggarwal2022gaussian},
\cite{huang2024pearcey};
\textbf{(iii)}~Kardar-Parisi-Zhang class behavior near arctic boundaries
\cite{BKMM2003},
\cite{Petrov2012},
\cite{huang2024concentration},
\cite{aggarwalhuang2025airy};
and \textbf{(iv)}~bulk lattice limits described by translation-invariant ergodic Gibbs measures
\cite{okounkov2003correlation},
\cite{Sheffield2008},
\cite{KOS2006},
\cite{Gorin2007Hexagon},
\cite{aggarwal2019universality}.
For a broader survey, see \cite{gorin2021lectures}.

As the extensive list of references demonstrates, a diverse toolkit
has emerged for studying
uniformly random lozenge tilings of the hexagon, including
the variational principle, orthogonal polynomial methods
(Riemann--Hilbert or spectral projection approaches),
determinantal formulas and their asymptotics, 
loop (Nekrasov) equations,
and discrete
complex analysis.
Many of these methods survive once the uniform measure is
tilted by weights, provided the weights preserve the
underlying orthogonal polynomial structure.
We study the most general family of weights of this kind,
namely the \emph{$q$-Racah weights} introduced in
\cite{borodin-gr2009q}; they correspond to a terminal node
of the $q$-Askey scheme \cite{Koekoek1996}. (Other
deformations, for example the doubly-periodic weights
that necessitate matrix-valued orthogonal polynomials
\cite{charlier2020doubly}, lie outside our scope.)

\begin{figure}[htpb]
\centering
\includegraphics[height=0.27\textwidth]{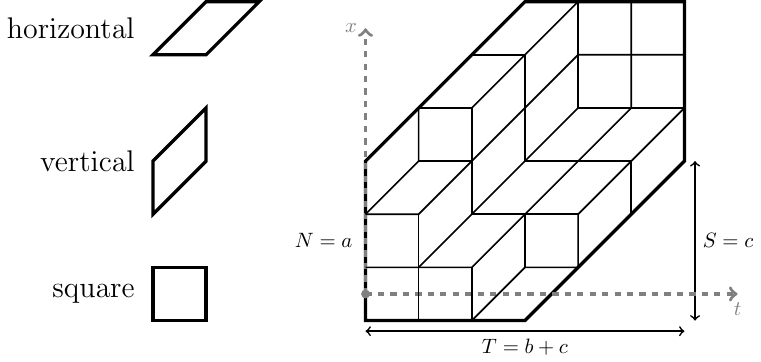}
\qquad
\includegraphics[height=0.28\textwidth]{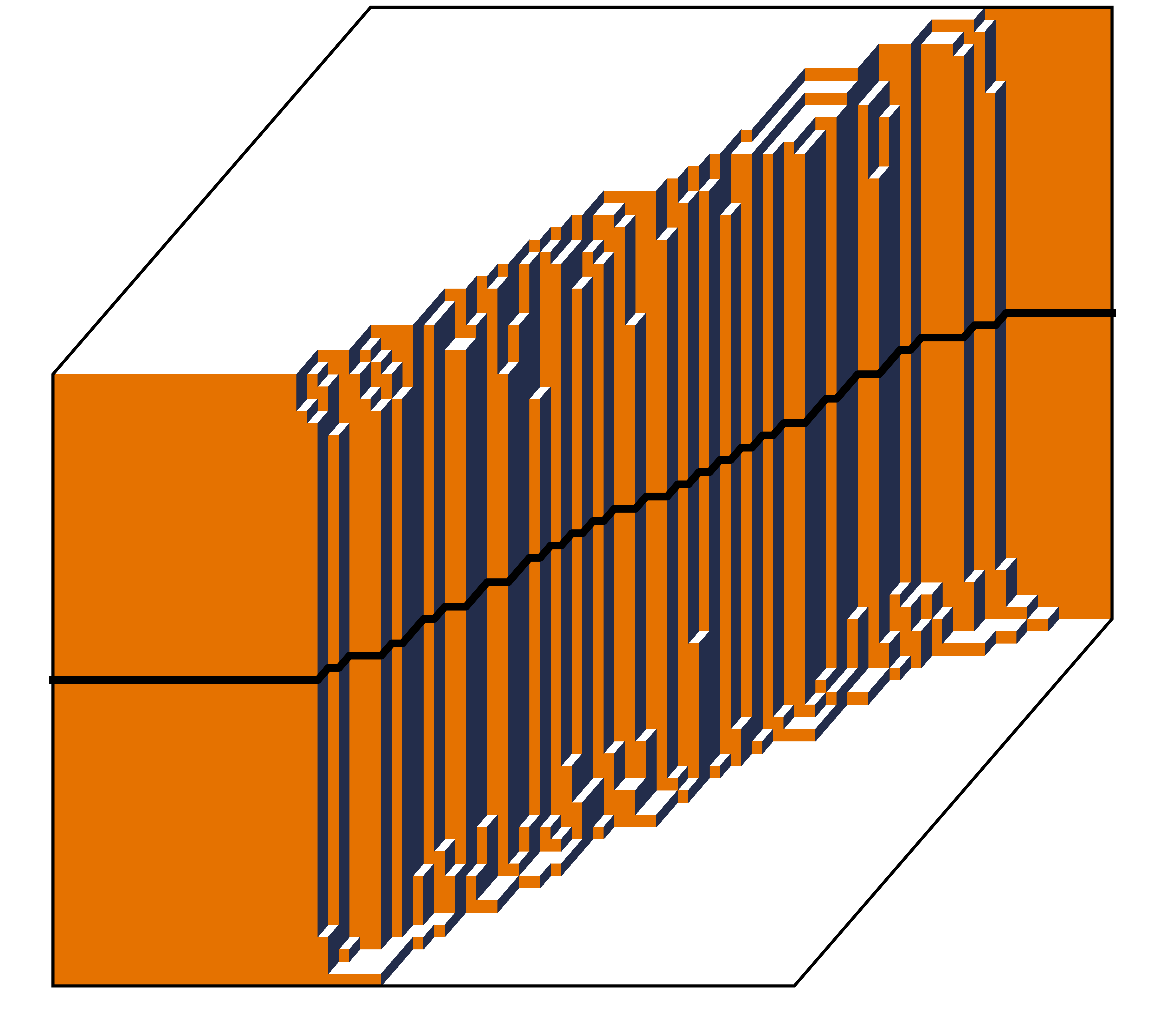}
\caption{Left: the three types of lozenges. Center: an example of a lozenge tiling of a hexagon whose side lengths are all equal to~$3$.
Observe that the plane can be affine transformed so that all lozenges become congruent;
however, it is convenient throughout the paper to use the equivalent representation shown here.
The sides of the hexagon are denoted by $a,b,c,a,b,c$, but throughout the paper we use the parameters
$N=a$, $T=b+c$, and $S=c$. Right: a perfect sample of the $q$-Racah random tiling with
$N=50$, $T=100$, $S=30$, $q=0.7$, and $\kappa=3\mathbf{i}$ obtained by our \texttt{Python} port
of the original algorithm from \cite{borodin-gr2009q}
(provided as
an ancillary file to the arXiv version of the paper). 
The cross-section of the 3D surface across the 
middle represents the barcode process.}
\label{fig:lozenge_hexagon_intro}
\end{figure}

The $q$-Racah probability measure on lozenge tilings of a
hexagon is defined by assigning the probability
\begin{equation}
	\label{eq:q_racah_weight_intro}
	\frac{1}{Z}
	\prod_{\substack{\textnormal{horizontal}\\\textnormal{lozenges $u$}}} \mathsf{w}_{q,\kappa}
	\bigl(\operatorname{height}(u)\bigr),
	\qquad
	\mathsf{w}_{q,\kappa}(j)\coloneqq \kappa q^{j-(S+1)/2}-\frac{1}{\kappa q^{j-(S+1)/2}},
\end{equation}
to each tiling, where $Z$ is the normalizing constant, $S=c$ is one of the sides of the hexagon,
and the height of a horizontal lozenge is defined
as the number of $1\times 1\times 1$ cubes under this lozenge
in the representation of a tiling as a 3D stepped surface.
The parameters are\footnote{That is, we consider the
so-called \emph{imaginary $q$-Racah case}. There are two
other cases, \emph{real} and \emph{trigonometric},
introduced in \cite[Section~2.2]{borodin-gr2009q}.
Throughout the paper, we focus only on the imaginary case since it
leads to the waterfall behavior.}
$\kappa\in \mathbf{i}\mathbb{R}$
and $q>0$. Note that all of the factors $\mathsf{w}_{q,\kappa}(j)$ are purely imaginary with positive imaginary part, 
and after the normalization, \eqref{eq:q_racah_weight_intro} defines an honest probability measure.

Because the total number of horizontal lozenges is fixed at
$S(T{-}S)$, the involution
$(q,\kappa)\mapsto(q^{-1},\kappa^{-1})$ leaves it invariant.
We therefore fix the convention $0<q<1$ throughout.
The limits $\kappa\to0$ or $\kappa\to \mathbf{i}\infty$
reduce the $q$-Racah measure to the volume-tilted measures
$q^{\mathsf{vol}}$ and $q^{-\mathsf{vol}}$, respectively, where
$\mathrm{vol}$ is the three-dimensional volume under the stepped surface.
The further limit $q\to1$ gives back the uniform measure.

\medskip

Until now, asymptotic investigations of the $q$-Racah model
in the regime where the side lengths of the hexagon grow
linearly with a parameter \(L\to\infty\) have focused
exclusively on the \emph{traditional $q\to1$ window}, i.e.\ the scaling \(q=e^{-c/L}\to1\).  Using the
spectral projection framework developed in
\cite{borodin2007asymptotics}, \cite{Olshansk2008-difference} (see
also \cite[Chapter~3.3]{TaoRMbook}), Borodin--Gorin--Rains
\cite{borodin-gr2009q} identified the local (bulk) lattice
limits in this window with the same translation-invariant
ergodic Gibbs measures present in the uniform case.
Subsequent works by Dimitrov--Knizel
\cite{dimitrov2019log} and Gorin--Huang
\cite{gorin2022dynamical} established Gaussian Free Field
fluctuations of the height function via two complementary
loop-equation approaches, while Duits--Duse--Liu
\cite{Duits2024lozenge} confirmed the same universal
behavior through an analysis of the recurrence coefficients
of the underlying orthogonal polynomials.

By contrast, in the \emph{fixed-$q$ regime}
almost nothing is rigorously
understood. Aside from the striking \emph{waterfall} phase
observed in simulations in \cite{borodin-gr2009q}
(see \Cref{fig:lozenge_hexagon_intro}, right, for an illustration),
the asymptotic behavior for fixed $q$ and $\kappa$ has so far
escaped a mathematical description.
(For the simpler \(q^{\pm\mathrm{vol}}\) measures, the
limiting objects were connected in
\cite{OkounkovKenyon2007Limit} to tropical curves, but those
results do not extend to the $q$-Racah setting.)
Describing the fixed-$q$
regime and the waterfall phase is the
central aim of the present work.

\medskip
Our \textbf{main results} are as follows.
\begin{enumerate}[$\bullet$]
  \item We prove that, in the fixed-$q$ regime, the $q$-Racah random tiling converges
		(with exponential bounds)
		to a deterministic \emph{waterfall profile}.
		In the waterfall region $\mathcal{W}$
		(a part of this profile in which locally one sees
		two types of
		lozenges),
		the two-dimensional lattice behavior
		\emph{collapses} into a one-dimensional random stepped interface
		(which we call a \emph{barcode}; this is the central broken line
		in the tiling in \Cref{fig:lozenge_hexagon_intro}, right).

  \item
		Extending the
		spectral projection framework to the fixed-$q$ regime,
		we obtain (coefficient-wise) limits of the
		relevant difference operators at every macroscopic
		location.
		Away from the center of the waterfall region, i.e.\ when
		$x\not\sim \tfrac12(S+t)$, the limiting operators are no
		longer self-adjoint. A systematic study of these
		non-self-adjoint operators is left for future work.

	\item
		By matching nonrigorous computations with numerical and
		perfect sampling data, we propose a determinantal
		correlation kernel for the one-dimensional
		\emph{barcode process}.
		The conjectural kernel exhibits a surprising
		$2\mathbb Z$-periodicity, breaking the original lattice
		homogeneity.
\end{enumerate}

\begin{remark}
	Most tiling models feature degenerate regions where the
	two-dimensional Gibbs structure disappears.  The familiar
	examples are the \emph{frozen facets} outside an arctic
	boundary, in which the 3D stepped surface is parallel to one of the
	coordinate planes, and exhibits no fluctuations.
	Certain boundary conditions or periodic weights can also
	give rise to \emph{semi-frozen} zones
	\cite{Mkrtchyan2014Periodic}, \cite{Mkrtchyan2019}, 
	\cite{BerggrenBorodin2023},
	in which locally
	only two types of lozenges appear, producing a barcode-type
	configuration that is, however, fully deterministic and
	periodic.

	The waterfall phase we uncover is fundamentally different.
	Here, the bulk two-dimensional Gibbs structure \emph{collapses} onto a single
	random stepped interface (the barcode), while the complementary regions remain fully frozen.
	To our knowledge, no prior tiling model with nonrandom weights exhibits a
	macroscopic facet that is simultaneously
	random and one-dimensional.

	We remark that an effectively one-dimensional random phase,
	in which two types of lozenges are arranged as steps of a
	random walk, is expected to occur in domino tilings of the
	Aztec diamond with the \emph{random} edge weights introduced
	in \cite{BufetovPetrovZografos2025}. The
	randomness of the weights is essential for this effect, whereas in our
	model the barcode phase appears under a completely
	deterministic choice of weights.
\end{remark}

\begin{remark}
The waterfall phase in our model represents a strong form of
dimensional reduction:
outside the waterfall region, all
fluctuations are frozen, and only the one-dimensional
barcode interface survives in the waterfall region.
This
can be contrasted with other known instances of dimensional
reduction in statistical mechanics, such as in the scaling
limit of the two-dimensional critical Ising model
\cite{WuMcCoyTracyBarouch1976}. In the Ising case, while the
correlation functions are famously described by
lower-dimensional mathematical structures, the underlying
two-dimensional field continues to exhibit non-trivial
fluctuations in both spatial dimensions.
\end{remark}

In the remainder of the Introduction, we provide a detailed
discussion of our methods and their limitations.

\subsection{Spectral projection approach and its limitations}
\label{sub:spectral_projection_intro}

Let $X=\{(t_i,x_i)\}\subset\mathbb Z^2$ be the random point configuration marking the
positions of the vertical and square lozenges in a tiling of the hexagon
(where $t$ and $x$ are the lattice coordinates defined as in
\Cref{fig:lozenge_hexagon_intro}, center).
For pairwise distinct $(s_1,x_1),\dots,(s_n,x_n)$, define the $n$-point correlation function
\[
\rho_m\bigl((s_1,x_1),\dots,(s_m,x_m)\bigr)
  \coloneqq\mathbb{P}\bigl\{(s_i,x_i)\in X, \ i=1,\ldots,m \bigr\},
\]
where $\mathbb{P}$ is the $q$-Racah probability measure.
Because this measure is
equivalent to a dimer model on the bipartite hexagonal
graph, the resulting correlation functions are
determinantal; see the lecture notes
\cite{Kenyon2007Lecture} for a general exposition.
That is, there exists a
kernel $K(s,x;t,y)$ such that
\[
\rho_m\bigl((s_1,x_1),\dots,(s_m,x_m)\bigr)
	=\det\bigl[K(s_i,x_i;s_j,x_j)\bigr]_{i,j=1}^m,
	\qquad m\ge1.
\]
As shown in \cite{borodin-gr2009q},
the kernel is
given in terms of the $q$-Racah orthogonal polynomials:
\begin{equation}
	\label{eq:qRacah_two-dimensional_kernel_intro}
	{K}(s,x;t,y)=
	\begin{cases}
		\displaystyle
		\sum_{n=0}^{N-1}
		(C_n^{t} C_n^{t+1}\ldots C_n^{s-1})^{-1}\ssp
		f_n^s(x)\ssp f_n^t(y),&\qquad  s\ge t;\\[12pt]
		\displaystyle
		-\sum_{n\ge N}
		(C_n^s C_n^{s+1}\ldots C_n^{t-1})\ssp
		f_n^s(x)\ssp f_n^t(y)
		,&\qquad  s<t,
	\end{cases}
\end{equation}
where the constants $C_n^{t}$ are explicit, and the
functions $f_n^{t}(x)$ form an orthonormal basis of the
space $\ell^{2}$ on the $t$-th vertical slice.
The
functions $f_n^{t}(x)$ are obtained from the relevant
$q$-Racah orthogonal polynomials (with parameters depending
on $q$, $\kappa$, $t$, and the side lengths of the
hexagon) by multiplying by the square root of the $q$-Racah weight function, and
then normalising. Complete formulas are given in
\Cref{sec:q_racah_tilings,sec:2d_correlations}.

On each fixed vertical slice $t$, the kernel
$K_t(x,y)\coloneqq K(t,x;t,y)=\sum_{n=0}^{N-1} f_n^t(x)f_n^t(y)$
can be interpreted as the
\emph{orthogonal spectral projection}
onto the positive part of the spectrum of a certain difference operator
with explicit product-form coefficients,
see
\eqref{eq:diff_operator_scaling_shifting} and
\eqref{cor:qRacah_hexagon_kernel_spectral_description}.
We take the limits of these coefficients for fixed $q$ and $\kappa$
as $T=\lfloor L\ssp \mathsf{T} \rfloor$, $S=\lfloor L\ssp \mathsf{S} \rfloor$, and $N=\lfloor L\ssp \mathsf{N} \rfloor$
and around a macroscopic location $(\mathsf{t},\mathsf{x})$,
where $t=\lfloor  L\ssp \mathsf{t} \rfloor$ and $x=\lfloor L\ssp \mathsf{x} \rfloor+\Delta x$.
The waterfall phase arises only when $\mathsf{N}<\mathsf{T}$, 
and we assume this condition throughout the rest of the Introduction.
Depending on the scaled location $(\mathsf{t},\mathsf{x})$ on a slice $\mathsf{t}$ intersecting the waterfall
region $\mathcal{W}$, we obtain various
limiting symmetric Jacobi (tridiagonal) operators
which act in the local shift $\Delta x$:
\begin{equation}
\label{eq:limit_of_difference_operator_intro_general_form}
	\bigl(T g\bigr)(\Delta x)=d(\Delta x+1)\ssp g(\Delta x+1)
	+
	a(\Delta x)\ssp g(\Delta x)+
	d(\Delta x)\ssp g(\Delta x-1),
\end{equation}
Their coefficients are given in the following proposition.

\begin{proposition}[\Cref{lemma:D_scaled_limit_of_diagonal_coefficient,lemma:D_scaled_limit_of_off_diagonal_coefficient}]
	\label{prop:limit_of_difference_operator_intro}
	\begin{enumerate}[$\bullet$]
		\item
			If $(\mathsf{t},\mathsf{x})$ is above or below $\mathcal{W}$, we have $d(\Delta x)=0$ and $a(\Delta x)=-1$, that is, the
			limiting difference operator is the negative identity $(-Id)$;
		\item If $(\mathsf{t},\mathsf{x})$ lies on the upper boundary of $\mathcal{W}$,
			then
			\begin{equation*}
				d(\Delta x)=-\kappa^{2} q^{2\Delta x+1/2},
				\qquad
				a(\Delta x)=-1-\kappa^{2}(1+q)q^{2\Delta x+1}.
			\end{equation*}
			On the lower boundary of $\mathcal{W}$, 
			one has to 
			replace the powers of $(q,\kappa)$ with those of
			$(q^{-1},\kappa^{-1})$, while the factor $1+q$ is kept
			the same.
		\item 
			Inside $\mathcal{W}$ but away
			from its center line $\lfloor
			L\ssp\mathsf{x}\rfloor=\tfrac12(S+t)$, the
			coefficients in the upper half are
			\begin{equation*}
				d(\Delta x)=-\kappa^{2} q^{2\Delta x+1/2},
				\qquad 
				a(\Delta x)=-\kappa^{2}(1+q)q^{2\Delta x+1}.
			\end{equation*}
			The formulas for the lower part are obtained from these by the same symmetry as in the preceding item.
		\item
			On the center line, we get
			\begin{equation}
				\label{eq:limit_of_difference_operator_center_line_intro}
				d(\Delta x)=\dfrac{-\kappa^2(1 - \kappa^2 q^{2\Delta x})^{-1} q^{2\Delta x+1/2}}{
				\sqrt{(1 - \kappa^2 q^{2\Delta x-1}) (1 - \kappa^2 q^{2\Delta x+1})}}
				,\qquad
				a(\Delta x)=
				\dfrac{-\kappa^2 (1+q) \ssp q^{2\Delta x+1}}{(1-\kappa^2 q^{2 \Delta x})(1-\kappa^2 q^{2\Delta x+2})}.
			\end{equation}
	\end{enumerate}
\end{proposition}

If the scaled point $(\mathsf t,\mathsf x)$ is either
outside the waterfall region $\mathcal{W}$ or sits on the
center line, the
pre-limit difference operator
converges to a \emph{bounded, self-adjoint} limit
--- either $(-Id)$, or the operator given by 
\eqref{eq:limit_of_difference_operator_intro_general_form}--\eqref{eq:limit_of_difference_operator_center_line_intro}.
Then, by the spectral projection method based on the strong
resolvent convergence (developed in 
\cite{borodin2007asymptotics}, \cite{Olshansk2008-difference}; we recall the main statement 
in \Cref{thm:funkan}),
the correlation kernel $K_t$ converges to 
zero outside $\mathcal{W}$
and to the identity operator on the center line.
This result essentially parallels the 
spectral approach that proved successful in the traditional regime
$q=e^{-c/L}\to1$ in \cite{borodin-gr2009q}, where it
yielded the limiting kernel at \emph{all} macroscopic
points inside the hexagon.

However, for fixed $q$, on the boundary of $\mathcal{W}$ and inside of $\mathcal{W}$ but not on the center line, the coefficients of 
\eqref{eq:limit_of_difference_operator_intro_general_form} grow rapidly at positive or negative infinity, so the
limiting difference operator is 
no longer self-adjoint and in fact has von Neumann deficiency indices $(1,1)$.
The non-uniqueness of self-adjoint extensions of 
prevents us from singling out a canonical spectral projection.
The most intriguing case is the boundary of $\mathcal{W}$, where 
one can see a nontrivial random configuration with all three types of lozenges --- 
thus, the spectral projection there should yield a new correlation kernel that is not zero nor the identity operator.
We leave the full functional-analytic treatment of these non-self-adjoint operators for future work.

Although the spectral projection approach breaks down inside the waterfall region away from the center line,
we bridge this gap by performing direct
estimates for the $q$-Racah orthogonal polynomial ensemble
(see \Cref{sub:exponential_concentration_intro} below).
These estimates strengthen the kernel convergence along the
center line to an exponential concentration statement, which
is the main result of the present paper.

\begin{theorem}[Combination of \Cref{prop:limit_of_kernel}, 
\Cref{thm:no_holes_in_waterfall_region}, and
\Cref{cor:kernel_goes_to_identity}]
	\label{thm:waterfall_LLN_intro}
	Assume that the side lengths of the hexagon, $a,b,c$,
	scale proportionally with a parameter $L\to+\infty$, while
	the $q$-Racah parameters $\kappa$ and $q$ remain fixed.
	Then the fixed-slice correlation kernel
	converges either to zero or the identity operator,
	depending on the scaled location $(\mathsf{t},\mathsf{x})$:\ssp\footnote{Here
	and throughout the paper, $\mathbf{1}_A$ stands for the indicator of an event or condition $A$.}
	\begin{equation*}
		\lim_{L\to+\infty}
		K_{\lfloor L\ssp \mathsf{t} \rfloor}
		(
		\lfloor L\ssp \mathsf{x}\rfloor+\Delta x,
		\lfloor L\ssp \mathsf{x}\rfloor+\Delta y)
		=
		\begin{cases}
			\mathbf{1}_{\Delta x=\Delta y}, &\textnormal{if $(\mathsf{t},\mathsf{x})\in \mathcal{W}$};
			\\
			0,&\textnormal{otherwise,}
		\end{cases}
		\quad \textnormal{for all fixed
		$\Delta x,\Delta y\in \mathbb{Z}$}.
	\end{equation*}
	Moreover, the probability of finding a
	horizontal lozenge inside $\mathcal{W}$ or a square or vertical
	lozenge outside $\mathcal{W}$ decays exponentially fast in $L$.
\end{theorem}

\subsection{Inter-slice transitions and the conjectural barcode kernel}
\label{sub:barcode_kernel_intro}

To access inter-slice correlations and the barcode process, one 
has to look at the full two-dimensional kernel
$K(s,x;t,y)$. Thanks to the structure of the $q$-Racah orthogonal polynomials,
the coefficients $C_n^t$ in 
\eqref{eq:qRacah_two-dimensional_kernel_intro} 
can be incorporated into the action of two-diagonal 
inter-slice operators~$\mathfrak{U}_t$:
\begin{equation}
	\label{eq:inter_slice_prelimit_operator_intro}
	\frac{\sqrt{\lambda(t)}}{\sqrt{\lambda(t+1)}}\ssp C_n^t \ssp f_n^{t+1}(x)=\bigl(\mathfrak{U}_t f_n^{t}\bigr)(x),
	\qquad 
	\bigl(\mathfrak{U}_t \ssp g\bigr)(x)\coloneqq
	u_1^t(x-1)\ssp g(x-1)+u_0^t(x)\ssp g(x),
\end{equation}
where $\lambda(\cdot)$
is a nonvanishing explicit gauge factor 
\eqref{eq:lambda_t_definition}
that does not change the determinantal process,
and the 
coefficients $u_0^t(x)$ and $u_1^t(x)$ 
have explicit product-form expressions,
see
\Cref{sub:2d_correlations_transitions_new}
for details.

We represent the $q$-Racah correlation kernel
\eqref{eq:qRacah_two-dimensional_kernel_intro}
as 
\begin{equation}
	\label{eq:K_sx_ty_intro_as_operators}
		\frac{\sqrt{\lambda(s)}}{\sqrt{\lambda(t)}}\ssp
		K(s,x;t,y)=
		\begin{cases}
			\displaystyle
			\left(
				\mathfrak{U}_t^{-1}
				\mathfrak{U}_{t+1}^{-1}
				\ldots
				\mathfrak{U}_{s-1}^{-1}
			\right)_y
			K_s(x,y)
			,&\qquad  s> t;\\[4pt]
			K_t(x,y),&\qquad s=t;\\[2pt]
			\displaystyle
			\left(
				\mathfrak{U}_{t-1}
				\ldots
				\mathfrak{U}_{s+1}
					\ssp
				\mathfrak{U}_{s}
			\right)_y
			\left( -\mathbf{1}_{x=y}+K_s(x,y) \right)
			,&\qquad  s<t.
		\end{cases}
\end{equation}
where the operators $\mathfrak{U}_{\bullet}$ act in the variable $y$.

Fixed-$q$ limits
of the coefficients $u_0^t(x)$ and $u_1^t(x)$
of the inter-slice operators $\mathfrak{U}_t$ \eqref{eq:inter_slice_prelimit_operator_intro}
are immediately accessible.
Taking the limit in the operators $\mathfrak{U}_t$ 
(together with rescaling and shifting to the required macroscopic location), we arrive at
the operator 
\begin{equation*}
	\bigl(\mathfrak{U}^{\mathrm{barcode}}_t f\bigr)(x)
	\coloneqq
	\mathfrak{a}(x-\tfrac{t}{2})\ssp f(x-1)+\mathfrak{a}(x-\tfrac{t}{2}+\tfrac12)\ssp f(x),
\qquad \mathfrak{a}(x)\coloneqq
\sqrt{
\frac{-\kappa^2 q^{2x-1}}
{ (1-\kappa ^2 q^{2x})(1-\kappa ^2 q^{2x-1})}},
\end{equation*}
where 
$x$ belongs to $\mathbb{Z}$ or $\mathbb{Z}+\frac{1}{2}$.
Due to the rapid decay of the coefficients
$\mathfrak{a}(x)$
at $x\to\pm\infty$, the inverse of the operator $\mathfrak{U}^{\mathrm{barcode}}_t$ is
not well-defined. Therefore, directly taking the fixed-$q$ limit of the
kernel \eqref{eq:K_sx_ty_intro_as_operators}
is not possible.

However, the bounded 
operator $\mathfrak{U}^{\mathrm{barcode}}_t$ itself admits a family of functions
$\mathcal{F}_n^t(x)$, $n\in \frac{1}{2}\mathbb{Z}_{\ge0}$,
orthogonal in both spaces $\ell^2(\mathbb{Z})$ and $\ell^2(\mathbb{Z}+\frac{1}{2})$,
which satisfy an identity similar to \eqref{eq:inter_slice_prelimit_operator_intro}:
\begin{equation*}
	-q^n\ssp \mathcal{F}_n^{t+1}(x) = \bigl(\mathfrak{U}^{\mathrm{barcode}}_t \mathcal{F}_n^t\bigr)(x).
\end{equation*}
We develop properties of these functions in \Cref{appendix}.
In particular, they may
be expressed through the continuous $q^{-1}$-Hermite polynomials 
\cite[Chapter~3.26]{Koekoek1996}, \cite{IsmailMasson1994}.

Motivated by this structure, we first 
conjectured the following
series representation for the density function
of the barcode process:
\begin{equation}
	\label{eq:rho_barcode_intro_density}
	\rho^{\mathrm{barcode}}(t)=
	\mathfrak{a}(1-\tfrac t2)
	\sum_{n\in \frac12\mathbb{Z}_{\ge0}}
	\frac{-q^{-n}
	\mathcal{F}_n^t(0)\mathcal{F}_n^{t-1}(0)
	}{\|\mathcal{F}_n\|^2_{\ell^2(\mathbb{Z})}}.
\end{equation}
However, this series \textbf{does not converge}. Numerical
experiments indicate that its divergence is rather
\emph{mild}: the tail behaves as $\,\ldots
+c-c+c-c+\ldots\,$. Consequently, the series in
\eqref{eq:rho_barcode_intro_density} can be regarded as
having ``\textbf{two distinct sums},'' depending on how one
pairs adjacent terms. Further numerical calculations and
simulation studies confirm that the barcode process is
\textbf{two-periodic}, with the even- and odd-site densities arising
from these two alternative summations of
\eqref{eq:rho_barcode_intro_density}.
This periodicity is unexpected but can be explained a posteriori,
see \Cref{rmk:barcode_periodicity_intro} below.

Turning to the full correlation kernel, one would get even more divergent series which need
to be regularized.
Based on further numerics,
we
conjecture the existence of a universal
limiting \emph{barcode kernel} governing the local statistics 
inside the waterfall region.
This conjectural kernel is given for $s\ge t$ by 
\begin{equation}
	\label{eq:K_barcode_intro}
	\begin{split}
		\mathcal{K}^{\mathrm{barcode}}(s,t)\coloneqq
		\lim_{M\to\infty}
		q^{(M+1)(s-t)}&\times
		(-1)^{t-s}
		\sqrt{\frac{-\kappa^2q^{1-s}}{(1-\kappa^2 q^{2-t})(1-\kappa^2 q^{1-s})}}
		\\&\hspace{60pt}\times
		\sum_{n=0}^{M+1/2}
		\frac{\bigl(-q^{n}\bigr)^{t-1-s}\mathcal{F}_n^{s}(0)\mathcal{F}_n^{t-1}(0)}
		{\|\mathcal{F}_n\|^2_{\ell^2(\mathbb{Z})}},
	\end{split}
\end{equation}
where $M\in \mathbb{Z}$, and the sum is over $n\in \frac12\mathbb{Z}_{\ge0}$.
For $s<t$, the kernel is extended by symmetry. 
\begin{conjecture}[Combination of \Cref{conj:barcode_limit_via_F,conj:conjectural_limit_of_full_barcode_kernel,conj:barcode_density_limit}]
	\label{conj:barcode_kernel_intro}
	Around any macroscopic point $(\mathsf{t},\mathsf{x})$ in the waterfall region $\mathcal{W}$, the
	random configuration of square lozenges
	on a horizontal slice
	converges to a
	determinantal \textbf{barcode process} on $\mathbb{Z}$ with the symmetric correlation kernel
	$\mathcal{K}^{\mathrm{barcode}}(s,t)$.\footnote{In particular, 
	we conjecture that the limit \eqref{eq:K_barcode_intro} exists for all integers $s\ge t$.}
	The barcode process depends only on the parameters $q$ and $\kappa$,
	but not on the geometry of the hexagon or the macroscopic location $(\mathsf{t},\mathsf{x})$
	inside $\mathcal{W}$, as long as the waterfall phase is present, i.e., $\mathsf{N}<\mathsf{T}$.
	
	The limiting kernel is $2\times 2$ block Toeplitz, that is,
	\begin{equation*}
		\begin{pmatrix}
			\mathcal{K}^{\mathrm{barcode}}(s,t)& \mathcal{K}^{\mathrm{barcode}}(s,t+1)\\
			\mathcal{K}^{\mathrm{barcode}}(s+1,t) & \mathcal{K}^{\mathrm{barcode}}(s+1,t+1)
		\end{pmatrix}
		=
			\begin{pmatrix}
				\mathcal{K}^{\mathrm{barcode}}_{00}(s-t) & \mathcal{K}^{\mathrm{barcode}}_{01}(s-t)\\
				\mathcal{K}^{\mathrm{barcode}}_{10}(s-t) & \mathcal{K}^{\mathrm{barcode}}_{11}(s-t)
			\end{pmatrix},
			\qquad s,t\in 2\mathbb{Z},
	\end{equation*}
	and satisfies
	$\mathcal{K}^{\mathrm{barcode}}(0,0)=1-\mathcal{K}^{\mathrm{barcode}}(1,1)$ 
	(the plot of this function
	in $q$ and $\kappa$ is given in \Cref{fig:density_rho_even_3d}).
	Consequently, the barcode process has global density $1/2$
	for every choice of $q$ and $\kappa$, and it is invariant
	under even but not odd shifts.
\end{conjecture}

\begin{remark}
	\label{rmk:barcode_periodicity_intro}
	The period-two behaviour of the barcode process was not
	anticipated a priori, since the $q$-Racah weights
	$\mathsf{w}_{q,\kappa}$ \eqref{eq:q_racah_weight_intro}
	assigned to horizontal lozenges do not possess this
	periodicity.
	However, this property may be explained as follows.
	In every vertical slice the number of horizontal lozenges is
	determined by the geometry of the hexagon, and moving from
	slice $t$ to slice $t+1$ changes this number by exactly one.
	This extra lozenge is created either above or below the
	waterfall region, which introduces an imbalance.
	The imbalance is rectified after two slices, i.e., after an
	even translation, which accounts for the invariance of the
	barcode process under even shifts but not under all integer
	shifts.
\end{remark}

\Cref{sub:s7_numerics,sub:s7_probabilistic_simulations}
corroborate \Cref{conj:barcode_kernel_intro} in two ways. First, 
we find a very strong numerical agreement
between the conjectural barcode kernel $\mathcal{K}^{\mathrm{barcode}}$ 
and the 
pre-limit correlation kernel $K(s,x;t,y)$
for a large hexagon, regardless of its geometry and the macroscopic location $(\mathsf{t},\mathsf{x})$ inside the waterfall region $\mathcal{W}$
around which the limit is taken.
We also observe that the sequence \eqref{eq:K_barcode_intro}
converges with the geometric rate $q^M$, and, moreover, the (suitably rescaled) pre-limit 
kernels $K(s,x;t,y)$ also converge to a limit with the geometric rate $q^{L}$.
Second, we ported the perfect sampling algorithm from
\cite{borodin-gr2009q} to \texttt{Python} to automatically collect large samples of the barcode process.
We compare the sample even- and odd-site densities, as well as two-point correlations, with the predictions 
based on $\mathcal{K}^{\mathrm{barcode}}$, and find a very good agreement.

A rigorous derivation of $\mathcal K^{\mathrm{barcode}}$ and the proof of \Cref{conj:barcode_kernel_intro}
will likely require new tools for 
dealing with the inverse of the operator
$\mathfrak{U}^{\mathrm{barcode}}_t$.
We leave these analytic developments to future work.

\subsection{Exponential concentration}
\label{sub:exponential_concentration_intro}

Complementing the spectral projection approach,
in \Cref{sec:asymptotic_clustering_of_nonintersecting_paths}
we also develop a
concentration result on a given vertical slice
by working directly with the
$q$-Racah orthogonal polynomial ensemble:
\begin{equation}
	\label{eq:qRacah_OPE_intro}
	\frac{1}{Z}\prod_{1\le i<j\le N}\bigl(
		\mu(x_i)-\mu(x_j)
	\bigr)^2
	\prod_{i=1}^N
	w(x_i),
	\qquad 
	\mu(x)\coloneqq q^{-x}+\kappa^2 q^{x+1-S-t},
\end{equation}
where the ``particles'' $x_1,\ldots,x_N $ are the positions of the
non-horizontal lozenges,
and $w(x)$ is the $q$-Racah weight function (under which the 
$q$-Racah polynomials are orthogonal).

The key idea
is to compare the probabilities of two $N$-particle
configurations that differ only by moving a single 
particle from
$(t,x)$ to $(t,y)$.
The ratio of these probabilities 
exposes a dominant
factor $q^{\mathscr E(x,y)}$, where the exponent
decomposes into a one-body term $\mathscr W(u)$
(integrated between $x$ and $y$) and 
an interaction term $\mathscr G(x,y,z_i)$, where $z_i$ are the positions of the other particles
that are not moved.
We 
show that the interaction is minimized when the other
$N-1$ particles form a densely packed block around the center line.

After this minimisation, the exponent $\mathscr E(x,y)$
becomes a discrete integral of a piecewise linear
function which 
has order $1/L$
inside the waterfall region. This allows us to
show that moving a single particle closer to the center line
is advantageous of order $q^{-cL}$ close to the boundary of the
waterfall region, and of order $q^{-c'L^2}$ when move is across the waterfall boundary.
These estimates translate into the exponential concentration of 
square and vertical lozenges inside the waterfall region, and
of the horizontal lozenges outside the waterfall region.
Together with the
spectral projection results, this establishes
\Cref{thm:waterfall_LLN_intro} for the entire hexagon.

\subsection{Outline of the paper}
\label{sub:outline}

In 
\Cref{sec:q_racah_tilings} 
we collect the necessary definitions related to the $q$-Racah random lozenge tilings,
their connection with $q$-Racah orthogonal polynomials, and the
spectral projection interpretation of the fixed-slice correlation kernel.
We mostly follow the notation of \cite{borodin-gr2009q}.
In
\Cref{sec:scaling_regime} we introduce the fixed-$q$ scaling regime,
and discuss how it differs from the traditional scaling $q\to1$.
We also outline the heuristics for the waterfall phenomenon
which are made rigorous in the subsequent \Cref{sec:behavior_on_slice_via_spectral_projections,sec:asymptotic_clustering_of_nonintersecting_paths}.
Namely, in \Cref{sec:behavior_on_slice_via_spectral_projections} we employ the spectral projection method to
get the limit of the correlation kernel at the center line and outside the waterfall region,
and in \Cref{sec:asymptotic_clustering_of_nonintersecting_paths} we
complement the analysis by direct estimates in the $q$-Racah orthogonal polynomial ensemble,
leading to the exponential concentration result.
These sections complete the proof of our main result, \Cref{thm:waterfall_LLN_intro}.

In \Cref{sec:2d_correlations} we review the two-dimensional
correlation structure of $q$-Racah random lozenge tilings
and set the stage for a subsequent nonrigorous analysis of
their asymptotic behavior that culminates in the conjectural
barcode kernel.  The presentation in this section is
completely rigorous; in particular, we derive asymptotic
formulas for the inter-slice transition kernel that yield
the limiting operator $\mathfrak{U}^{\mathrm{barcode}}_{t}$.
\Cref{sec:2d_correlations_asymptotics_numerics_conjectures} then combines a heuristic derivation of the barcode kernel with numerical and probabilistic evidence.  We first carry out an informal asymptotic analysis explaining the eventual form of the kernel, and afterwards corroborate the predictions by means of exact-formula numerics and perfect sampling data.

\Cref{appendix} collects several properties of the functions $\mathcal{F}_n^t(x)$ 
appearing in the barcode kernel, which may be of independent interest.
\Cref{sec:Mathematica_code_appendix} contains the \texttt{Mathematica} code used for the numerical analysis of the 
exact formulas performed in 
\Cref{sub:s7_numerics}.

\subsection{Acknowledgements}
\label{sub:acknowledgements}

We are grateful to Vadim Gorin for valuable remarks and for sharing with us the original \texttt{Fortran} source code implementing the shuffling algorithm for $q$-Racah random lozenge tilings described in \cite[Section~6]{borodin-gr2009q}.
AK was partially supported by the NSF grant DMS-2348756.
LP was partially supported by the NSF grant DMS-2153869 and the Simons Collaboration Grant for Mathematicians 709055.
Part of this research was performed in Spring 2024 while
LP was
visiting the
program
``Geometry, Statistical Mechanics, and Integrability''
at the Institute for Pure and Applied Mathematics
(IPAM), which is supported by the NSF grant DMS-1925919.

\section{Preliminaries on $q$-Racah random tilings}
\label{sec:q_racah_tilings}

In this section, we recall the definition and properties of
the (imaginary) $q$-Racah probability measure on the set of lozenge tilings of a hexagon.
This measure was
introduced and studied in \cite{borodin-gr2009q}, and
we mostly follow the notation of that paper.

\subsection{Random lozenge tilings and nonintersecting paths}
\label{sub:tilings_and_paths}

Consider a hexagon drawn on the triangular lattice with
integer sides $a,b,c,a,b,c$. We are interested in tilings of this
hexagon by three types of lozenges. We will always work with
an affine transformed lattice, in which one of the lozenges
looks like a unit square. The hexagon with $a=b=c=3$ and an
example of its lozenge tiling are given in
\Cref{fig:tiling_and_path_picture}, left.

\begin{figure}[htpb]
	\centering
	\includegraphics[width=\textwidth]{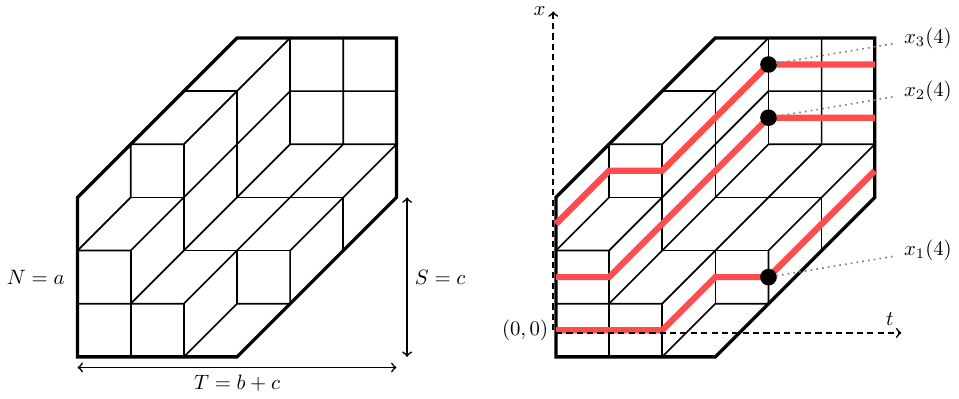}
	\caption{Left: An example of a lozenge tiling of a hexagon with sides $a=b=c=3$.
	Right: The corresponding ensemble of $N=a$ nonintersecting paths.}
	\label{fig:tiling_and_path_picture}
\end{figure}

We set $T\coloneqq b+c$, $S\coloneqq c$, and $N\coloneqq a$, and use the parameters $(T,S,N)$
throughout the paper.
The lozenge
tilings of our hexagon are in bijective correspondence with
ensembles of $N$ nonintersecting lattice paths that satisfy
the following conditions:
\begin{enumerate}[$\bullet$]
	\item We view the horizontal axis as the ``time'' $t\in \{0,1,\ldots,T\}$.
	The paths start at locations $(0,0),(0,1),\ldots,(0,N-1)$ at time $t=0$.
	\item The paths end at locations $(T,S),(T,S+1),\ldots,(T,S+N-1)$ at time $t=T$.
	\item The paths make two types of steps:
	\begin{enumerate}[$\circ$]
		\item horizontal, $(t,x)\to (t+1,x)$,
		\item and diagonal, $(t,x)\to (t+1,x+1)$.
	\end{enumerate}
	\item The paths do not intersect each other.
\end{enumerate}
See \Cref{fig:tiling_and_path_picture}, right, for an example the path ensemble corresponding to a given lozenge tiling.

We denote by
\begin{equation}
	\label{eq:nonintersecting_path_ensemble_X_t_defn}
	X(t)\coloneqq ( x_1(t)<x_2(t)<\ldots <x_N(t) )
\end{equation}
the locations of the paths at time $t$. For example, in
\Cref{fig:tiling_and_path_picture}, right, we have
$X(4)=(1,4,5)$. These locations must satisfy
\begin{equation}
	\label{eq:hexagon_constraints}
	\max(0, t+S-T)\le x_i(t)\le \min(t+N-1, S+N-1)
	,\qquad i=1,\ldots,N,\quad t=0,\ldots,T.
\end{equation}

\subsection{The $q$-Racah random lozenge tilings}

\begin{definition}[$q$-Racah measure on nonintersecting path ensembles]
	\label{def:qRacah_tilings}
	We put a probability measure on the set of all
	nonintersecting path ensembles. The measure depends on the
	geometry of the hexagon (parametrized by $(T,S,N)$), and
	on two parameters $q>0$ and $\kappa\in
	\mathbf{i}\mathbb{R}$.
	The
	probability weight of a nonintersecting path ensemble is
	proportional to
	\begin{equation}
		\label{eq:q_kappa_weight_of_ensemble}
		\prod_{\textnormal{holes in the path ensemble}} \mathsf{w}_{q,\kappa}\bigl( x-\tfrac t2+1\bigr),
		\qquad
		\mathsf{w}_{q,\kappa}(j) \coloneqq \kappa q^{j-(S+1)/2}-\frac{1}{\kappa q^{j-(S+1)/2}}.
	\end{equation}
	Here the holes are points $(t,x)$ such that $x\notin X(t)$,
	but $x$ satisfies the inequalities
	\eqref{eq:hexagon_constraints}. In other words,
	holes are lozenges of the type
	\begin{tikzpicture}[baseline = (current bounding box.south),scale=.25]
		\draw[thick] (0,0)--++(1,1)--++(1,0)--++(-1,-1)--cycle;
	\end{tikzpicture}.
	For example, in \Cref{fig:tiling_and_path_picture}, right, the holes are $(1,2),(2,1),(2,4),(3,0),(3,2),(3,5),(4,2),(4,3),(5,3)$.
	The probability weights are normalized so that the total probability of all path ensembles is $1$.
	We call the resulting probability measure the \emph{$q$-Racah measure} on nonintersecting path ensembles
	(equivalently, on lozenge tilings of the hexagon).
\end{definition}

\begin{remark}
	\label{rmk:q_assumptions_same_generality}
	The number of holes in the path ensemble is a conserved quantity: it is equal to $S(T-S)$, which depends on the geometry of the hexagon but not on the path ensemble. This implies that
	the $q$-Racah measure is invariant under the transformation
	$(q,\kappa)\mapsto (1/q,1/\kappa)$. Throughout the
	paper, we will assume that $q\in (0,1)$, as this
	does not restrict the generality.
\end{remark}

\begin{proposition}
	\label{prop:q_kappa_weight_of_ensemble_x_as_a_product_over_xi}
	In terms of the coordinates $x_i(t)$, the probability weight of a path ensemble is equal to
	\begin{equation}
		\label{eq:q_kappa_weight_of_ensemble_x}
		\prod_{t=0}^T
		\mathrm{const}(t,T,S,N)
		\prod_{i=1}^N
		\frac{q^{x_i(t)}}{1-\kappa^2 q^{2x_i(t)-S-t+1}}.
	\end{equation}
\end{proposition}
\begin{proof}
	This is essentially \cite[Proposition~3.3]{borodin-gr2009q}. For the reader's convenience, we provide a brief proof here.
	To get \eqref{eq:q_kappa_weight_of_ensemble_x} from
	\eqref{eq:q_kappa_weight_of_ensemble}, replace the product
	over the holes by the reciprocal of the product over the
	points in the path ensemble (this replacement contributes a constant
	factor independent of the path ensemble). Then, for
	$j=x-\tfrac t2+1$, we have
	\begin{equation*}
		\left(\kappa q^{j-(S+1)/2}-\frac{1}{\kappa q^{j-(S+1)/2}}\right)^{-1}=
		\frac{\kappa q^{\frac{1}{2} (S+t+2x+1)}}{\kappa ^2 q^{2 x+1}-q^{S+t}}
		=\mathrm{const}_0(t,T,S,N)\cdot \frac{q^{x}}{1-\kappa^2 q^{2x-S-t+1}},
	\end{equation*}
	which completes the proof.
\end{proof}

From \Cref{prop:q_kappa_weight_of_ensemble_x_as_a_product_over_xi}, one can readily notice the following degenerations of the
$q$-Racah measure on nonintersecting path ensembles:
\begin{enumerate}[$\bullet$]
	\item As $\kappa\to0$, the probability weight of a path
		ensemble becomes proportional to $q^{\mathsf{vol}}$, where
		$\mathsf{vol}$ is the volume under the three-dimensional
		surface corresponding to the lozenge tiling.
	\item As $\kappa\to 0$ and further $q\to 1$, the probability
		measure becomes uniform.
\end{enumerate}
Throughout the paper, we assume that $\kappa\ne0$. Otherwise, the measure $q^{\mathsf{vol}}$
(studied in much more detail in the literature)
does not display a dimensional collapse phenomenon.
For definiteness in some formulas involving $\kappa$ to the first power,
we also assume that $\kappa\in \mathbf{i}\mathbb{R}_{>0}$.

\subsection{$q$-Racah orthogonal polynomial ensemble}
\label{sub:q_racah_orthogonal_polynomials}

Under the $q$-Racah measure on nonintersecting path
ensembles, the tuples $\{X(t)\}_{0\le t\le T}$ form a Markov
chain, where $t$ plays the role of time. For each fixed $t$,
the distribution of $X(t)$ is connected to the \emph{$q$-Racah
orthogonal polynomial ensemble}. We refer to \cite{Konig2005}
for a general survey of orthogonal polynomial ensembles.

In this subsection, we recall the $q$-Racah orthogonal polynomials and define
the corresponding orthogonal polynomial ensemble.
In the following \Cref{sub:qRacah_to_tilings}, we will
connect the orthogonal polynomial ensemble to the distribution of $X(t)$ coming from the nonintersecting
paths.
For the notation
around the $q$-Racah polynomials, we follow
\cite[Chapter~3.2]{Koekoek1996}. The connection of random
nonintersecting paths of \Cref{def:qRacah_tilings} to
$q$-Racah polynomials follows
\cite[Theorem~4.1]{borodin-gr2009q} (see also
\cite[Theorem~7.3.5]{dimitrov2019log} for a compact account
of these results).

Throughout the paper, $(a;q)_k\coloneqq (1-a)(1-aq)\ldots (1-aq^{k-1})$ stands for the $q$-Pochhammer
symbol.\footnote{Since $|q|<1$, the infinite $q$-Pochhammer symbol $(a;q)_\infty=\prod_{k=0}^{\infty} (1-aq^k)$
is also well-defined.}
The $q$-hypergeometric function ${}_4\phi_3$ is given by
\cite[(1.2.22)]{GasperRahman}:
\begin{equation}
	\label{eq:q_hypergeometric_function_43}
	\begin{split}
		& {}_4\phi_3\left(\begin{array}{c}
			a, b, c, d \\
			e, f, g
		\end{array} \Big\vert\, q; z\right) \coloneqq \sum_{n=0}^{\infty} \frac{(a ; q)_n(b ; q)_n(c ; q)_n(d ; q)_n}{(e ; q)_n(f ; q)_n(g ; q)_n(q ; q)_n} z^n
	\end{split}
\end{equation}

Let $M\in \mathbb{Z}_{\ge0}$ and $\alpha,\beta,\gamma,\delta\in \mathbb{R}$ be such that $\gamma q=q^{-M}$. Define the following \emph{$q$-Racah weight function} on $\left\{ 0,1,\ldots,M  \right\}$:
\begin{equation}
	\label{eq:wqR}
	w^{qR}(x)\coloneqq
\frac{(\alpha q;q)_x (\beta \delta q;q)_x (\gamma q;q)_x (\gamma \delta q ; q)_x}
{(q;q)_x (\alpha^{-1} \gamma \delta q;q)_x (\beta^{-1} \gamma q;q)_x (\delta q ; q )_x} \frac{\left(1-\gamma \delta q^{2 x+1}\right)}{(\alpha \beta q)^x(1-\gamma \delta q)},
	\qquad x\in \left\{ 0,1,\ldots,M  \right\}.
\end{equation}
The weight $w^{qR}$ is the orthogonality weight for the $q$-Racah polynomials which are defined as
\begin{equation}
	\label{eq:qR_polynomials}
	R_n(\mu(x);\alpha,\beta,\gamma,\delta\mid q)\coloneqq
	{}_4\phi_3\left(\begin{array}{c}
			q^{-n}, \alpha \beta q^{n+1}, q^{-x}, \gamma \delta q^{x+1} \\
			\alpha q,\beta \delta q, \gamma q
	\end{array} \Big\vert\, q; q\right),\qquad n=0,1,\ldots,M .
\end{equation}
These are polynomials of degree $n$ in the variable $\mu(x)$ defined as
\begin{equation}
	\label{eq:mu_qRacah_defn}
	\mu(x)\coloneqq q^{-x}+\gamma \delta q^{x+1}.
\end{equation}
The orthogonality means that
\begin{equation}
	\label{eq:qRacah_orthogonality}
	\sum_{x=0}^{M}
	w^{qR}(x)\ssp R_m(\mu(x))R_n(\mu(x))=\mathbf{1}_{m=n}\ssp h_n,
\end{equation}
where
the squared norms $h_n$
have the form
\begin{equation}
	\label{eq:qRacah_norms}
	h_n=
	\frac{(\alpha \beta q^2;q)_M (\delta^{-1} ; q)_M}{(\alpha \delta^{-1} q;q)_M (\beta q ; q)_M}
	\frac{(1-\alpha \beta q)(\delta q^{-M})^n}{(1-\alpha \beta q^{2 n+1})}
	\frac{(q;q)_n (\alpha \beta q^{M+2};q)_n (\alpha \delta^{-1} q;q)_n (\beta q ; q)_n}
	{(\alpha q;q)_n (\alpha \beta q;q)_n ( \beta \delta q;q)_n  (q^{-M} ; q)_n}.
\end{equation}

The $q$-Racah polynomials $R_n$ are eigenfunctions of a
distinguished difference operator acting in~$x$
\cite[(3.2.6)]{Koekoek1996}, that is, they satisfy
\begin{equation}
	\label{eq:qRacah_eigenfunction}
	\begin{split}
		&q^{-n}(1-q^n)(1-\alpha \beta q^{n+1}) R_n(\mu(x))
		\\&\hspace{50pt}
		=
		B(x) R_n(\mu(x+1))-[B(x)+D(x)] R_n(\mu(x))(x)+D(x) R_n(\mu(x-1)),
	\end{split}
\end{equation}
where
\begin{equation}
	\label{eq:qRacah_BD_coefficients}
	\begin{split}
		B(x)&\coloneqq
		\frac{(1-\alpha q^{x+1})(1-\beta \delta q^{x+1})(1-\gamma q^{x+1})(1-\gamma \delta q^{x+1})}{(1-\gamma \delta q^{2 x+1})(1-\gamma \delta q^{2 x+2})}
		;
		\\
		D(x)&\coloneqq
		\frac{q(1-q^x)(1-\delta q^x)(\beta-\gamma q^x)(\alpha-\gamma \delta q^x)}{(1-\gamma \delta q^{2 x})(1-\gamma \delta q^{2 x+1})}.
	\end{split}
\end{equation}
In \eqref{eq:qRacah_eigenfunction}, we have $x=0,1,\ldots,M$. Indeed,
$B(M)=D(0)=0$, so the values of $R_n(\mu(x))$ for $x=-1$ or $x=M+1$ do not enter the eigenrelation \eqref{eq:qRacah_eigenfunction}.

\begin{remark}
	\label{rmk:three_regimes_in_qRacah}
The $q$-Racah polynomials admit three parameter regimes
$(\alpha,\beta,\gamma,\delta)$ that ensure the weight
function vanishes for $x>M$. These regimes are specified by
$\alpha q=q^{-M}$, $\beta\delta q=q^{-M}$, or $\gamma
q=q^{-M}$. Accordingly, the norms $h_n$ take different forms
in each regime. For random tilings, it suffices to consider
only the third regime $\gamma q=q^{-M}$, in which case $h_n$
is given by \eqref{eq:qRacah_norms}.
\end{remark}

Let us now describe the $N$-particle \emph{$q$-Racah orthogonal polynomial ensemble} (abbreviated  \emph{$q$-Racah OPE}) on $\{ 0,1,\ldots,M \}$ with parameters
$\alpha,\beta,\gamma,\delta,M$ satisfying
\begin{equation}
	\label{eq:qRacah_parameters_nonnegative}
	M \ge N-1,\qquad q\in(0,1),\qquad \gamma q=q^{-M},
	\qquad \alpha \ge \gamma,\qquad \beta \ge \gamma,\qquad \delta\le 0.
\end{equation}
With these restrictions on the parameters, one can check that the weights $w^{qR}(x)$
\eqref{eq:wqR}
are positive for all $x\in \{ 0,1,\ldots,M \}$.

\begin{definition}
	\label{def:qRacah_ensemble}
	The $q$-Racah ensemble
	$\mathfrak{R}^{qR(N)}_{M,\alpha,\beta,\gamma,\delta}$
	depending on the parameters satisfying \eqref{eq:qRacah_parameters_nonnegative},
	is a probability measure on $N$-tuples of particles $(x_1<\ldots<x_N )$ in $\{0,1,\ldots,M \}$,
	with probability weights given by
	\begin{equation}
		\label{eq:qRacah_ensemble_defn}
		\mathfrak{R}^{qR(N)}_{M,\alpha,\beta,\gamma,\delta}(x_1,\ldots,x_N )=
		\frac{1}{Z(N,M,\alpha,\beta,\gamma,\delta)}
		\prod_{1\le i<j\le N}\left( \mu(x_i)-\mu(x_j) \right)^2
		\prod_{i=1}^{N}w^{qR}(x_i).
	\end{equation}
	Here $Z(N,M,\alpha,\beta,\gamma,\delta)$ is the normalization constant ensuring that the total probability is~$1$. The condition $M \ge N-1$ in \eqref{eq:qRacah_parameters_nonnegative} is required so that $N$ particles can fit into $\{0,1,\ldots, M \}$.
\end{definition}

\subsection{From orthogonal polynomials to lozenge tilings}
\label{sub:qRacah_to_tilings}

Let us now connect the random nonintersecting path ensembles of \Cref{def:qRacah_tilings} to the $q$-Racah OPE. For this, we need to define four zones within the hexagon with sides $a,b,c,a,b,c$ (recall that $T=b+c$, $S=c$, $N=a$), where the $q$-Racah parameters take different forms. These zones are determined by the value of the horizontal coordinate (time) $t$, and are defined as follows:
\begin{align}
	\label{eq:hexagon_zone_1}
	0\le t\le \min(S-1,T-S-1)&\quad \Rightarrow \quad
	\parbox{.5\textwidth}{$M=t+N-1$, $\alpha=q^{-S-N}$,\\$\beta=q^{S-T-N}$, $\gamma=q^{-t-N}$, $\delta=\kappa^2 q^{-S+N}$;}
	\\[7pt]
	\label{eq:hexagon_zone_2}
	S\le t\le T-S-1&\quad \Rightarrow \quad
	\parbox{.5\textwidth}{$M=S+N-1$, $\alpha=q^{-t-N}$,\\$\beta=q^{t-T-N}$, $\gamma=q^{-S-N}$, $\delta=\kappa^2 q^{-t+N}$;}
	\\[7pt]
	\label{eq:hexagon_zone_3}
	T-S\le t\le S-1&\quad \Rightarrow \quad
	\parbox{.5\textwidth}{$M=T-S+N-1$, $\alpha=q^{-T-N+t}$,\\$\beta=q^{-t-N}$, $\gamma=q^{-T-N+S}$, $\delta=\kappa^2 q^{-T+t+N}$;}
	\\[7pt]
	\label{eq:hexagon_zone_4}
	\max(S,T-S)\le t\le T&\quad \Rightarrow \quad
	\parbox{.5\textwidth}{$M=T-t+N-1$, $\alpha=q^{-T-N+S}$,\\$\beta=q^{-S-N}$, $\gamma=q^{-T-N+t}$, $\delta=\kappa^2 q^{-T+S+N}$;}
\end{align}
For each given hexagon, at most three of these zones are present, as conditions \eqref{eq:hexagon_zone_2} and \eqref{eq:hexagon_zone_3} are mutually exclusive. See \Cref{fig:hexagon_zones} for illustrations.

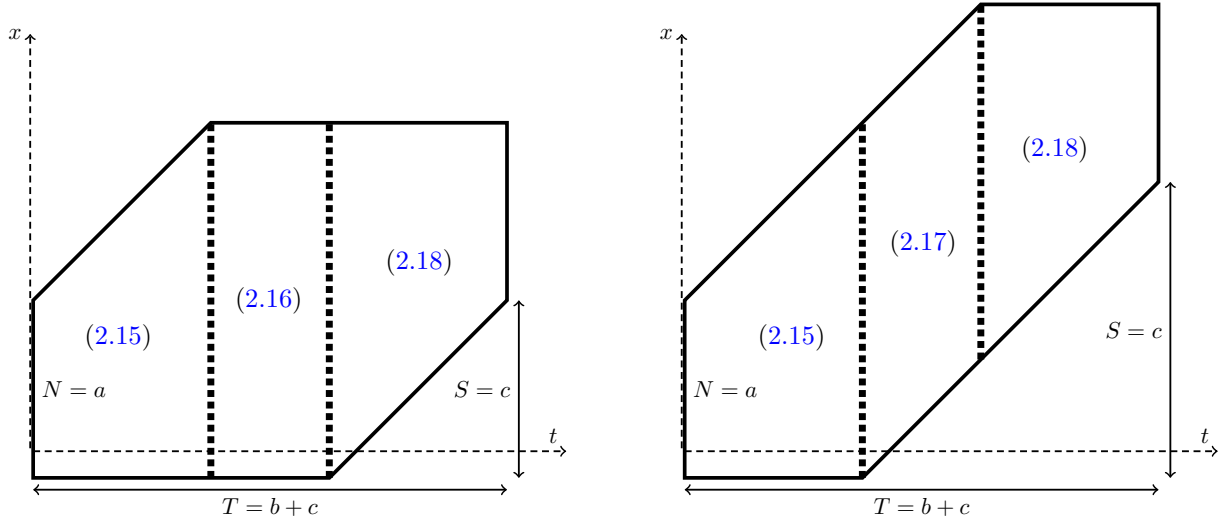
\begin{figure}[htpb]
	\centering
	\begin{tikzpicture}
		\node[anchor=west] at (0,0) {\includegraphics[width=\textwidth]{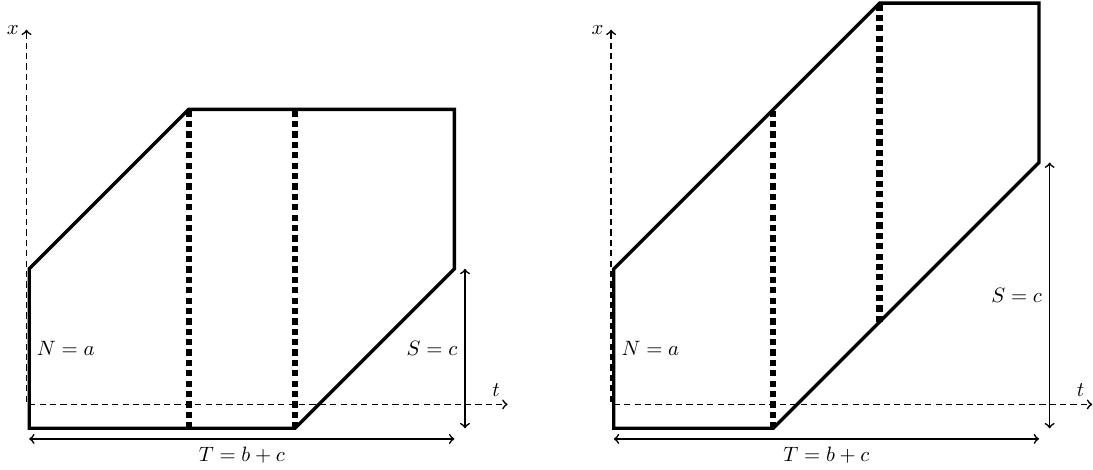}};
		\node at (1.7,-1+0) {\small\eqref{eq:hexagon_zone_1}};
		\node at (1.7+2,-1+.5) {\small\eqref{eq:hexagon_zone_2}};
		\node at (1.7+4,-1+1) {\small\eqref{eq:hexagon_zone_4}};
		\node at (1.65+9,-1) {\small\eqref{eq:hexagon_zone_1}};
		\node at (1.65+1.75+9,-1+1.25) {\small\eqref{eq:hexagon_zone_3}};
		\node at (1.65+3.5+9,-1+2.5) {\small\eqref{eq:hexagon_zone_4}};
	\end{tikzpicture}
	\caption{Zones inside the hexagon in which
	$X(t)$ (the path configuration at a slice
	$t=\textnormal{const}$) has the $q$-Racah OPE distribution
	with different choices of parameters (see
	\Cref{thm:qRacah_ensemble_and_tilings}). For all $T,S,N$,
	the zones \eqref{eq:hexagon_zone_1} and
	\eqref{eq:hexagon_zone_4} are present inside the hexagon.
	The zone \eqref{eq:hexagon_zone_2} is present if and only
	if $S\le T-S$, and otherwise we have the zone
	\eqref{eq:hexagon_zone_3}. On a border slice between two
	zones, the $q$-Racah parameters can be chosen in either of
	the two ways (for example, formulas in
	\eqref{eq:hexagon_zone_1} and \eqref{eq:hexagon_zone_2}
	coincide for $t=S$).}
	\label{fig:hexagon_zones}
\end{figure}

\begin{theorem}[{\cite[Theorem~4.1]{borodin-gr2009q} and \cite[Theorem~7.3.5]{dimitrov2019log}}]
	\label{thm:qRacah_ensemble_and_tilings}
	Fix $T,S,N$. In the zones \eqref{eq:hexagon_zone_1} and \eqref{eq:hexagon_zone_2}, the distribution of $X(t)$ \eqref{eq:nonintersecting_path_ensemble_X_t_defn} under the $q$-Racah measure (\Cref{def:qRacah_tilings}) is given by the $N$-particle $q$-Racah OPE on $\{0,1,\ldots,M\}$. In the zones \eqref{eq:hexagon_zone_3} and \eqref{eq:hexagon_zone_4}, the distribution of the shifted configuration
	\begin{equation}
		\label{eq:shifted_X_t}
		X(t)+T-t-S=(x_1(t)+T-t-S<x_2(t)+T-t-S<\ldots<x_N(t)+T-t-S)
	\end{equation}
	is given by the $N$-particle $q$-Racah OPE on $\{0,1,\ldots,M\}$. In all cases, the parameters $M, \alpha,\beta,\gamma,\delta$ are given by the corresponding formulas in \eqref{eq:hexagon_zone_1}--\eqref{eq:hexagon_zone_4}.
\end{theorem}

\subsection{Determinantal correlation kernel and its operator interpretation}
\label{sub:dpp_and_operator_interpretation}

It is known \cite{Konig2005}, \cite[Section~4]{Borodin2009}
that orthogonal polynomial ensembles (in particular,
the~$N$-particle $q$-Racah OPE on $\{0,1,\ldots,M\}$) are
determinantal point processes.
Moreover, the two-dimensional
random point configuration
$\mathscr{X}\coloneqq\{X(t)\}_{0\le t\le T}$ forms a
determinantal point process on $\mathbb{Z}^2$.
The two-dimensional statement follows from the particular
way of how the distributions of $X(t)$ on different slices
are stitched together, and may be derived from
\cite[Section~4]{Borodin2009}.
Alternatively,
the determinantal structure of $\mathscr{X}$ follows from a
dimer interpretation of the $q$-Racah random tiling model,
cf.~\cite[Corollary~3]{Kenyon2007Lecture}.

Let the space-time correlation functions be defined as
\begin{equation*}
	\rho_n(t_1, x_1 ; \ldots ; t_n, x_n)
	\coloneqq
	\operatorname{\mathbb{P}} \left\{
	\textnormal{the configuration $\mathscr{X}$ contains all of $(t_1, x_1),\ldots,(t_n, x_n)$}\right\},
\end{equation*}
where $(t_1, x_1),\ldots,(t_n, x_n)$ are $n$ pairwise distinct points.
For any $n\ge 1$, these correlation functions are given by determinants
\begin{equation}
	\label{eq:dpp}
	\rho_n(t_1, x_1 ; \ldots ; t_n, x_n)=\det [K( t_i, x_i ; t_j, x_j)]_{i,j=1}^n,
\end{equation}
coming from a single function of two points inside the hexagon.
This function
$K(t,x;s,y)$ is called the \emph{correlation kernel}.

An explicit expression for $K(t,x;s,y)$ was obtained in \cite[Theorem~7.5]{borodin-gr2009q} in terms of $q$-Racah polynomials. We now recall this result, together with an operator interpretation of the kernel
\begin{equation}
	\label{eq:K_t_x_y_fixed_slice}
	K_t(x,y)\coloneqq K(t,x;t,y)
\end{equation}
on a fixed slice $t=\textnormal{const}$, which also appears in \cite{borodin-gr2009q}.

\medskip

Consider a $q$-Racah OPE $\mathfrak{R}^{qR(N)}_{M,\alpha,\beta,\gamma,\delta}$ on $\{0,1,\ldots,M\}$ with parameters satisfying \eqref{eq:qRacah_parameters_nonnegative}. Let us pass from the $q$-Racah polynomials \eqref{eq:qR_polynomials} to the functions
\begin{equation}
	\label{eq:qRacah_orthonormal_basis_f_n}
	f_n(x)=f_n(x; \alpha,\beta,\gamma,\delta\mid q)
	\coloneqq \frac{R_n(\mu(x))}{\sqrt{h_n}}\ssp \sqrt{w^{qR}(x)},
	\qquad
	n=0,1,\ldots, M,
\end{equation}
which form an orthonormal basis in the Hilbert space
$\ell^2\left( \left\{ 0,1,\ldots,M  \right\} \right)$.
Thanks to \eqref{eq:qRacah_eigenfunction}, the $f_n(x)$'s are eigenfunctions of the difference operator
\begin{equation}
	\label{eq:qRacah_orthonormal_difference_operator}
	\bigl(\mathfrak{D}^{qR}g\bigr)(x)\coloneqq
	\sqrt{\frac{w^{qR}(x)}{w^{qR}(x+1)}}B(x)\ssp g(x+1)-[B(x)+D(x)]\ssp g(x)
	+
	\sqrt{\frac{w^{qR}(x)}{w^{qR}(x-1)}}D(x)\ssp g(x-1),
\end{equation}
where $B(x)$ and $D(x)$ are given by \eqref{eq:qRacah_BD_coefficients}.
The eigenvalue of $f_n(x)$ is
\begin{equation}
	\label{eq:qRacah_eigenvalues}
	\mathrm{ev}^{qR}_n\coloneqq q^{-n}(1-q^n)(1-\alpha \beta q^{n+1}).
\end{equation}
In particular, $\mathrm{ev}^{qR}_0=0$.
Due to our assumptions on the parameters \eqref{eq:qRacah_parameters_nonnegative}, the eigenvalues $\mathrm{ev}^{qR}_n$ are all nonpositive, and strictly decrease in $n$ for $0\le n\le M$.

\begin{proposition}
	[{\cite{mehta2004random}, \cite[Lemma~2.8]{Konig2005}}]
	\label{prop:qRacah_kernel_OPE}
	The correlation kernel of the $N$-particle $q$-Racah OPE
	on $\{0,1,\ldots,M\}$
	is given by
	\begin{equation}
		\label{eq:qRacah_kernel}
		K^{qR}(x,y)=\sum_{n=0}^{N-1} f_n(x)f_n(y).
	\end{equation}
	This is a kernel of the orthogonal projection operator onto
	$\operatorname{\mathrm{span}}\left( f_0,f_1,\ldots,f_{N-1}  \right)$
	in the Hilbert space
	$\ell^2\left( \left\{ 0,1,\ldots,M  \right\} \right)$.
\end{proposition}

Following \cite[Section~8.2]{borodin-gr2009q} which implements the operator approach pioneered in
\cite{borodin2007asymptotics}, \cite{Olshansk2008-difference}
(see also \cite[Chapter~3.3]{TaoRMbook}), we view $K^{qR}$ as the kernel of the orthogonal spectral projection onto the subspace of $\ell^2\left( \left\{ 0,1,\ldots,M  \right\} \right)$ corresponding to the spectral interval
\begin{equation}
	\label{eq:qRacah_kernel_spectral_interval_for_N_particles}
	\left[ -q^{-N+1}(1-q^{N-1})(\alpha \beta q^{N}-1),0 \right]
\end{equation}
of the operator $\mathfrak{D}^{qR}$ \eqref{eq:qRacah_orthonormal_difference_operator}.

\medskip

Let us now apply the spectral interpretation of $K^{qR}$ to the correlation kernel $K_t(x,y)$ \eqref{eq:K_t_x_y_fixed_slice}
on a fixed vertical slice inside the hexagon.
Observe that in all four zones
\eqref{eq:hexagon_zone_1}--\eqref{eq:hexagon_zone_4}, we
have $\alpha\beta q^N=q^{-T-N}$. Therefore, the spectral
interval
\eqref{eq:qRacah_kernel_spectral_interval_for_N_particles}
stays the same throughout the hexagon.
Denote
\begin{equation}
	\label{eq:tilde_x}
	\tilde x\coloneqq\begin{cases}
		x,&\textnormal{if $t$ is in the zone \eqref{eq:hexagon_zone_1} or \eqref{eq:hexagon_zone_2}};\\
		x+T-t-S,&\textnormal{if $t$ is in the zone \eqref{eq:hexagon_zone_3} or \eqref{eq:hexagon_zone_4}}.
	\end{cases}
\end{equation}
One can check that the coefficients of the operator $\mathfrak{D}^{qR}$ \eqref{eq:qRacah_orthonormal_difference_operator} evaluated at $\tilde x$ are also the same in all four zones \eqref{eq:hexagon_zone_1}--\eqref{eq:hexagon_zone_4}. More precisely, we have
\begin{equation}
	\label{eq:qRacah_BD_coefficients_tilde_x_for_hexagon}
	\begin{split}
		B(\tilde x)&=
		\frac{q^{-2 N-T} (q^T-\kappa ^2 q^{x+1}) (q^{N+S}-q^{x+1}) (q^{N+t}-q^{x+1})
		(q^{S+t}-\kappa ^2 q^{x+1})}{(q^{S+t}-\kappa ^2 q^{2 x+1}) (q^{S+t}-\kappa ^2 q^{2 x+2})},
		\\
		D(\tilde x)&=
		\frac{(1-q^x) q^{-2 N-T+1} (q^S-\kappa ^2 q^{N+x}) (q^t-\kappa ^2 q^{N+x})
		(q^{S+t}-q^{T+x})}{(q^{S+t}-\kappa ^2 q^{2 x}) (q^{S+t}-\kappa ^2 q^{2 x+1})},
	\end{split}
\end{equation}
and
\begin{equation}
	\label{eq:qRacah_BD_coefficients_tilde_x_for_hexagon_2_with_weights}
	\begin{split}
		&\frac{w^{qR}(\tilde x)}{w^{qR}(\tilde x+1)}
		=
		\frac{q(1-q^{x+1}) (q^S-\kappa ^2 q^{N+x+1}) (q^t-\kappa
		^2 q^{N+x+1}) (q^{S+t}-\kappa ^2 q^{2 x+1})
			(q^{S+t}-q^{T+x+1})}{(q^T-\kappa ^2 q^{x+1})
			(q^{N+S}-q^{x+1}) (q^{N+t}-q^{x+1}) (q^{S+t}-\kappa ^2
			q^{x+1}) (q^{S+t}-\kappa ^2 q^{2 x+3})}
		;
		\\
		&\frac{w^{qR}(\tilde x)}{w^{qR}(\tilde x-1)}
		=
		\frac{(q^T-\kappa ^2 q^x) (q^{N+S}-q^x)
		(q^{N+t}-q^x) (q^{S+t}-\kappa ^2 q^x) (q^{S+t}-\kappa ^2
			q^{2 x+1})}{(q^x-1) (\kappa ^2 q^{N+x}-q^S) (\kappa ^2
			q^{N+x}-q^t) (q^{S+t+1}-\kappa ^2 q^{2 x})
			(q^{T+x}-q^{S+t})}.
	\end{split}
\end{equation}
%
We conclude that the kernel $K_t(x,y)$ has the following
spectral description, which is the same in all four zones
\eqref{eq:hexagon_zone_1}--\eqref{eq:hexagon_zone_4}:
\begin{corollary}
	\label{cor:qRacah_hexagon_kernel_spectral_description}
	The fixed-slice correlation kernel $K_t(x,y)$
	\eqref{eq:K_t_x_y_fixed_slice} of the $q$-Racah
	nonintersecting path ensemble inside the hexagon with
	sides $a,b,c,a,b,c$ (where $T=b+c$, $S=c$, $N=a$) is the
	kernel of the orthogonal spectral projection onto the
	subspace corresponding to the spectral interval
	\begin{equation}
		\label{eq:spectral_interval_original_N_in_corollary}
		\left[ -q^{-N+1}(1-q^{N-1})(q^{-T-N}-1),0 \right]
	\end{equation}
	of the
	difference operator
	\begin{equation}
		\label{eq:qRacah_orthonormal_difference_operator_for_hexagon_final_in_s2}
		\bigl(\mathfrak{D}g\bigr)(x)\coloneqq
		\sqrt{\frac{w^{qR}(\tilde x)}{w^{qR}(\tilde  x+1)}}B(\tilde  x)\ssp g(x+1)
		-[B(\tilde x)+D(\tilde x)]\ssp g(x)
		+
		\sqrt{\frac{w^{qR}(\tilde x)}{w^{qR}(\tilde x-1)}}D(\tilde x)\ssp g(x-1).
	\end{equation}
	This operator on functions of $x$
	acts in the (finite-dimensional)
	$\ell^2$ space on the $t$-th vertical slice of the hexagon.
	Its coefficients are read off from
	\eqref{eq:qRacah_BD_coefficients_tilde_x_for_hexagon} and
	\eqref{eq:qRacah_BD_coefficients_tilde_x_for_hexagon_2_with_weights}.
\end{corollary}

\begin{remark}
	\label{rmk:selfadjoint}
	One can check that
	\begin{equation*}
		\sqrt{\frac{w^{qR}(\tilde x-1)}{w^{qR}(\tilde  x)}}B(\tilde  x-1)
		=
		\sqrt{\frac{w^{qR}(\tilde x)}{w^{qR}(\tilde x-1)}}D(\tilde x),
	\end{equation*}
	which means that the operator $\mathfrak{D}$ is
	self-adjoint with respect to the standard inner product in
	$\ell^2$ on the $t$-th vertical slice.
\end{remark}

\section{Scaling regime}
\label{sec:scaling_regime}

\subsection{Traditional scaling of $q$-weighted lozenge tilings}
\label{sub:traditional_scaling_regime}

We consider the asymptotic regime when the sides of the hexagon
are scaled proportionally to some large parameter $L\to+\infty$.
Typically, the parameter $q$
(for the $q^{\mathsf{vol}}$ or the $q$-Racah measure on lozenge tilings)
also goes to $1$, at the rate
$\exp\left( -c/L \right)$, where $c\in \mathbb{R}$ is fixed.
In this regime, random tilings develop a
so-called \emph{liquid region},
in which local lattice distributions follow translation-invariant
ergodic Gibbs measures (\emph{pure states}, for
short) on lozenge tilings of the whole plane $\mathbb{Z}^2$.
Since in the limit the parameter $q$ becomes $1$, the limiting pure states
(for both the $q^{\mathsf{vol}}$ and the $q$-Racah ensembles)
satisfy the Gibbs property with respect to the
\textit{uniform resampling} of a tiling in any finite subregion,
conditioned on the boundary of this subregion.
Such Gibbs measures are uniquely determined by two parameters
\cite{Sheffield2008}, which in terms of nonintersecting paths
may be taken as their vertical density and horizontal slope.
These pure states universally arise as local lattice
limits in uniformly random lozenge tilings of arbitrary polygons
\cite{aggarwal2019universality}.

The literature on this asymptotic regime of the measure $q^{\mathsf{vol}}$ is vast, and includes
\cite{nienhuis1984triangular},
\cite{CohnLarsenPropp},
\cite{cerf2001low},
\cite{okounkov2003correlation},
\cite{ferrari2003step},
\cite{KOS2006},
\cite{OkounkovKenyon2007Limit},
\cite{BG2011non},
\cite{mkrtchyan2017gue},
\cite{DiFran2019qvol},
\cite{petrov2023asymptotics},
\cite{Ahn2022lozenge}.
Main results on the asymptotic behavior of the
uniform and
$q^{\mathsf{vol}}$ measures are summarized in
\cite{Kenyon2007Lecture} and \cite{gorin2021lectures}.
Asymptotic behavior of the $q$-Racah measure as $q\to 1$
was studied in
\cite{borodin-gr2009q},
\cite{dimitrov2019log},
\cite{gorin2022dynamical},
\cite{Duits2024lozenge}.
In these references, the limiting objects within the liquid
region are two-dimensional, such as the pure Gibbs states at the lattice
level, Gaussian Free Field describing global fluctuations,
or the GUE corners process at points where the liquid region
comes close to the boundary of the polygon.

The novelty of the present work is that we keep the parameter $q$ fixed
while scaling the sides of the hexagon to infinity.
As was first observed in simulations in \cite{borodin-gr2009q}, this leads to
formation of a so-called \emph{waterfall region} inside the hexagon.
We describe this region in \Cref{sub:scaling_and_waterfall_region} below.

\subsection{Waterfall scaling}
\label{sub:scaling_and_waterfall_region}

We consider the limit behavior of the
$q$-Racah nonintersecting path ensemble as the sides of the
hexagon grow proportionally to infinity as
$L\to+\infty$:
\begin{equation}
		\label{eq:limit_regime}
		T=\lfloor L\mathsf{T} \rfloor,
		\
		S=\lfloor L\mathsf{S} \rfloor ,
		\
		N=\lfloor L\mathsf{N} \rfloor ,
	\
	\mathsf{T}>\mathsf{S}>0,\ \mathsf{N}>0;
	\quad
	q\in(0,1),
	\,
	\kappa\in \mathbf{i}\mathbb{R}_{>0} \ \textnormal{fixed}.
\end{equation}
In this regime, the
$q$-Racah nonintersecting paths display a very different
limit behavior from the traditional one (\Cref{sub:traditional_scaling_regime}).
Namely, for fixed $q$ and $\kappa$, 
the two-dimensional lattice behavior
\emph{collapses} into a one-dimensional random stepped interface (which we call a \emph{barcode}).
One can also think that this random interface satisfies (a degenerate version of)
the
$q$-deformed Gibbs property,
which does not allow any changes under resampling.

Let us define the macroscopic
\emph{waterfall region} where this new behavior
occurs.

\begin{definition}[Waterfall region]
	\label{def:waterfall_region}
	Fix the scaled dimensions
	$\mathsf{T},\mathsf{S},\mathsf{N}$
	of the hexagon, and let $(\mathsf{t},\mathsf{x})$ be the scaled
	coordinates. Denote by $\mathcal{H}$ the scaled hexagon, that is,
	\begin{equation}
		\label{eq:scaled_hexagon}
		\mathcal{H}\coloneqq \left\{
		(\mathsf{t},\mathsf{x})\colon 0< \mathsf{t}<\mathsf{T},\
		\max(0,\mathsf{t}-\mathsf{T}+\mathsf{S})< \mathsf{x}<
		\min(\mathsf{S}+\mathsf{N},\mathsf{t}+\mathsf{N})\right\}.
	\end{equation}
	Let
	\begin{equation}
		\label{eq:tl_tr}
		\mathsf{t}_l\coloneqq
		|\mathsf{N}-\mathsf{S}|
		,
		\qquad
		\mathsf{t}_r\coloneqq
		\min(\mathsf{N}+\mathsf{S},2\mathsf{T}-\mathsf{N}-\mathsf{S}).
	\end{equation}
	Denote the \emph{waterfall region} by
	\begin{equation}
		\label{eq:waterfall_region}
		\mathcal{W}\coloneqq \left\{ (\mathsf{t},\mathsf{x})\in \mathcal{H}\colon
			\mathsf{t}_l< \mathsf{t}< \mathsf{t}_r,
			\
			|2\mathsf{x}-\mathsf{S}-\mathsf{t}|< \mathsf{N}
		\right\}.
	\end{equation}
	Also, let the \emph{center line} be the following line segment of slope $\frac{1}{2}$:
	\begin{equation}
		\label{eq:center_line}
		\mathcal{C}\coloneqq
		\left\{ (\mathsf{t},\mathsf{x})\in \mathcal{H}\colon
		\mathsf{t}_l< \mathsf{t}< \mathsf{t}_r,\
		\mathsf{x}=\tfrac12(\mathsf{t}+\mathsf{S})
		\right\}
		\subset \mathcal{W}.
	\end{equation}
	See \Cref{fig:waterfall_region} for illustrations of the waterfall
	region and the corresponding exact samples
	from the $q$-Racah measure generated by the
	shuffling algorithm of \cite{borodin-gr2009q}.
\end{definition}

\begin{remark}
	\label{rmk:thin_waterfall}
	Note that when $\mathsf{N}\ge \mathsf{T}$, we have
	$\mathsf{t}_l\ge \mathsf{t}_r$. This means that the
	waterfall region is nonempty only for $\mathsf{N}<
	\mathsf{T}$.
	In
	\Cref{sec:asymptotic_clustering_of_nonintersecting_paths}
	below,
	we show that tor $\mathsf{N}>\mathsf{T}$,
	the nonintersecting paths can only have asymptotic slopes
	$0$ or $1$ (depending on the part of the hexagon), so
	the
	waterfall behavior does not occur.
	It would be interesting to probe 
	the case of ``thin waterfall'' when
	$T/N\to1$
	while $T-N\gg 1$ (see \Cref{fig:waterfall_region_frozen}
	for exact samples), but we do not address this
	question in the present work.
\end{remark}

\begin{figure}[htbp]
	\centering
	\includegraphics[height=210pt]{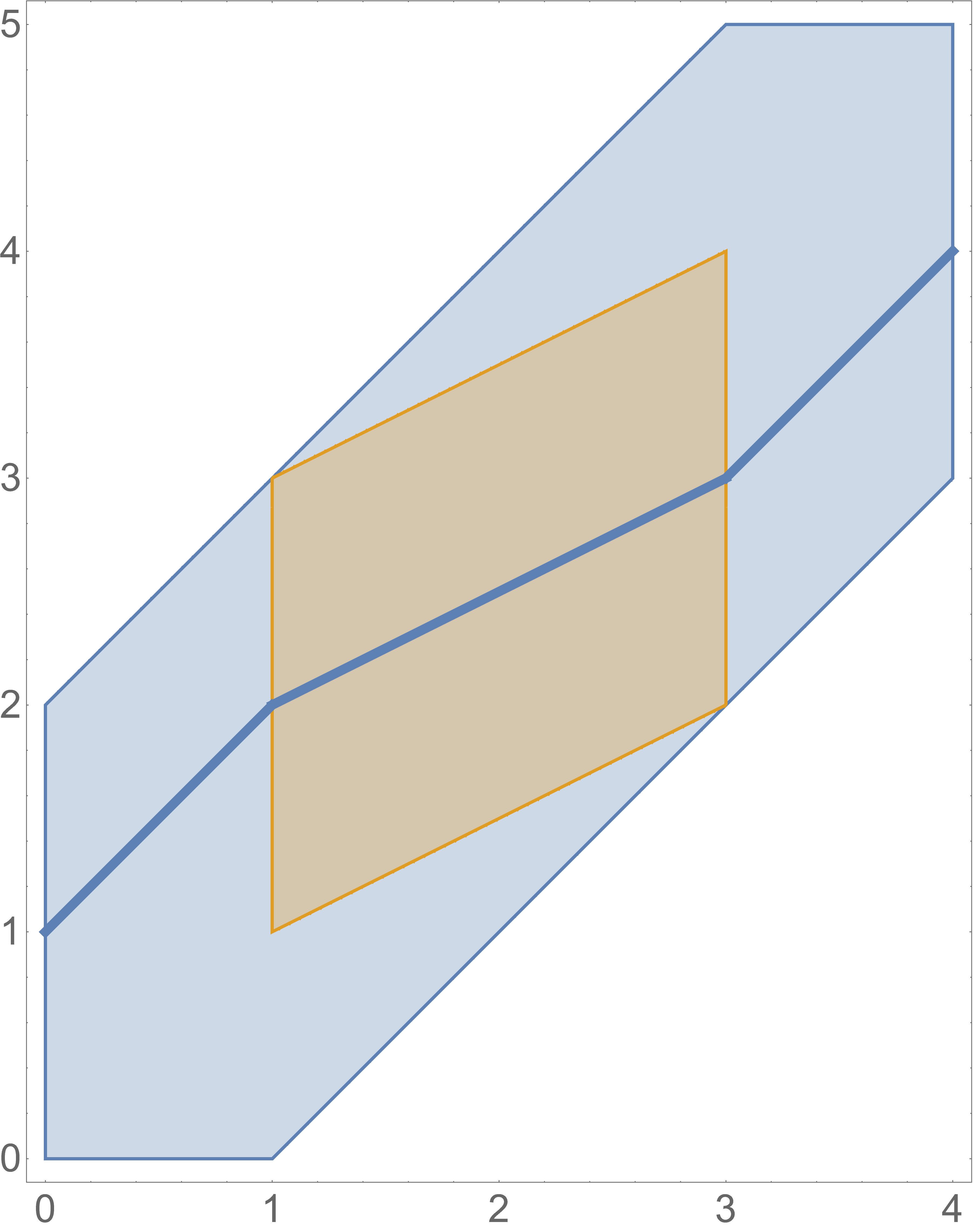}\qquad\qquad
	\includegraphics[height=210pt]{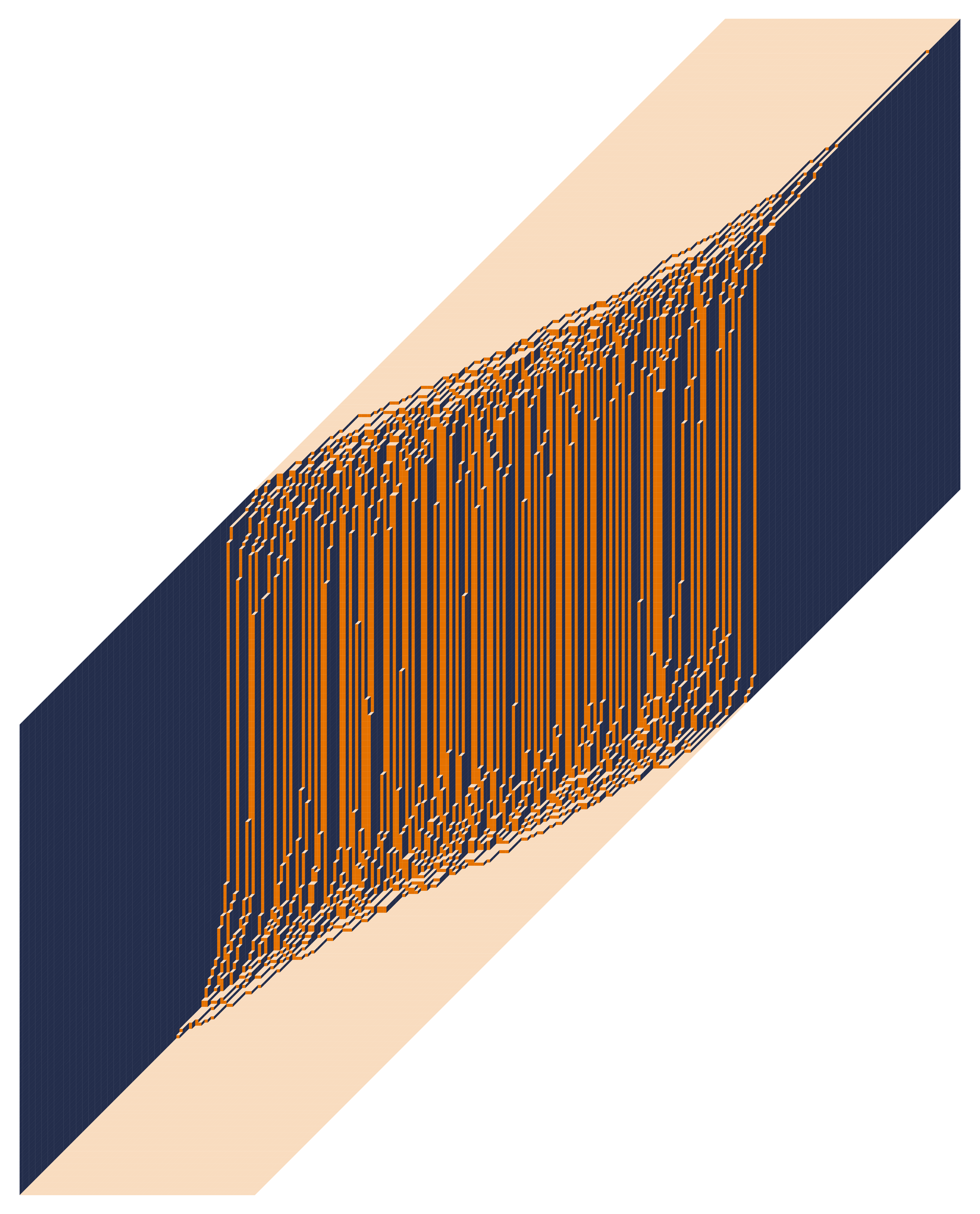}\qquad
	\\\vspace*{10pt}
	\begin{tabular}{cc}
		\includegraphics[height=126pt]{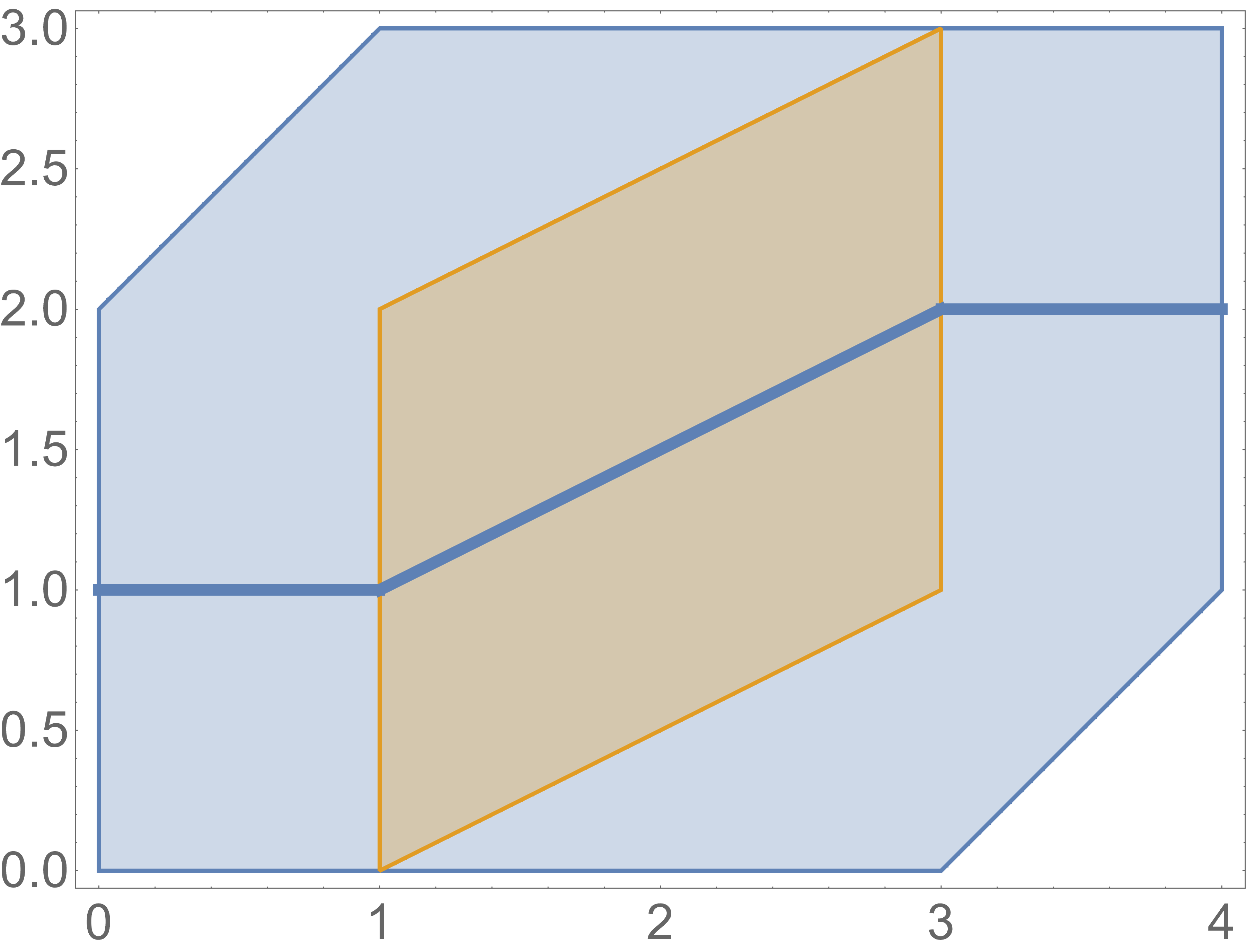}\qquad \qquad \includegraphics[height=126pt]{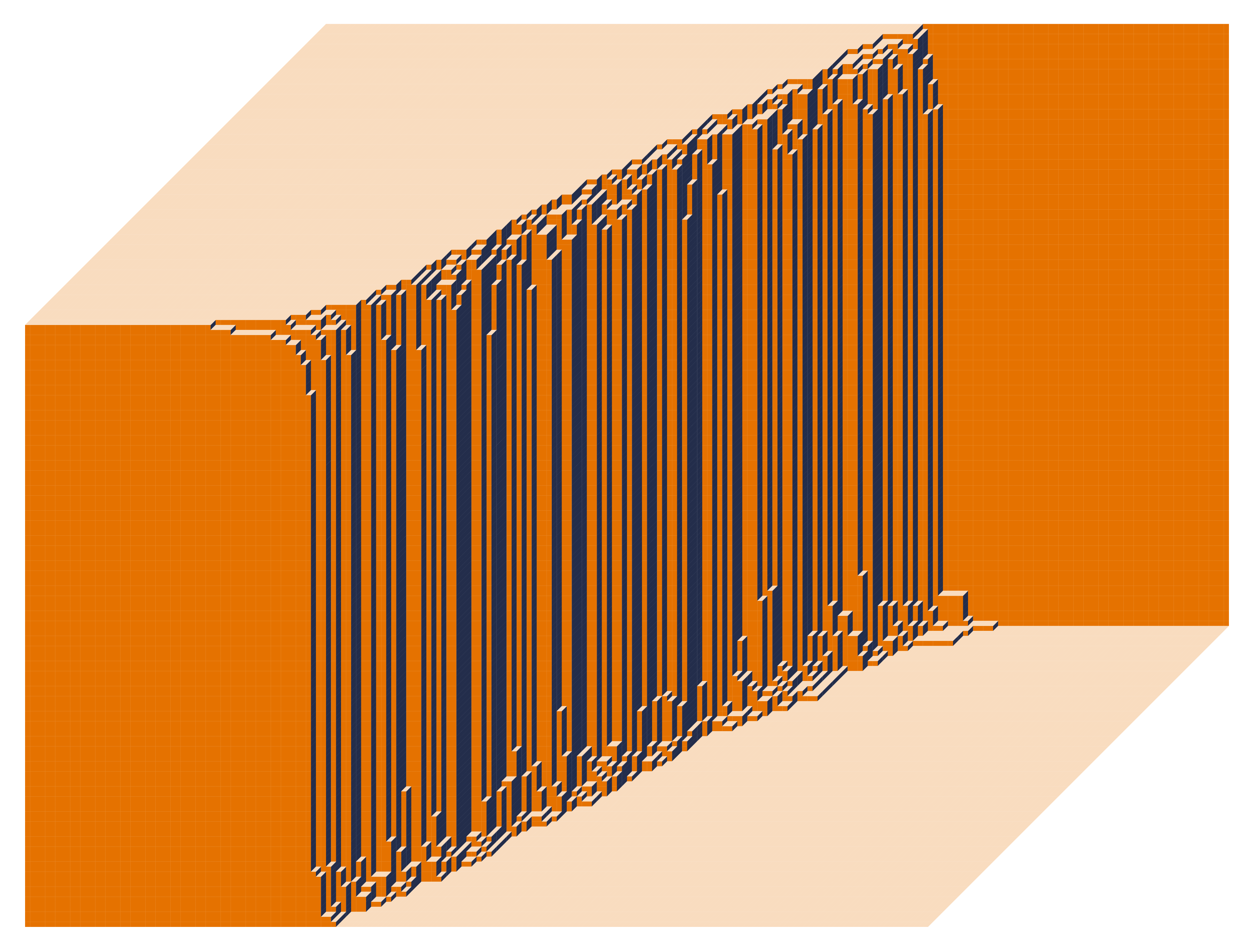}\\[10pt]
		\includegraphics[height=126pt]{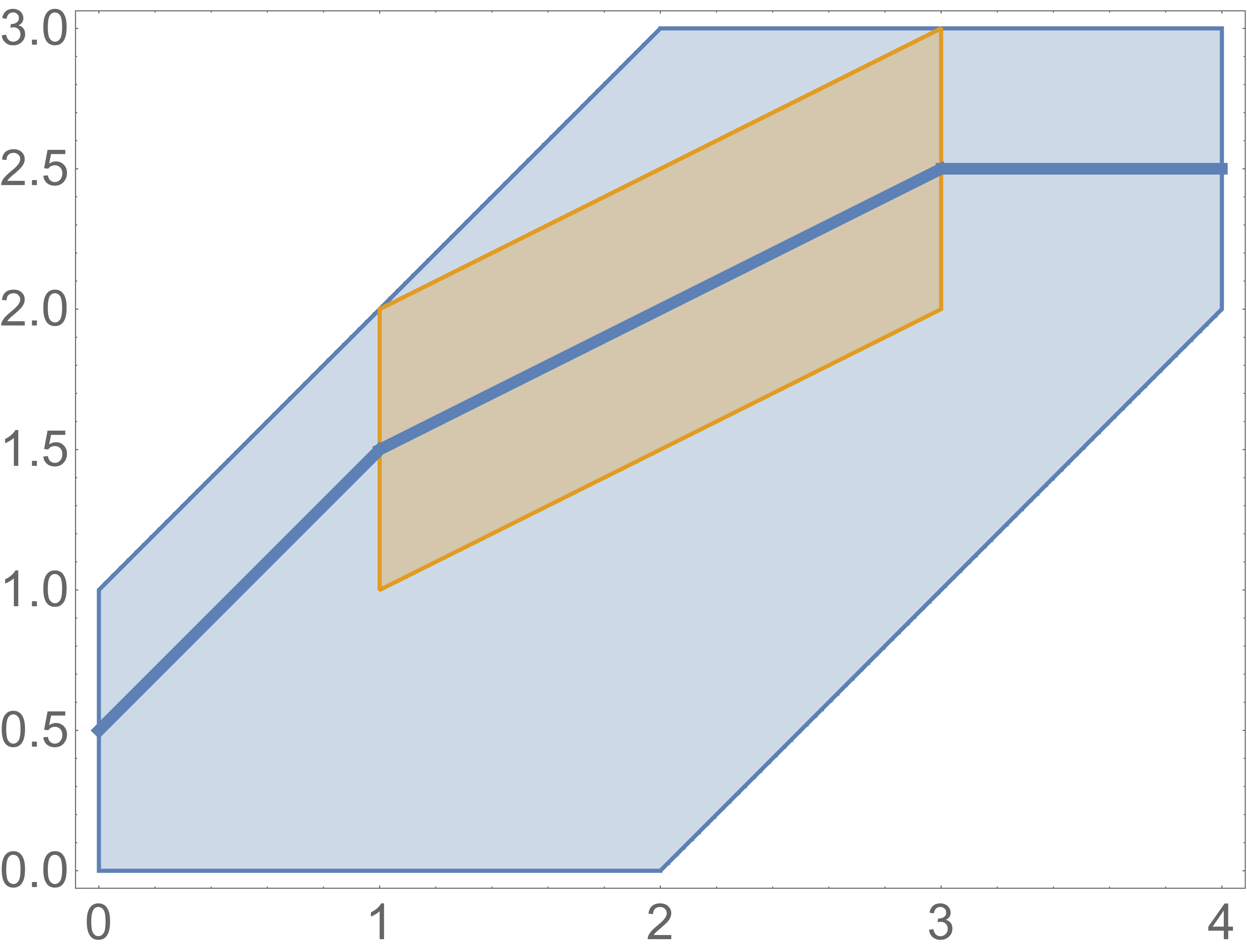}\qquad \qquad \includegraphics[height=126pt]{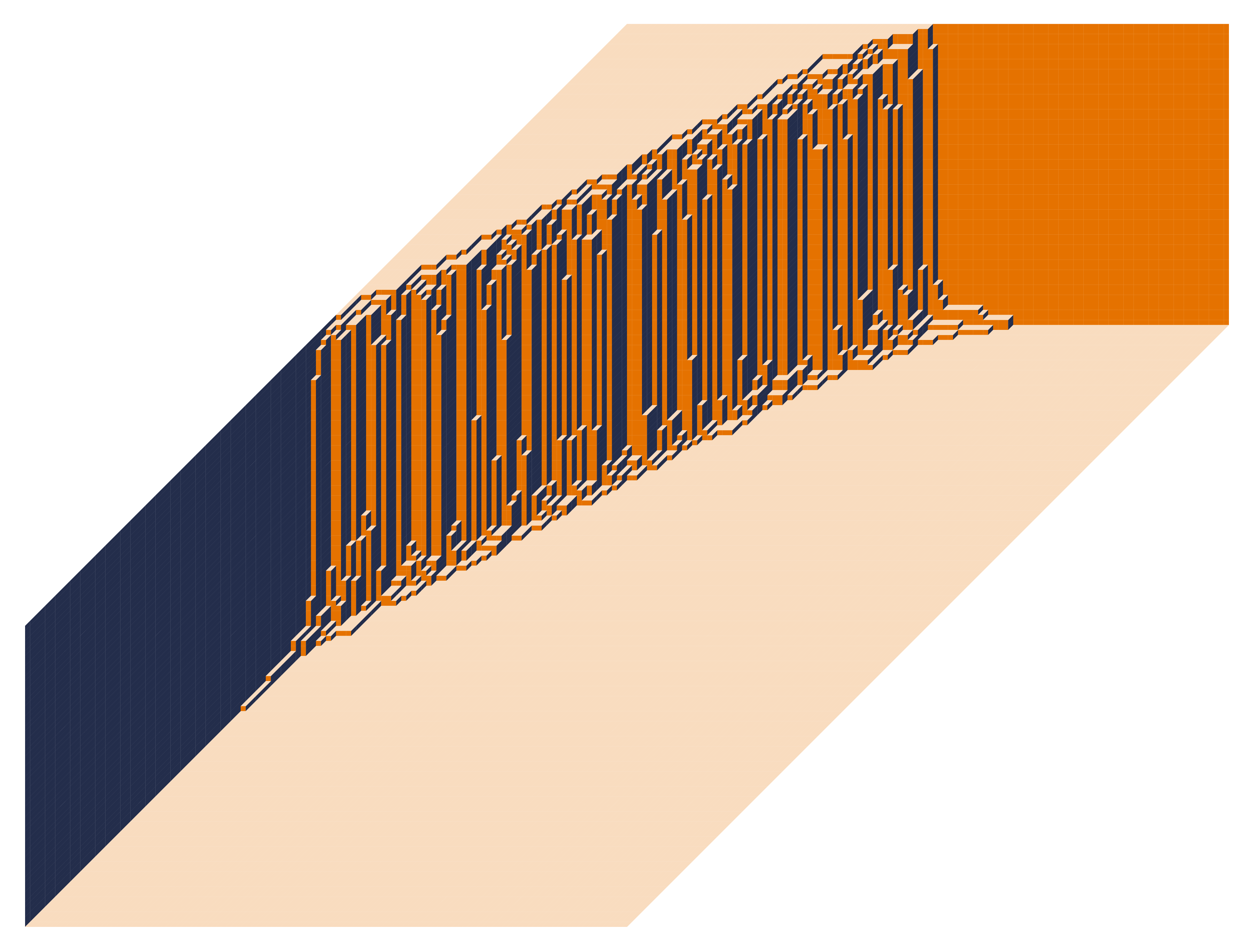}
	\end{tabular}
	\caption{Examples of the waterfall region $\mathcal{W}$
		inside the hexagon in the coordinate system
		$(\mathsf{t},\mathsf{x})$, together with an exact sample
		produced by the shuffling algorithm from \cite{borodin-gr2009q}.
		On the left, we have displayed the large-scale limit of
		the trajectory of the nonintersecting path starting in
		the middle.  Inside~$\mathcal{W}$, this path stays close
		to the center line~$\mathcal{C}$, and in particular, has
		asymptotic slope~$1/2$.  Outside~$\mathcal{W}$, the
		slope is either $0$ or $1$, and the path proceeds right
		or right-up without fluctuations. Limit
		trajectories of all other nonintersecting paths are
		parallel translations of the center one. The parameters
		of the simulations are, from top to bottom:
		$(T,S,N,q)=(300,225,150,0.9), (240,60,120,0.8),
		(240,120,60,0.8)$, and $\kappa=\mathbf{i}$ throughout.
		The horizontal lozenges (the lighter-colored
		ones) are exponentially (in $L$) rare in the waterfall region.
		This behavior corresponds to the asymptotic
		clustering of the paths.}
	\label{fig:waterfall_region}
\end{figure}
\begin{figure}[htpb]
	\centering
	\includegraphics[height=160pt]{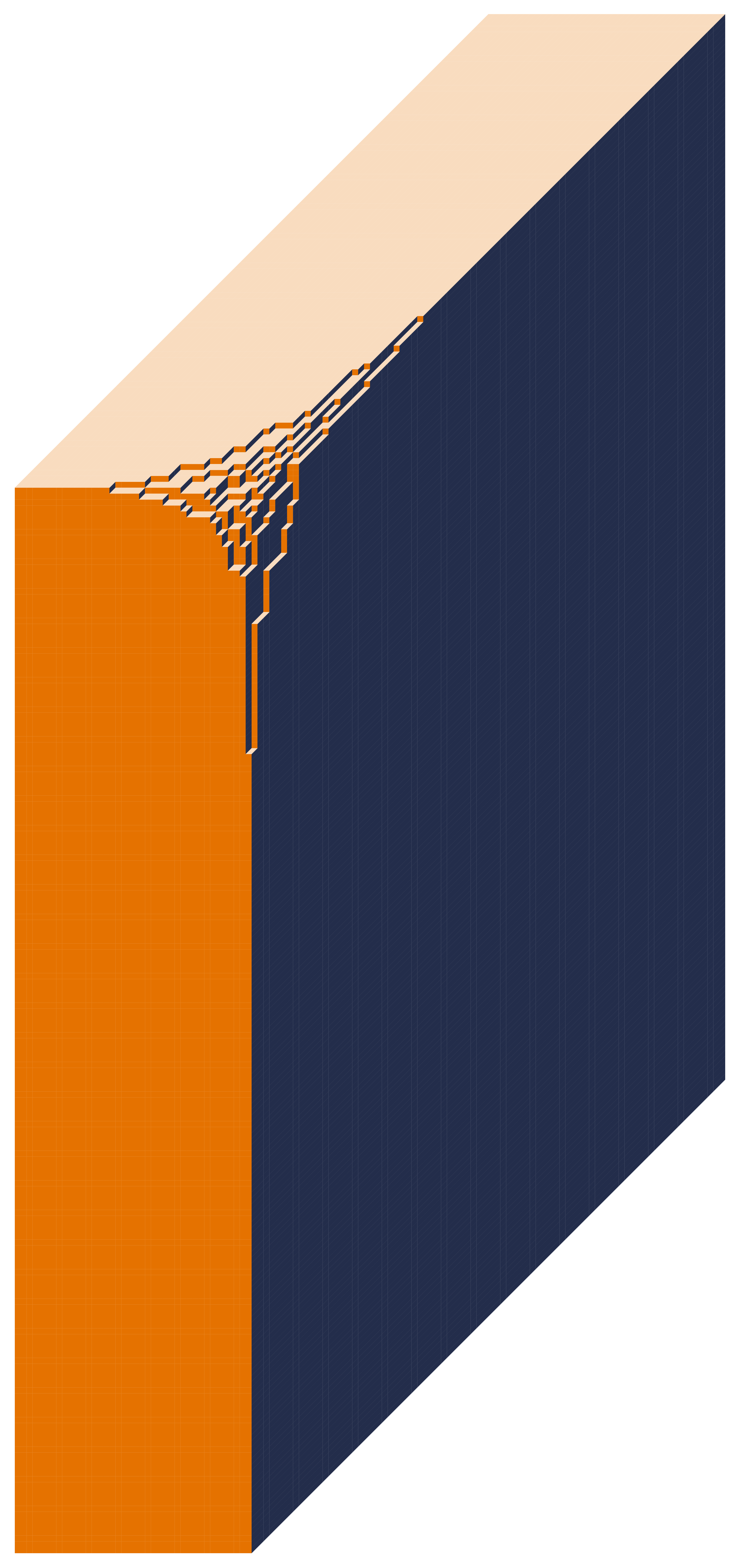}\hspace{80pt}
	\includegraphics[height=160pt]{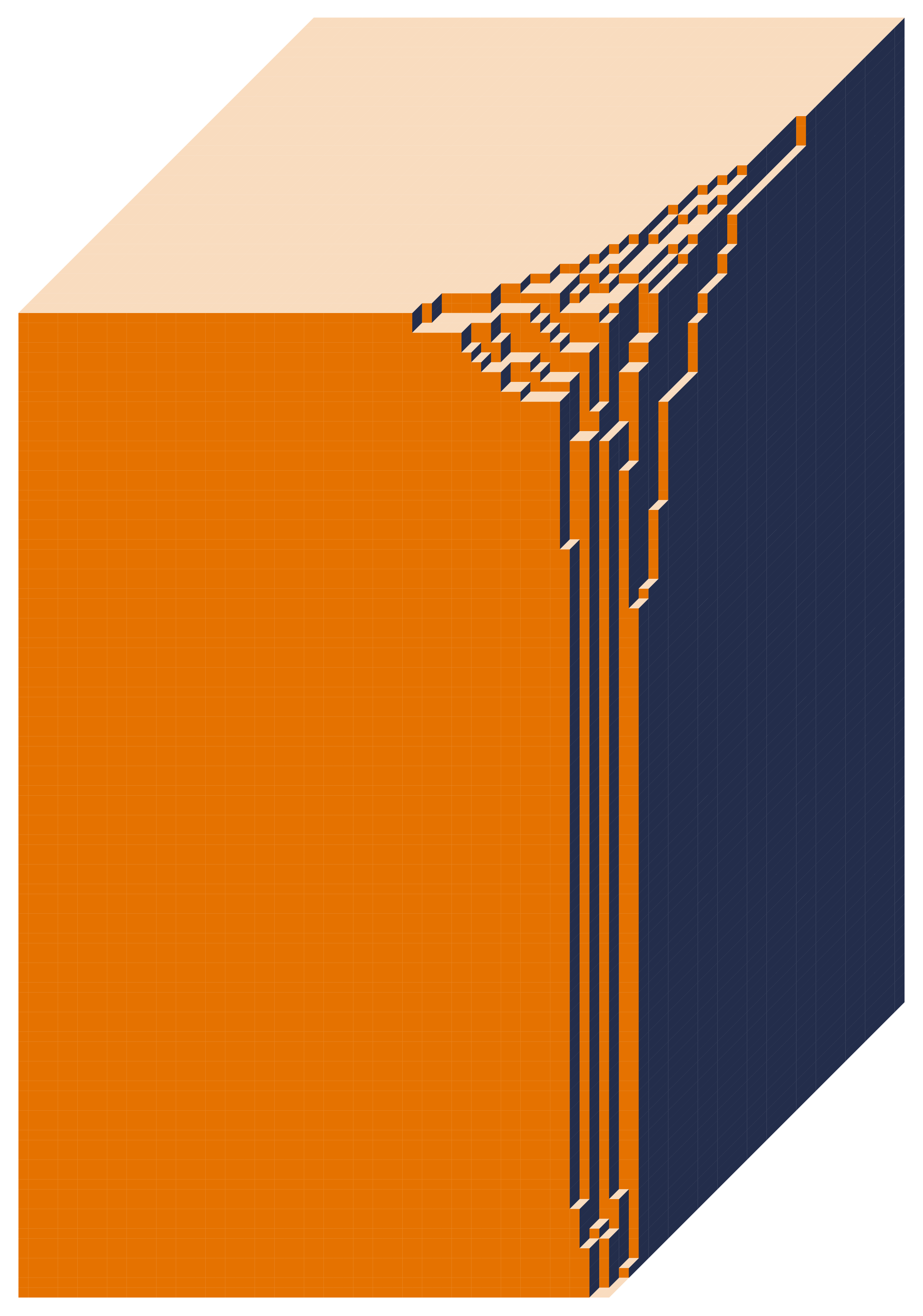}
	\caption{Exact samples from the $q$-Racah measure generated
	by the shuffling algorithm of \cite{borodin-gr2009q} in the
	regime $N>T$. The parameters are $q=0.85$,
	$\kappa=\mathbf{i}$, and $(N,T,S)=(180,120,80)$ (left) and
	$(N,T,S)=(100,90,60)$ (right). In the left-hand picture no
	waterfall behavior is present; presumably the limit shape at
	scales $\ll N$ coincides with that for $q^{\mathsf{vol}}$
	random plane partitions without boundary studied in
	\cite{cerf2001low},
	\cite{ferrari2003step},
	\cite{okounkov2003correlation}. In the right-hand picture we observe a
	``thin waterfall,'' which appears when $T/N$ is close to $1$
	(see also \Cref{rmk:thin_waterfall}).}
	\label{fig:waterfall_region_frozen}
\end{figure}

\begin{definition}
	\label{def:above_below_waterfall}
	Denote by $\mathcal{W}^{\pm}$ the regions above and below the waterfall, that is,
	\begin{equation*}
		\mathcal{W}^{+}\coloneqq
		\begin{cases}
			\{(\mathsf{t},\mathsf{x})\in \mathcal{H}\colon
			\mathsf{x}>\mathsf{N}+(\mathsf{t}-\mathsf{T}+\mathsf{S})^+\},&\textnormal{if $\mathsf{N}>\mathsf{T}$};\\
				\{(\mathsf{t},\mathsf{x})\in \mathcal{H}\colon
			\mathsf{t}_l<\mathsf{t}<\mathsf{t}_r,\ \mathsf{x}>\tfrac12(\mathsf{S+\mathsf{t}+\mathsf{N}})
		\}\\\hspace{20pt}
			\cup
				\{(\mathsf{t},\mathsf{x})\in \mathcal{H}\colon
					\mathsf{t}\le \mathsf{t}_l,\
					\mathsf{x}>\mathsf{N}
			\}\mathbf{1}_{\mathsf{N}>\mathsf{S}}
			\\\hspace{20pt}
			\cup
				\{(\mathsf{t},\mathsf{x})\in \mathcal{H}\colon
					\mathsf{t}\ge \mathsf{t}_r,\
					\mathsf{x}>\mathsf{S}+\mathsf{N}+\mathsf{t}-\mathsf{T}
				\}\mathbf{1}_{\mathsf{S}+\mathsf{N}>\mathsf{T}}
		,
		&\textnormal{if $\mathsf{N}<\mathsf{T}$},\\
		\end{cases}
	\end{equation*}
	and
	\begin{multline*}
		\mathcal{W}^{-}\coloneqq
				\{(\mathsf{t},\mathsf{x})\in \mathcal{H}\colon
			\mathsf{t}_l<\mathsf{t}<\mathsf{t}_r,\ \mathsf{x}<\tfrac12(\mathsf{S+\mathsf{t}-\mathsf{N}})
		\}
			\\
			\cup
				\{(\mathsf{t},\mathsf{x})\in \mathcal{H}\colon
					\mathsf{t}\le \mathsf{t}_l,\
					\mathsf{x}<\mathsf{t}
			\}\mathbf{1}_{\mathsf{N}<\mathsf{S}}
			\cup
				\{(\mathsf{t},\mathsf{x})\in \mathcal{H}\colon
					\mathsf{t}\ge \mathsf{t}_r,\
					\mathsf{x}<\mathsf{S}
				\}\mathbf{1}_{\mathsf{S}+\mathsf{N}<\mathsf{T}}.
	\end{multline*}
	Here the indicator $\mathbf{1}_{\mathsf{N}>\mathsf{S}}$ means that the
	subset in the union is present only if $\mathsf{N}>\mathsf{S}$, and similarly for all
	other indicators.
	In the samples in \Cref{fig:waterfall_region}, right,
	the regions $\mathcal{W}^{\pm}$ are the lighter-colored ones. They consist solely of the horizontal lozenges
	\begin{tikzpicture}[baseline = (current bounding
		box.south),scale=.25]
	\draw[thick] (0,0)--++(1,1)--++(1,0)--++(-1,-1)--cycle;
	\end{tikzpicture}.\footnote{The below region $\mathcal{W}^-$ exists only if
	$\mathsf{N}<\mathsf{T}$. In \Cref{sec:asymptotic_clustering_of_nonintersecting_paths} below,
	we show that for
	$\mathsf{N}>\mathsf{T}$, the nonintersecting paths stay
	as low as they can: First, they go straight with
	slope $0$ until $\mathsf{t}=\mathsf{T}-\mathsf{S}$, and then
	continue diagonally with slope $1$.}

	Also denote $\mathcal{P}\coloneqq \mathcal{H}\setminus(\mathcal{W}^+\cup \mathcal{W}^-)$,
	this is the region inside the hexagon where we expect the nonintersecting paths to cluster together.
	We call $\mathcal{P}$ the \emph{saturation band}. It has constant width~$\mathsf{N}$
	and includes the waterfall region.
\end{definition}

\subsection{Heuristics for the saturation band and the waterfall region}
\label{sub:heuristics_for_saturation_band_and_waterfall_region}

Heuristically, the appearance of the saturation band $\mathcal{P}$ where the nonintersecting
paths cluster together with density~$1$
can be observed from the weights
$\mathsf{w}_{q,\kappa}$ of the holes in the path ensemble, see \eqref{eq:q_kappa_weight_of_ensemble}. Indeed, we have
\begin{equation}
	\label{eq:heuristics_waterfall}
	\mathsf{w}_{q,\kappa}(x-\tfrac{t}{2}+1)=\kappa q^{\frac12+( x-\frac{t+S}{2} )}-\left( \kappa q^{\frac12+( x-\frac{t+S}{2} )} \right)^{-1}.
\end{equation}
Since $\kappa\in \mathbf{i}\mathbb{R}_{>0}$, both summands in
\eqref{eq:heuristics_waterfall} have the same sign. When
$(t,x)$ is close to the center line~$\mathcal{C}$
\eqref{eq:center_line}, both summands are asymptotically
of order $1$. However, when $(t,x)$ is far
from~$\mathcal{C}$, one of the summands dominates and goes
to infinity exponentially in $L$. Therefore, all holes in the path ensemble are
encouraged to stay as far from the center line as
possible. There are two possibilities, depending on whether $\mathsf{t}_l<\mathsf{t}<\mathsf{t}_r$:

\begin{enumerate}[$\bullet$]
	\item
		If $\mathsf{t}_l<\mathsf{t}<\mathsf{t}_r$,
		there is enough room inside the hexagon $\mathcal{H}$ for
		the holes to stay away from~$\mathcal{C}$ symmetrically.
		This leads to the formation of the waterfall region $\mathcal{W}$: It is a band of width
		$\mathsf{N}$ centered around $\mathcal{C}$.
		Since the two summands in
		\eqref{eq:heuristics_waterfall} are symmetric with respect
		to $\mathcal{C}$, the waterfall region $\mathcal{W}$ indeed must inherit this
		symmetry.
	\item
		When $\mathsf{t}\notin [\mathsf{t}_l,\mathsf{t}_r]$ or $\mathsf{N}>\mathsf{T}$,
		there is no room inside $\mathcal{H}$ to fit a symmetric band of width
		$\mathsf{N}$ around~$\mathcal{C}$. Then the holes
		live only on one side of $\mathcal{C}$, and the nonintersecting
		paths are frozen and have slope $0$ or $1$.
\end{enumerate}
In
\Cref{sec:asymptotic_clustering_of_nonintersecting_paths}
below, we engage in a detailed analysis of the $q$-Racah
orthogonal polynomial ensemble probability weights
\eqref{eq:qRacah_ensemble_defn} to make this heuristic
rigorous, and prove our main result (\Cref{thm:waterfall_LLN_intro}).

\section{Asymptotic behavior on a vertical slice via spectral projections}
\label{sec:behavior_on_slice_via_spectral_projections}

In this section, we consider the asymptotic behavior of the
fixed-slice correlation kernel $K_t(x,y)$
\eqref{eq:K_t_x_y_fixed_slice} in the regime described in
\Cref{sub:scaling_and_waterfall_region}.
We employ the spectral description of $K_t$ via the
difference operator $\mathfrak{D}$ given by
\eqref{eq:qRacah_orthonormal_difference_operator_for_hexagon_final_in_s2}.
It turns out that the spectral approach, which successfully
worked throughout the whole hexagon in the traditional scaling regime $q\to 1$,
can only describe the asymptotic behavior
on the center line $\mathcal{C}$ when $q$ is fixed.

\subsection{Limit of the coefficients}
\label{sub:limit_of_coefficients_difference_operator}

Recall the limit regime \eqref{eq:limit_regime}.
We consider a fixed vertical slice at
$t=\lfloor L\ssp \mathsf{t} \rfloor$ inside the hexagon
(where $\mathsf{t}$ is the scaled position of the slice).
We aim to
find the limit of the (suitably shifted and scaled)
difference operator $\mathfrak{D}$
\eqref{eq:qRacah_orthonormal_difference_operator_for_hexagon_final_in_s2}:
\begin{equation}
	\label{eq:diff_operator_scaling_shifting}
	\bigl(\mathfrak{D}^{\mathrm{scaled}}g\bigr)(x)\coloneqq
	\underbrace{
		q^{N+T+(N-|2\lfloor L\mathsf{x}\rfloor-S-t|)^+}
		}_{\eqcolon C_{\mathsf{x}}}
	\Bigl[
	\bigr(\mathfrak{D}g\bigl)(x)+q^{-N+1}(1-q^{N-1})(q^{-T-N}-1)\ssp g(x)\Bigr].
\end{equation}
Here and throughout the paper, we use the notation
\begin{equation}
	\label{eq:a_plus_minus_notation}
	a^+ \coloneqq \max(a,0),\quad
	a^- \coloneqq \min(a,0),\qquad a\in \mathbb{R}.
\end{equation}
We will also abbreviate $C_{\mathsf{x}}\coloneqq q^{N+T+(N-|2\lfloor L\mathsf{x}\rfloor-S-t|)^+}$,
which is the scaling prefactor in \eqref{eq:diff_operator_scaling_shifting}.

The operator $\mathfrak{D}^{\mathrm{scaled}}$ is symmetric
with respect to the standard inner product in
$\ell^2(\mathbb{Z})$ (see \Cref{rmk:selfadjoint}). Its
coefficients have compact support in $x$, so
$\mathfrak{D}^{\mathrm{scaled}}$ is bounded and thus
self-adjoint. The subspace $\ell^2_0(\mathbb{Z})$
consisting of finite linear combinations of the standard
basis vectors $e_x(y)\coloneqq\mathbf{1}_{y=x}$ is an
\emph{essential domain}\footnote{This subspace is also
sometimes called the \emph{core} of an operator.}
for $\mathfrak{D}^{\mathrm{scaled}}$.
Due to the shift in
\eqref{eq:diff_operator_scaling_shifting}, the fixed-slice correlation
kernel $K_t$ \eqref{eq:K_t_x_y_fixed_slice} is the kernel of
the orthogonal projection operator in $\ell^2(\mathbb{Z})$
onto the spectral interval $[0,+\infty)$ of the operator
$\mathfrak{D}^{\mathrm{scaled}}$.

Assume that in \eqref{eq:limit_regime}
we have $\mathsf{N}<\mathsf{T}$, which guarantees the existence of the
waterfall region.
Let $x = \lfloor L\ssp \mathsf{x}\rfloor+\Delta x$
and $t=\lfloor L\ssp \mathsf{t} \rfloor$, where $\Delta x\in
\mathbb{Z}$ is fixed. We also assume that  $\mathsf{t}_l <
\mathsf{t} < \mathsf{t}_r$.
Throughout the rest of this
subsection, we consider the asymptotic behavior of the
coefficients of $\mathfrak{D}^{\mathrm{scaled}}$ one by one,
depending on the position of $(\mathsf{t},\mathsf{x})$
inside the hexagon.

\begin{lemma}[Diagonal coefficient]
	\label{lemma:D_scaled_limit_of_diagonal_coefficient}
	Let $\mathsf{t}_l < \mathsf{t} < \mathsf{t}_r$, and the
	point $(\mathsf{t},\mathsf{x})$ be inside the scaled
	hexagon $\mathcal{H}$ \eqref{eq:scaled_hexagon}.
	Under the scaling described before the lemma,
	the diagonal coefficient of the operator $\mathfrak{D}^{\mathrm{scaled}}$ \eqref{eq:diff_operator_scaling_shifting} has the following limit:
	\begin{equation}
		\label{eq:D_scaled_limit_of_diagonal_coefficient}
		\begin{split}
			&\lim_{L\to+\infty}
			C_{\mathsf{x}}\cdot
			\bigl(
				q^{-N+1}(1-q^{N-1})(q^{-T-N}-1)
				-B(\tilde x)-D(\tilde x)
			\bigr)
			\\
			&\hspace{40pt}
			=
			\begin{cases}
				-1, & \mathsf{x}<\tfrac{1}{2}(\mathsf{S}+\mathsf{t}-\mathsf{N});\\
				-1-\kappa^{-2}(1+q)\ssp q^{-1-2\Delta x},
				& \mathsf{x}=\tfrac12(\mathsf{S}+\mathsf{t}-\mathsf{N}) ,\ \lfloor L\ssp \mathsf{x}\rfloor=\tfrac{1}{2}(S+t-N)
				;\\
				-\kappa^{-2}(1+q)\ssp q^{-1-2\Delta x}, & \tfrac{1}{2}(\mathsf{S}+\mathsf{t}-\mathsf{N})<\mathsf{x}<\tfrac{1}{2}(\mathsf{S}+\mathsf{t});\\
				\dfrac{-\kappa^2 (1+q) \ssp q^{2\Delta x+1}}{(1-\kappa^2 q^{2 \Delta x})(1-\kappa^2 q^{2\Delta x+2})}, & \mathsf{x}=\tfrac12(\mathsf{S}+\mathsf{t}),\ \lfloor L\ssp \mathsf{x}\rfloor=\tfrac{1}{2}(S+t);\\
				-\kappa^2(1+q) \ssp q^{2\Delta x+1}, & \tfrac{1}{2}(\mathsf{S}+\mathsf{t})<\mathsf{x}<\tfrac{1}{2}(\mathsf{S}+\mathsf{t}+\mathsf{N});\\
				-1-\kappa^2(1+q) \ssp q^{2\Delta x+1}, & \mathsf{x}= \tfrac12(\mathsf{S}+\mathsf{t}+\mathsf{N}) ,\ \lfloor L\ssp \mathsf{x}\rfloor=\tfrac{1}{2}(S+t+N);
				\\
				-1, & \mathsf{x}>\tfrac{1}{2}(\mathsf{S}+\mathsf{t}+\mathsf{N}).
			\end{cases}
		\end{split}
	\end{equation}
\end{lemma}
Note that depending on the geometry of the hexagon (determined by $\mathsf{T},\mathsf{S},\mathsf{N}$), not all of the above cases may be present.
\begin{remark}
	\label{rmk:discrete_assumptions_at_the_border}
	In the three border cases (when $\mathsf{x}$ is not in an
	interval), we need to make more precise discrete assumptions
	about the parameters to fix the shift of the $x$-coordinate
	around the border. We also assume that the integers
	multiplied by $\frac12$ are even. The latter can always be
	achieved by replacing $S$ with $S\pm 1$. This shift does not
	affect the generality of the asymptotic analysis because the
	limiting quantity $\mathsf{S}$ can still be an arbitrary
	real number within a suitable interval.
\end{remark}
\begin{proof}[Proof of \Cref{lemma:D_scaled_limit_of_diagonal_coefficient}]
	Combining all the terms in
	$q^{-N+1}(1-q^{N-1})(q^{-T-N}-1)
	-B(\tilde x)-D(\tilde x)$, we have
	\begin{align}
		\nonumber
			&C_{\mathsf{x}}\cdot
			\bigl(
				q^{-N+1}(1-q^{N-1})(q^{-T-N}-1)
				-B(\tilde x)-D(\tilde x)
			\bigr)
			\\ \nonumber &\hspace{10pt}
			=
			\frac{q^{(N-|2\lfloor L\mathsf{x}\rfloor-S-t|)^+}}
			{(q^{S+t}-\kappa^2q^{2x})(q^{S+t}-\kappa^2 q^{2x+2})}\ssp
			\Bigl(
				\bigl[
					q^{2S+t+T+x+1}+q^{S+2t+T+x+1}-q^{2S+2t+T+1}
					\\ \nonumber&\hspace{80pt}
					-q^{2S+2t}
					+q^{-N+S+t+T+x+1}-q^{-N+S+t+T+2x+1}
					\\ \nonumber&\hspace{80pt}
					-q^{-N+S+t+T+2x+2}+q^{-N+2S+2t+x+1}
				\bigr]
				\\ \nonumber&\hspace{40pt}
				-\kappa^2
				\bigl[
					-q^{N+S+t+T+x+1}-q^{S+t+T+2x+1}-q^{S+t+T+2x+3}+q^{N+S+t+T+2x+1}
					\\ \nonumber&\hspace{80pt}
					+q^{N+S+t+T+2x+2}-q^{2S+t+x+1}-q^{S+2t+x+1}-q^{N+2S+2t+x+1}
					\\&\hspace{80pt}
					\label{eq:D_scaled_limit_of_diagonal_coefficient_proof1}
					-q^{S+t+2x}-q^{S+t+2x+2}+q^{2S+t+2x+1}+q^{2S+t+2x+2}
					\\ \nonumber&\hspace{80pt}
					+q^{S+2t+2x+1}+q^{S+2t+2x+2}+q^{S+T+2x+1}+q^{S+T+2x+2}
					\\ \nonumber&\hspace{80pt}
					-q^{S+T+3x+2}+q^{t+T+2x+1}+q^{t+T+2x+2}-q^{t+T+3x+2}
					\\ \nonumber&\hspace{80pt}
					+q^{-N+S+t+2x+1}+q^{-N+S+t+2x+2}-q^{-N+S+t+3x+2}-q^{-N+T+3x+2}
				\bigr]
				\\ \nonumber&\hspace{40pt}
				+\kappa^4
				\bigl[
					-q^{N+S+t+2x+1}-q^{N+S+t+2x+2}+q^{N+S+t+3x+2}+q^{S+3x+2}
					\\ \nonumber&\hspace{80pt}
					+q^{t+3x+2}+q^{N+T+3x+2}-q^{T+4x+3}-q^{4x+2}
				\bigr]
			\Bigr).
	\end{align}
	Let us now consider the cases in \eqref{eq:D_scaled_limit_of_diagonal_coefficient} one by one.

	\medskip\noindent
	\textbf{Case $\mathsf{x}<(\mathsf{S}+\mathsf{t}-\mathsf{N})/2$.}
	We have $(N-|2\lfloor L\mathsf{x}\rfloor-S-t|)^+=0$.
	The dominant contribution from the denominator in \eqref{eq:D_scaled_limit_of_diagonal_coefficient_proof1} is $\kappa^{-4}q^{-4x-2}$. Combining this power with all the powers in the long sum in the numerator, one can directly check that all terms except for $\kappa^{-4}q^{-4x-2}\cdot(-\kappa^4 q^{4x+2})=-1$ vanish. This yields the first case of \eqref{eq:D_scaled_limit_of_diagonal_coefficient}.

	\medskip\noindent
	\textbf{Case $\mathsf{x}=(\mathsf{S}+\mathsf{t}-\mathsf{N})/2$.}
	Here, we still have $(N-|2\lfloor L\mathsf{x}\rfloor-S-t|)^+=0$, and the dominant contribution from the denominator is the same as in the previous case. Plugging $x=\Delta x+\frac12(S+t-N)$ into \eqref{eq:D_scaled_limit_of_diagonal_coefficient_proof1}, one can directly check that there are three surviving terms in the limit as $L\to+\infty$. One term is $(-1)$, as in the previous case. The other two terms arise from the coefficient by $-\kappa^2$, and after multiplication by the contribution of the denominator, they have the form
	\begin{equation}
		\label{eq:D_scaled_limit_of_diagonal_coefficient_proof2}
		-\kappa^{-2}(1+q)\ssp q^{-N+S+t-2x-1}=-\kappa^{-2}(1+q)\ssp q^{-2\Delta x-1}.
	\end{equation}
	This yields the second case of \eqref{eq:D_scaled_limit_of_diagonal_coefficient}.

	\medskip\noindent
	\textbf{Case $(\mathsf{S}+\mathsf{t}-\mathsf{N})/2<\mathsf{x}<(\mathsf{S}+\mathsf{t})/2$.}
	Here the extra factor $q^{(N-|2\lfloor L\mathsf{x}\rfloor-S-t|)^+}
	=
	q^{N-S-t+2 \lfloor L\mathsf{x}\rfloor}$ is not $1$ anymore, while the contribution from the denominator stays the same as in the previous two cases.
	One can check that with this extra factor, there are no new surviving terms in the limit. The summand $(-1)$ multiplied by $q^{N-S-t+2 \lfloor L\mathsf{x}\rfloor}$ vanishes in the limit. The two remaining terms become the same as in \eqref{eq:D_scaled_limit_of_diagonal_coefficient_proof2}:
	\begin{equation*}
		-\kappa^{-2}q^{N-S-t+2 \lfloor L\mathsf{x}\rfloor}(1+q)\ssp q^{-N+S+t-2x-1}=-\kappa^{-2}(1+q)\ssp q^{-2\Delta x-1}.
	\end{equation*}
	This establishes the third case of \eqref{eq:D_scaled_limit_of_diagonal_coefficient}.

	\medskip\noindent
	\textbf{Case $\mathsf{x}=(\mathsf{S}+\mathsf{t})/2$.}
	In the denominator in \eqref{eq:D_scaled_limit_of_diagonal_coefficient_proof1}, the contribution comes from all terms:
	\begin{equation*}
		(q^{S+t}-\kappa^2q^{2x})(q^{S+t}-\kappa^2 q^{2x+2})=q^{2S+2t}(1-\kappa^2 q^{2 \Delta x})(1-\kappa^2 q^{2\Delta x+2}).
	\end{equation*}
	The combined prefactor in front of the numerator is thus $q^{(N-|2\lfloor L\mathsf{x}\rfloor-S-t|)^+-2S-2t}=
	q^{N-2S-2t}$. One can check that with this prefactor, all terms in the long sum in \eqref{eq:D_scaled_limit_of_diagonal_coefficient_proof1} except two vanish in the limit. These two terms come from the same summands as in \eqref{eq:D_scaled_limit_of_diagonal_coefficient_proof2}:
	\begin{equation*}
		-\kappa^2 q^{N-2S-2t} (1+q) \ssp q^{-N+S+t+2x+1}=
		-\kappa^2 (1+q) \ssp q^{2\Delta x+1}.
	\end{equation*}
	This establishes the fourth case of \eqref{eq:D_scaled_limit_of_diagonal_coefficient}.

	\medskip\noindent
	\textbf{Remaining cases $\mathsf{x}>(\mathsf{S}+\mathsf{t})/2$.}
	The remaining three cases in \eqref{eq:D_scaled_limit_of_diagonal_coefficient} are symmetric with the first three cases. In all of them, the dominant contribution from the denominator in \eqref{eq:D_scaled_limit_of_diagonal_coefficient_proof1} is $q^{-2S-2t}$, and the rest of their computation is very similar to the first three cases. We omit the details.
\end{proof}

\begin{lemma}
	[Off-diagonal coefficients]
	\label{lemma:D_scaled_limit_of_off_diagonal_coefficient}
	Let $\mathsf{t}_l < \mathsf{t} < \mathsf{t}_r$, and the point $(\mathsf{t},\mathsf{x})$ is inside the scaled hexagon $\mathcal{H}$ \eqref{eq:scaled_hexagon}.
	Under the scaling described before \Cref{lemma:D_scaled_limit_of_diagonal_coefficient},
	the off-diagonal coefficients of the operator $\mathfrak{D}^{\mathrm{scaled}}$ \eqref{eq:diff_operator_scaling_shifting} have the following limits (cf. \Cref{rmk:discrete_assumptions_at_the_border} for a discussion of assumptions in the border cases):
	\begin{equation}
		\label{eq:D_scaled_limit_of_off_diagonal_coefficient}
		\begin{split}
			&
			\lim_{L\to+\infty}
			C_{\mathsf{x}} \cdot \sqrt{\frac{w^{qR}(\tilde x-1)}{w^{qR}(\tilde  x)}}B(\tilde  x-1)
			=
			\lim_{L\to+\infty}
			C_{\mathsf{x}} \cdot \sqrt{\frac{w^{qR}(\tilde x)}{w^{qR}(\tilde  x-1)}}D(\tilde  x)
			\\
			&\hspace{10pt}
			=
			\begin{cases}
				0,
				& \mathsf{x}<\tfrac{1}{2}(\mathsf{S}+\mathsf{t}-\mathsf{N});\\
				-\kappa^{-2}q^{-2\Delta x +1/2},
				& \mathsf{x}=\tfrac12(\mathsf{S}+\mathsf{t}-\mathsf{N}) ,\ \lfloor L\ssp \mathsf{x}\rfloor=\tfrac{1}{2}(S+t-N)
				;\\
				-\kappa^{-2}q^{-2\Delta x +1/2},
				& \tfrac{1}{2}(\mathsf{S}+\mathsf{t}-\mathsf{N})<\mathsf{x}<\tfrac{1}{2}(\mathsf{S}+\mathsf{t});\\
				\dfrac{-\kappa^2 q^{2\Delta x+1/2}}{(1 - \kappa^2 q^{2\Delta x})\sqrt{(1 - \kappa^2 q^{2\Delta x-1}) (1 - \kappa^2 q^{2\Delta x+1})}},
				& \mathsf{x}=\tfrac12(\mathsf{S}+\mathsf{t}),\ \lfloor L\ssp \mathsf{x}\rfloor=\tfrac{1}{2}(S+t);\\
				-\kappa^{2}q^{2\Delta x +1/2},
				& \tfrac{1}{2}(\mathsf{S}+\mathsf{t})<\mathsf{x}<\tfrac{1}{2}(\mathsf{S}+\mathsf{t}+\mathsf{N});\\
				-\kappa^{2}q^{2\Delta x +1/2},
				& \mathsf{x}= \tfrac12(\mathsf{S}+\mathsf{t}+\mathsf{N}) ,\ \lfloor L\ssp \mathsf{x}\rfloor=\tfrac{1}{2}(S+t+N);
				\\
				0,
				& \mathsf{x}>\tfrac{1}{2}(\mathsf{S}+\mathsf{t}+\mathsf{N}).
			\end{cases}
		\end{split}
	\end{equation}
\end{lemma}
\begin{proof}
	The equality between the two quantities in the first line in \eqref{eq:D_scaled_limit_of_off_diagonal_coefficient} holds before the limit, see \Cref{rmk:selfadjoint}. Thus, we can consider only the first one.
	We have using \eqref{eq:qRacah_BD_coefficients_tilde_x_for_hexagon}--\eqref{eq:qRacah_BD_coefficients_tilde_x_for_hexagon_2_with_weights}:
	\begin{equation}
		\label{eq:D_scaled_limit_of_off_diagonal_coefficient_proof1}
		\begin{split}
			&
			q^{N+T}
			\sqrt{\frac{w^{qR}(\tilde x-1)}{w^{qR}(\tilde  x)}}B(\tilde  x-1)
			=
			\frac{q^{-N+1/2} }
			{(q^{S+t}-\kappa ^2 q^{2 x})\sqrt{(q^{S+t}-\kappa ^2 q^{2 x-1}) (q^{S+t}-\kappa ^2 q^{2 x+1})}}
			\\&\hspace{110pt}
			\times
			\Bigl(
				(1-q^x) (q^{S+t}-q^{T+x})(q^{N+S}-q^x) (q^{N+t}-q^x)
			\\&\hspace{120pt}
			\times
				(q^S-\kappa ^2 q^{N+x})
				(q^t-\kappa ^2 q^{N+x})
				(q^T-\kappa ^2 q^x)
				(q^{S+t}-\kappa ^2 q^x)
			\Bigr)^{\frac12}.
		\end{split}
	\end{equation}
	We have for all seven cases in \eqref{eq:D_scaled_limit_of_off_diagonal_coefficient}:
	\begin{equation*}
		1-q^x\sim 1, \qquad
		q^{S+t}-q^{T+x}\sim q^{S+t}, \qquad
		q^{N+S}-q^x\sim -q^{x}, \qquad
		q^{N+t}-q^x\sim -q^{x}.
	\end{equation*}
	Thus, it remains to address the product of the four other factors that contain $\kappa$.
	When $\mathsf{x}<\frac12(\mathsf{S}+\mathsf{t})$, the dominant contribution coming from the denominator in \eqref{eq:D_scaled_limit_of_off_diagonal_coefficient_proof1} is $\kappa^{-4}q^{-4x}$.
	Let us expand:
	\begin{equation}
		\label{eq:D_scaled_limit_of_off_diagonal_coefficient_proof2}
		\begin{split}
			&\Biggl[\kappa^{-4}q^{-4x-N+1/2}
			\sqrt{q^{S+t+2x}
			(q^S-\kappa ^2 q^{N+x})
			(q^t-\kappa ^2 q^{N+x})
			(q^T-\kappa ^2 q^x)
			(q^{S+t}-\kappa ^2 q^x)}\Biggr]^2
			\\&\hspace{20pt}=
			\kappa^{-8}q^{-6x-2N+S+t+1}
			\Bigl[
				q^{2 S+2 t+T}
				-\kappa^2
				q^{S+t+x} \left(q^{N+S+T}+q^{N+t+T}+q^{S+t}+q^T\right)
				\\&\hspace{60pt}
				+
				\kappa^4
				q^{2 x} \left(q^{2 N+S+t+T}+q^{N+2 S+t}+q^{N+S+2 t}+q^{N+S+T}+q^{N+t+T}+q^{S+t}\right)
				\\&\hspace{60pt}
				- \kappa^6
				q^{N+3 x} \left(q^{N+S+t}+q^{N+T}+q^t+q^S\right)
				+
				\kappa^8
				q^{2 N+4 x}
			\Bigr].
		\end{split}
	\end{equation}
	We consider the limits of
	\eqref{eq:D_scaled_limit_of_off_diagonal_coefficient_proof2} in the first three cases in
	\eqref{eq:D_scaled_limit_of_off_diagonal_coefficient}, namely, when $\mathsf{x}<\frac12(\mathsf{S}+\mathsf{t})$.

	\medskip\noindent
	\textbf{Case $\mathsf{x}<(\mathsf{S}+\mathsf{t}-\mathsf{N})/2$.}
	One can readily check that the limit of each individual power of $q$ in \eqref{eq:D_scaled_limit_of_off_diagonal_coefficient_proof2} is $0$. This establishes the first case in \eqref{eq:D_scaled_limit_of_off_diagonal_coefficient}.

	\medskip\noindent
	\textbf{Case $\mathsf{x}=(\mathsf{S}+\mathsf{t}-\mathsf{N})/2$.}
	Setting
	$\lfloor L\ssp \mathsf{x}\rfloor=\tfrac{1}{2}(S+t-N)$ and $x=\Delta x+\frac12(S+t-N)$, we see that the only surviving term in the limit in \eqref{eq:D_scaled_limit_of_off_diagonal_coefficient_proof2} is $\kappa^{-4}q^{-2N+2 S+2 t-4 x+1}= \kappa^{-4}q^{-4\Delta x+1}$. Taking the square root, we get the second case in \eqref{eq:D_scaled_limit_of_off_diagonal_coefficient}.

	\medskip\noindent
	\textbf{Case $(\mathsf{S}+\mathsf{t}-\mathsf{N})/2<\mathsf{x}<(\mathsf{S}+\mathsf{t})/2$.}
	In this case, the expansion \eqref{eq:D_scaled_limit_of_off_diagonal_coefficient_proof2} gets an additional factor $q^{N-S-t+2 \lfloor L\mathsf{x}\rfloor}$ coming from $C_{\mathsf{x}}$. This factor keeps the only surviving term in the expansion \eqref{eq:D_scaled_limit_of_off_diagonal_coefficient_proof2} to be the same as in the previous case:
	\begin{equation*}
		\kappa^{-4}q^{-2N+2 S+2 t-4 x+1}q^{2(N-S-t+2 \lfloor L\mathsf{x}\rfloor)} =
		\kappa^{-4}q^{4 \lfloor L\mathsf{x}\rfloor-4 x+1}= \kappa^{-4}q^{-4\Delta x+1}.
	\end{equation*}
	This establishes the third case in \eqref{eq:D_scaled_limit_of_off_diagonal_coefficient}.

	\medskip\noindent
	\textbf{Case $\mathsf{x}=(\mathsf{S}+\mathsf{t})/2$.} In this case, all terms in the denominator in \eqref{eq:D_scaled_limit_of_off_diagonal_coefficient_proof1} are of the same order:
	\begin{equation*}
		\begin{split}
			&(q^{S+t}-\kappa ^2 q^{2 x})\sqrt{(q^{S+t}-\kappa ^2 q^{2 x-1}) (q^{S+t}-\kappa ^2 q^{2 x+1})}
			\\&\hspace{60pt}
			=
			q^{2(S+t)}(1 - \kappa^2 q^{2\Delta x})\sqrt{(1 - \kappa^2 q^{2\Delta x-1}) (1 - \kappa^2 q^{2\Delta x+1})}.
		\end{split}
	\end{equation*}
	Combining the dominant power $q^{-2(S+t)}$ with the prefactor $q^{N}$ coming from $C_{\mathsf{x}}$, we see that the numerator in \eqref{eq:D_scaled_limit_of_off_diagonal_coefficient_proof1} behaves as
	\begin{equation*}
		\begin{split}
			&
			q^{\Delta x+1/2}
			\sqrt{
			(q^{\frac12(S-t)}-\kappa ^2 q^{N+\Delta x})
			(q^{\frac12(t-S)}-\kappa ^2 q^{N+\Delta x})
			(q^{T-\frac12(S+t)}-\kappa ^2 q^{\Delta x})
			(q^{\frac12(S+t)}-\kappa ^2 q^{\Delta x})}
			\\&
			\hspace{20pt}\sim
			q^{\Delta x+1/2}
			\sqrt{\kappa^4 q^{2\Delta x}}=-\kappa^2 q^{2\Delta x+1/2}.
		\end{split}
	\end{equation*}
	This yields the fourth case in \eqref{eq:D_scaled_limit_of_off_diagonal_coefficient}.

	\medskip\noindent
	\textbf{Remaining cases $\mathsf{x}>(\mathsf{S}+\mathsf{t})/2$.}
	When $\mathsf{x}>\frac12(\mathsf{S}+\mathsf{t})$, the dominant contribution from the denominator is $q^{-2(S+t)}$, and the remaining limits in \eqref{eq:D_scaled_limit_of_off_diagonal_coefficient_proof1} are obtained in a symmetric manner. We omit the details.
\end{proof}

\Cref{lemma:D_scaled_limit_of_diagonal_coefficient,lemma:D_scaled_limit_of_off_diagonal_coefficient}
complete the proof of \Cref{prop:limit_of_difference_operator_intro}
from the Introduction.
There, we
assumed that $\mathsf{N}<\mathsf{T}$ and
$\mathsf{t}_l<\mathsf{t}<\mathsf{t}_r$.
Let us record the limits in all the other cases
when $(\mathsf{t},\mathsf{x})$ belongs to a region with asymptotically no nonintersecting paths.
These are the
lighter-colored regions inside the hexagons in the exact samples in \Cref{fig:waterfall_region,fig:waterfall_region_frozen}.
Recall the notation $\mathcal{W}^{\pm}$
(\Cref{def:above_below_waterfall}).

\begin{lemma}
	\label{lemma:case_of_no_waterfall}
	Assume that $\mathsf{t}\notin(\mathsf{t}_l,\mathsf{t}_r)$ and $(\mathsf{t},\mathsf{x})\in \mathcal{W}^+\cup \mathcal{W}^-$.
	Then $C_{\mathsf{x}}=q^{N+T}$, and
	the limits as in the two previous lemmas
	take the form
	\begin{equation}
		\label{eq:case_of_no_waterfall}
		\begin{split}
			&\lim_{L\to+\infty}
			q^{N+T}
			\bigl(
				q^{-N+1}(1-q^{N-1})(q^{-T-N}-1)
				-B(\tilde x)-D(\tilde x)
			\bigr)
			=-1,
			\\
			&\lim_{L\to+\infty}
			q^{N+T}\sqrt{\frac{w^{qR}(\tilde x-1)}{w^{qR}(\tilde  x)}}B(\tilde  x-1)
			=
			\lim_{L\to+\infty}
			q^{N+T}\sqrt{\frac{w^{qR}(\tilde x)}{w^{qR}(\tilde  x-1)}}D(\tilde  x)
			= 0.
		\end{split}
	\end{equation}
\end{lemma}
\begin{proof}
	The proof is very similar to the cases in the proofs of
	\Cref{lemma:D_scaled_limit_of_diagonal_coefficient,lemma:D_scaled_limit_of_off_diagonal_coefficient}
	when $(\mathsf{t},\mathsf{x})$ is outside the waterfall region, so we omit it.
\end{proof}

\subsection{Diagonalization on the center line}
\label{sub:diagonalization_of_the_operator_on_the_center_line}

We can explicitly diagonalize the operator that emerges in the limit along the center line. Define
\begin{equation}
	\label{eq:T_center_def}
		\bigl( \mathfrak{T}^{\ssp\mathrm{ctr}} g \bigr)(x)
		\coloneqq
		d^{\ssp\mathrm{ctr}}(x+1)\ssp g(x+1)
		+
		a^{\mathrm{ctr}}(x)\ssp g(x)
		+
		d^{\ssp\mathrm{ctr}}(x)\ssp g(x-1),
\end{equation}
where
\begin{equation*}
			d^{\ssp\mathrm{ctr}}(x)\coloneqq \dfrac{-\kappa^2 q^{2 x+1/2}}{(1 - \kappa^2 q^{2 x})\sqrt{(1 - \kappa^2 q^{2 x-1}) (1 - \kappa^2 q^{2 x+1})}},
			\quad
			a^{\mathrm{ctr}}(x)\coloneqq \dfrac{-\kappa^2 (1+q) \ssp q^{2 x+1}}{(1-\kappa^2 q^{2  x})(1-\kappa^2 q^{2 x+2})},
\end{equation*}
and
we used the variable $x\in \mathbb{Z}$ instead of $\Delta x$
to simplify the notation.
Thanks to the exponential decay of the coefficients as $x\to \pm\infty$, the symmetric operator \eqref{eq:T_center_def} is bounded in $\ell^2(\mathbb{Z})$, hence self-adjoint.

Consider the functions
\begin{equation}
	\label{eq:T_center_eigenfunctions}
	g_n(x)\coloneqq
	q^{x(x+1)}(-\kappa^{2})^x\sqrt{q^{-x}-\kappa^2 q^{x+1}}
	\sum_{i=0}^{n}
	\kappa ^{2 i} q^{(2i-n)x-i(n-i-1)}\,\frac{(q;q)_n}{(q;q)_i(q;q)_{n-i}}
	,
\end{equation}
where $n=0,1,2,\ldots$ and $x\in \mathbb{Z}$.

\begin{proposition}
	\label{prop:center_line_orthogonality}
	The functions $g_n(x)$, $x\in \mathbb{Z}$,
	$n=0,1,2,\ldots$, given by \eqref{eq:T_center_eigenfunctions}
	form an orthogonal basis in
	$\ell^2(\mathbb{Z})$ and are eigenfunctions of the
	operator $\mathfrak{T}^{\ssp\mathrm{ctr}}$. The eigenvalue
	of $g_n(x)$ is $q^{n+1}$, $n=0,1,2,\ldots$.
\end{proposition}
\begin{proof}
	Define
	\begin{equation}
		\label{eq:w_center_orthog_weight}
		w^\mathrm{ctr}(x)\coloneqq q^{2x(x+1)}(-\kappa^{2})^{2x}(q^{-x}-\kappa^2 q^{x+1}),
		\qquad \tilde g_n(x) \coloneqq \frac{g_n(x)}{\sqrt{w^\mathrm{ctr}(x)}}.
	\end{equation}
	For each $n\in\mathbb{Z}_{\ge0}$, $\tilde g_n(x)$ is a Laurent polynomial in $q^x$ which contains powers $q^{(2n-i)x}$, $i=0,\ldots,n $.
	These Laurent polynomials must be eigenfunctions of the modified operator
	\begin{equation*}
		\bigl(\widetilde{\mathfrak{T}}^{\mathrm{ctr}}\tilde g\bigr)(x)\coloneqq
		B^{\mathrm{ctr}}(x)
		\ssp \tilde g(x+1) + a^{\mathrm{ctr}}(x)\ssp \tilde g(x) +
		d^{\ssp\mathrm{ctr}}(x)\ssp \tilde g(x-1),
	\end{equation*}
	where
	\begin{equation*}
		\begin{split}
			B^{\mathrm{ctr}}(x)&\coloneqq
			d^{\ssp\mathrm{ctr}}(x+1)\ssp \sqrt{\frac{w^\mathrm{ctr}(x+1)}{w^\mathrm{ctr}(x)}}
			=
			\frac{q^{4(x+1)}\kappa^4}{(1-\kappa^2 q^{2x+1})(1-\kappa^2 q^{2x+2})},
			\\
			d^{\ssp\mathrm{ctr}}(x)&\coloneqq
			d^{\ssp\mathrm{ctr}}(x)\ssp \sqrt{\frac{w^\mathrm{ctr}(x-1)}{w^\mathrm{ctr}(x)}}=
			\frac{q}{(1-\kappa^2 q^{2x+1})(1-\kappa^2 q^{2x})}.
		\end{split}
	\end{equation*}
	The eigenrelation
	$\bigl(\widetilde{\mathfrak{T}}^{\mathrm{ctr}}\tilde g_n\bigr)(x) = q^{n+1}\tilde g_n(x)$, $n=0,1,\ldots $,
	is equivalent to an identity of Laurent polynomials in $q^x$. The latter identity is verified in a straightforward
	manner using their explicit coefficients.

	The functions $g_n(x)$ belong to $\ell^2(\mathbb{Z})$ thanks to the rapid decay of the power $q^{x(x+1)}$
	at $x\to\pm\infty$. Since they are eigenfunctions of $\mathfrak{T}^{\ssp\mathrm{ctr}}$,
	they are orthogonal in $\ell^2(\mathbb{Z})$.

	Each normalized function $\tilde
	g_n(x)\in\ell^2(\mathbb{Z},w^{\mathrm{ctr}})$ is a polynomial in $(q^{-x}+\kappa^2 q^{x+1})$ of
	degree~$n$. This follows
	by combining the opposite powers $q^{(2i-n)x}$ and $q^{(n-2i)x}$ with the
	same $q$-binomial coefficients
	in front.
	The powers $(q^{-x}+\kappa^2 q^{x+1})^m$, $m\in \mathbb{Z}_{\ge0}$, span the weighted space $\ell^2(\mathbb{Z},w^{\mathrm{ctr}})$.
	This implies
	that the $g_n(x)$'s span $\ell^2(\mathbb{Z})$,
	and completes the proof.
\end{proof}

\Cref{prop:center_line_orthogonality} implies that the operator $\mathfrak{T}^{\ssp\mathrm{ctr}}$ has positive spectrum.
Note that the orthogonal spectral projection onto the nonnegative part of the spectrum of $\mathfrak{T}^{\ssp\mathrm{ctr}}$ is
the identity operator $Id$ in $\ell^2(\mathbb{Z})$, which has the kernel
$\mathbf{1}_{x=y}$, $x,y\in\mathbb{Z}$.

\begin{remark}
	The Laurent polynomials $\tilde g_n$ \eqref{eq:w_center_orthog_weight} are the \emph{Stieltjes-Wigert orthogonal polynomials} $S_n(y;q)$ (e.g., see \cite[Chapter~3.27]{Koekoek1996}), up to normalization and change of variable:
	\begin{equation*}
		\tilde g_n(x)=q^{-n x} (q;q)_n\, S_n\left( -\kappa^2 q^{2x-(n-1)^+};q \right),\qquad
		S_n(y;q)\coloneqq \frac{1}{(q;q)_n}\,_1\phi_1
		\left(\begin{array}{c}q^{-n} \\ 0\end{array} \Big\vert\, q ;-q^{n+1} y\right).
	\end{equation*}
	The Stieltjes-Wigert polynomials are at the same level of
	the $q$-Askey scheme as the $q$-Hermite ones, and their
	orthogonality weight $w^\mathrm{ctr}(x)$ on $\mathbb{Z}$
	(given by \eqref{eq:w_center_orthog_weight}) may be
	thought of as a \mbox{$q$-analogue} of the usual Gaussian
	distribution.  However, due to this change of variable,
	the form of the difference equation for the $\tilde
	g_n$'s in the proof of
	\Cref{prop:center_line_orthogonality} is very different
	from the one for the Stieltjes-Wigert polynomials
	\cite[(3.27.5)]{Koekoek1996}. We do not use the
	connection to the Stieltjes-Wigert polynomials in \Cref{prop:center_line_orthogonality}
	or elsewhere in the
	paper.
\end{remark}

\subsection{Convergence of spectral projections}
\label{sub:convergence_of_the_spectral_projection}

The limits of the difference operator $\mathfrak{D}^{\mathrm{scaled}}$ \eqref{eq:diff_operator_scaling_shifting}
on the center line and outside the waterfall region are
bounded operators in $\ell^2(\mathbb{Z})$
(namely, $\mathfrak{T}^{\ssp\mathrm{ctr}}$ \eqref{eq:T_center_def}
and the negative identity $(-Id)$, respectively).
This allows us to
obtain the convergence of the correlation kernel $K_t$ \eqref{eq:K_t_x_y_fixed_slice}
on a slice in these parts of the hexagon. We employ the following general result:

\begin{theorem}[{\cite[combination of Theorems VIII.24(b) and VIII.25(a)]{reed1972}}]
	\label{thm:funkan}
	Assume that $A_N$ and $A$ are self-adjoint operators on the same Hilbert space $H$,
	and they have a common essential domain $H_0\subset H$.
	Assume that $A_N\varphi\to A\varphi$ as $N\to +\infty$ for all $\varphi\in H_0$.\footnote{The
	convergence on the common essential domain implies that $A_N\to A$ in the strong resolvent sense.}
	Let $a,b\in \mathbb{R}\cup\{\pm\infty\}$, $a<b$, be such that $a,b$ are not eigenvalues of $A$.
	Then
	\begin{equation*}
		P_{(a,b)}(A_N)\ssp\varphi \to P_{(a,b)}(A)\ssp\varphi\quad \textnormal{as $N\to+\infty$\quad for all
		$\varphi\in H$},
	\end{equation*}
	where $P_{(a,b)}(A_N)$ denotes the orthogonal spectral projection onto the interval $(a,b)$ corresponding
	to the operator $A_N$, and similarly for $P_{(a,b)}(A)$.
\end{theorem}

Let us now recall the limit regime for the kernel. We assume that the sides of the hexagon
$T,S,N$ grow proportionally to $L\to+\infty$ as in \eqref{eq:limit_regime},
and the point
$(\mathsf{t},\mathsf{x})$ belongs to the scaled hexagon
$\mathcal{H}$ \eqref{eq:scaled_hexagon}. Recall the notation
$\mathsf{t}_l,\mathsf{t}_r$ \eqref{eq:tl_tr}.
We have the following convergence of the correlation kernel $K_t$
on the center line and outside the waterfall region:

\begin{proposition}[Convergence of the correlation kernel]
	\label{prop:limit_of_kernel}
	We have the following limits of the correlation kernel $K_t$ on a slice.
	On the center line, that is
	$\mathsf{N}<\mathsf{T}$,
	$\mathsf{t}_l<\mathsf{t}<\mathsf{t}_r$,
	$\mathsf{x}=\tfrac12(\mathsf{S}+\mathsf{t})$, and $\lfloor L\ssp \mathsf{x}\rfloor=\tfrac{1}{2}(S+t)$,
	we have
	\begin{equation}
		\label{eq:idenity_limit_of_kernel}
		\lim_{L\to+\infty} K_{\lfloor L\ssp \mathsf{t} \rfloor }(\lfloor L\ssp \mathsf{x}\rfloor+\Delta x,
		\lfloor L\ssp \mathsf{x}\rfloor+\Delta y)=\mathbf{1}_{\Delta x=\Delta y}\quad \textnormal{for all
		$\Delta x,\Delta y\in \mathbb{Z}$}.
	\end{equation}
	Outside the waterfall region, that is,
	for $(\mathsf{t},\mathsf{x})\in \mathcal{W}^+\cup \mathcal{W}^-$,
	we have
	\begin{equation}
		\label{eq:zero_limit_of_kernel}
		\lim_{L\to+\infty}
		K_{\lfloor L\ssp \mathsf{t} \rfloor }
		(\lfloor L\ssp \mathsf{x}\rfloor+\Delta x,
		\lfloor L\ssp \mathsf{x}\rfloor+\Delta y)=0\quad \textnormal{for all
		$\Delta x,\Delta y\in \mathbb{Z}$}.
	\end{equation}
\end{proposition}
\begin{proof}
	The correlation kernel $K_t$ is the kernel
	of the orthogonal projection operator in $\ell^2(\mathbb{Z})$ onto the spectral interval $[0,+\infty)$ of the operator $\mathfrak{D}^{\mathrm{scaled}}$ \eqref{eq:diff_operator_scaling_shifting}.
	By \eqref{eq:qRacah_eigenvalues},
	the eigenvalues of
	$\mathfrak{D}^{\mathrm{scaled}}$ around $0$ are
	equal to
	\begin{equation*}
		\{\ldots,-(1-q) (1-q^T) ,0 ,
			q^{-1}(1-q)(1-q^{T+2})
		,\ldots\}\sim
		\{\ldots,-(1-q),0,q^{-1}(1-q),\ldots\}
		,
	\end{equation*}
	so, for large $L$, the spectral projection onto the open
	interval $(-(1-q)/2,+\infty)$ yields the same operator $K_t$.

	We now apply the general \Cref{thm:funkan}.
	The operator $\mathfrak{D}^{\mathrm{scaled}}$ and the
	limiting operators $(-Id)$ and
	$\mathfrak{T}^{\ssp\mathrm{ctr}}$ are bounded in
	$H=\ell^2(\mathbb{Z})$.
	Taking $H_0=\ell_0^2(\mathbb{Z})$ (the space of finitely supported functions),
	we see that the convergence of the coefficients of $\mathfrak{D}^{\mathrm{scaled}}$
	(\Cref{lemma:D_scaled_limit_of_diagonal_coefficient,lemma:D_scaled_limit_of_off_diagonal_coefficient,lemma:case_of_no_waterfall})
	implies the required convergence on the common essential
	domain $H_0$. This completes the proof.
\end{proof}

\begin{remark}
	\label{rmk:not_so_easy_for_other_cases}
	Inside the waterfall region but not on the center line,
	the limiting operators arising from the
	coefficients of $\mathfrak{D}^{\mathrm{scaled}}$
	(\Cref{lemma:D_scaled_limit_of_diagonal_coefficient,lemma:D_scaled_limit_of_off_diagonal_coefficient,lemma:case_of_no_waterfall})
	are
	not bounded. Moreover, they are
	symmetric, but
	\emph{not self-adjoint}. One can check that each of these operators has
	von Neumann deficiency indices $(1,1)$.
	Thus, these limiting operators each
	possess a one-parameter family of self-adjoint extensions.
	Therefore, the general \Cref{thm:funkan} does not immediately apply
	in these regions.
	We do not address the spectral analysis of
	the limiting non-self-adjoint operators
	outside the center line
	further in this paper.
\end{remark}

\Cref{prop:limit_of_kernel} implies that the probability to
find a path at any given location outside the saturation
band $\mathcal{P}$ goes to zero. In principle, it does not
rule out the presence of individual random paths that roam
in $\mathcal{H}\setminus \mathcal{P}$, nor the presence of
holes between the paths in the saturation band. In
\Cref{sec:asymptotic_clustering_of_nonintersecting_paths}
below,
we address these questions using a different approach
based on the heuristics discussed in
\Cref{sub:heuristics_for_saturation_band_and_waterfall_region}.

\section{Concentration in the $q$-Racah ensemble. Proof of \texorpdfstring{\Cref{thm:waterfall_LLN_intro}}{the law of large numbers}}
\label{sec:asymptotic_clustering_of_nonintersecting_paths}

In this section we employ a method different from the one used in \Cref{sec:behavior_on_slice_via_spectral_projections}. Specifically, we work with the explicit probability weights of the $q$-Racah orthogonal polynomial ensemble (OPE)~\eqref{eq:qRacah_ensemble_defn} to derive a concentration-type estimate for the probability that a hole
\begin{tikzpicture}[baseline = (current
bounding box.south),scale=.25]
\draw[thick] (0,0)--++(1,1)--++(1,0)--++(-1,-1)--cycle;
\end{tikzpicture}
appears in the nonintersecting path ensemble inside the saturation band $\mathcal{P}\subset \mathcal{H}$ (see \Cref{def:waterfall_region,def:above_below_waterfall}).
This probability is expected to be small because, by
\Cref{prop:limit_of_kernel}, the nonintersecting paths do
not typically wander outside the saturation band
$\mathcal{P}$ as $L\to+\infty$. Our objective, however, is
to obtain a uniform bound that applies to every location
where paths might stray into
$\mathcal{H}\setminus\mathcal{P}$, or holes might appear
within $\mathcal{P}$.  We establish an exponential upper
bound for this probability by considering the ratio of
$q$-Racah OPE probabilities of two $N$-particle
configurations on the same slice that differ only by moving
one particle from $(t,x)$ to $(t,y)$.

Our plan is as follows. In
\Cref{sub:estimating_qRacah_probabilities} we identify the
exact leading power of $q$ in this ratio of the $q$-Racah
OPE probabilities. Then, in
\Cref{sub:minimizing_over_configuration,sub:exponential_estimates},
we bound this leading power in the asymptotic regime where
the sides of the hexagon grow proportionally. These bounds
yield the main result of the section,
\Cref{thm:no_holes_in_waterfall_region}, and
complete the proof of
\Cref{thm:waterfall_LLN_intro} from the Introduction.

\subsection{Ratio of the $q$-Racah probabilities}
\label{sub:estimating_qRacah_probabilities}

Fix the dimensions of the hexagon
$(T,S,N)$ as in \Cref{sub:tilings_and_paths},
and consider two points $(t,x)$ and $(t,y)$ on a given vertical slice.
Let us move a particle from $(t,x)$ to $(t,y)$, and
estimate the ratio of the corresponding $q$-Racah OPE
probabilities \eqref{eq:qRacah_ensemble_defn}.
Let $\vec z = (z_1<\ldots<z_{N-1} )$ (with $z_i\ne x,y$ for all $i$) denote the
configuration of the particles which do not move.
Recall the
$q$-Racah orthogonality measure $w^{qR}$ \eqref{eq:wqR} on the segment $\{0,1,\ldots,M \}$,
the four zones~\eqref{eq:hexagon_zone_1}--\eqref{eq:hexagon_zone_4}
inside the hexagon
for the $q$-Racah parameters $(M,\alpha,\beta,\gamma,\delta)$, and the notation
$\mu(x)$~\eqref{eq:mu_qRacah_defn} and $\tilde x$~\eqref{eq:tilde_x}.
We have
\begin{equation}
	\label{eq:qRacah_ratio_prelimit}
	\frac
	{\operatorname{\mathbb{P}}\left( X(t)=\vec z \cup \{x\} \right)}
	{\operatorname{\mathbb{P}}\left( X(t)=\vec z \cup \{y\} \right)}
	=\frac{w^{qR}(\tilde x)}{w^{qR}(\tilde y)}
	\prod_{i=1}^{N-1}\left( \frac{\mu(\tilde x)-\mu(\tilde z_i)}{\mu(\tilde y)-\mu(\tilde z_i)} \right)^2.
\end{equation}
Our first objective is to extract the leading powers of $q$ in this product.
Denote
\begin{equation}
	\label{eq:bold_k_notation}
	\mathbf{k}\coloneqq \frac{\log(-\kappa^2)}{\log q}\in \mathbb{R},
	\quad \textnormal{so that }{-\kappa^2}=q^{\mathbf{k}}.
\end{equation}
Define the functions
\begin{equation}
	\label{eq:W_ratio_big}
	\begin{split}
		&
		\mathscr{W}_{\mathbf{k}}(x\mid T,S,N,t)
		\coloneqq
		S+t-2x-1
		+\min (S,N+x+\mathbf{k}+1)
		\\&
		\hspace{60pt}
		+\min (t,N+x+\mathbf{k}+1)
		+\min (S+t,\mathbf{k}+2 x+1)
		\\&
		\hspace{60pt}
		-\min (T,  x+\mathbf{k}+1)
		-\min (S+t,x+\mathbf{k}+1)
		-\min (S+t,\mathbf{k}+2 x+3)
		,
	\end{split}
\end{equation}
and
\begin{equation}
	\label{eq:G_function_big}
	\begin{split}
		\mathscr{G}_{\mathbf{k}}(x,y,z\mid A)
		\coloneqq
		(z-x)^-
		+
		(z-A+x+\mathbf{k}+1)^-
		-
		(z-y)^-
		-
		(z-A+y+\mathbf{k}+1)^-.
	\end{split}
\end{equation}
In \eqref{eq:G_function_big}, we used the notation \eqref{eq:a_plus_minus_notation}.
In the next two
\Cref{lemma:W_ratio_big_extraction,lemma:G_function_big_extraction}
we obtain estimates (up to multiplicative constants) for the
factors in the product in \eqref{eq:qRacah_ratio_prelimit}.
These estimates are given in terms of the functions
$\mathscr{W}_{\mathbf{k}}$ and $\mathscr{G}_{\mathbf{k}}$.

\begin{lemma}
	\label{lemma:W_ratio_big_extraction}
	There exists an absolute constant $c_1=c_1(q,\kappa)>0$ not depending on the size of the hexagon, such that for all $T,S,N,t$ and points
	$x,y$ on the $t$-th vertical slice of the hexagon, we have
	\begin{equation}
		\label{eq:W_ratio_big_asymptotic}
		c_1
		<
		q^{
			\operatorname{\mathrm{sgn}}(x-y)
		\sum_{u=\min(x,y)}^{\max(x,y)-1}
		\mathscr{W}_{\mathbf{k}}(u\mid T,S,N,t)}
		\cdot
		\frac{w^{qR}(\tilde x)}{w^{qR}(\tilde y)}
		<
		c_1^{-1}.
	\end{equation}
\end{lemma}
\begin{proof}
	We use the first formula in \eqref{eq:qRacah_BD_coefficients_tilde_x_for_hexagon_2_with_weights}
	which provides a unified\ expression
	for the ratio
	$\frac{w^{qR}(\tilde u)}{w^{qR}(\tilde u+1)}$ in all four zones in the hexagon.
	From each
	linear factor in this product,
	we extract
	$q$ raised to the minimal power. In the factors
	not containing $\kappa$, the dominating term which we extract is always the
	same. For example, in $(q^{N+t}-q^{x+1})$, we always have
	$N+t\ge x+1$, so
	the dominating term is $q^{x+1}$.
	In factors containing $\kappa$, we write $-\kappa^2=q^{\mathbf{k}}$,
	and include expressions involving~$\mathbf{k}$ in the comparison.
	The combined power of $q$
	thus arising from $\frac{w^{qR}(\tilde u)}{w^{qR}(\tilde u+1)}$
	becomes equal to
	$\mathscr{W}_{\mathbf{k}}(u\mid T,S,N,t)$ \eqref{eq:W_ratio_big}.
	Representing
	$\frac{w^{qR}(\tilde x)}{w^{qR}(\tilde y)}$ as a telescoping product of
	consecutive ratios, we get the desired power of $q$ in
	\eqref{eq:W_ratio_big_asymptotic}.
	Indeed, if $x<y$, then the compensating sum is from $x$ to $y-1$ with the minus sign, and
	if $x>y$, then we write
	$\frac{w^{qR}(\tilde x)}{w^{qR}(\tilde y)}=
	\bigr(
		\frac{w^{qR}(\tilde y)}{w^{qR}(\tilde x)}
	\bigr)^{-1}$,
	and the compensating sum is from $y$ to $x-1$, but with the plus sign.

	There could be a multiplicative correction
	in extracting the power $q^{\mathscr{W}_{\mathbf{k}}(u\mid T,S,N,t)}$
	from
	$\frac{w^{qR}(\tilde u)}{w^{qR}(\tilde u+1)}$
	if $u\in \mathbb{Z}$ is such that
	both terms in a given linear factor in \eqref{eq:qRacah_BD_coefficients_tilde_x_for_hexagon_2_with_weights}
	are comparable (for example, if $q^{S+t}\sim -\kappa^2 q^{2u+1}$).
	However, the product over all such values $u$ is bounded, as one can see from the
	following computation:
	\begin{equation}
		\label{eq:sample_theta_like_computation}
		\prod\nolimits_{u\in \mathbb{Z}}
		(q^A-\kappa^2 q^u)\ssp
		q^{-\min(A,u+\mathbf{k})}=
		( -q^{\{\mathbf{k}\}};q )_{\infty}
		( -q^{1-\{\mathbf{k}\}};q )_{\infty},
	\end{equation}
	where $\{\mathbf{k}\}=\mathbf{k}-\lfloor \mathbf{k} \rfloor$ is the fractional part of $\mathbf{k}$.
	The right-hand side of \eqref{eq:sample_theta_like_computation} is
	bounded away from $0$ and $+\infty$ by an absolute constant depending only on $q$ and $\kappa$.
	Combining these bounds for the six linear factors in
	\eqref{eq:qRacah_BD_coefficients_tilde_x_for_hexagon_2_with_weights}
	containing $\kappa$, we obtain the multiplicative bounds by $c_1,c_1^{-1}$
	in~\eqref{eq:W_ratio_big_asymptotic}.
	This completes the proof.
\end{proof}

\begin{lemma}
	\label{lemma:G_function_big_extraction}
	There exists an absolute constant $c_2=c_2(q,\kappa)>0$
	not depending on the size of the hexagon, such that for any $t$, any
	configuration
	$\vec z = (z_1<\ldots<z_{N-1} )$,
	and any $x,y$ on the $t$-th vertical slice of the hexagon
	(with $z_i\ne x,y$ for all $i$), we have
	\begin{equation}
		\label{eq:G_function_big_asymptotic}
		c_2
		<
		q^{-2\sum_{i=1}^{N-1}
		\mathscr{G}_{\mathbf{k}}(x,y,z_i\mid S+t)}
		\cdot
		\prod_{i=1}^{N-1}\left( \frac{\mu(\tilde x)-\mu(\tilde z_i)}{\mu(\tilde y)-\mu(\tilde z_i)} \right)^2
		<
		c_2^{-1}.
	\end{equation}
\end{lemma}
\begin{proof}
	For pairwise distinct $x,y,z_i$ on the $t$-th vertical slice inside the hexagon,
	in all four zones \eqref{eq:hexagon_zone_1}--\eqref{eq:hexagon_zone_4}, we have
	\begin{equation}
		\label{eq:mu_ratio_asymptotic_new}
		\left( \frac{\mu(\tilde x)-\mu(\tilde z_i)}{\mu(\tilde y)-\mu(\tilde z_i)} \right)^2 =
		\left(
			\frac{1-q^{z_i-x}}{1-q^{z_i-y}}\ssp
			\frac{1-\kappa^2 q^{z_i-S-t+x+1}}{1-\kappa^2 q^{z_i-S-t+y+1}}
		\right)^2.
	\end{equation}
	Extracting the smallest power of $q$ from each linear factor
	in \eqref{eq:mu_ratio_asymptotic_new}, exactly as in the
	proof of \Cref{lemma:W_ratio_big_extraction}, we obtain the
	factor $q^{2\mathscr{G}_{\mathbf{k}}(x,y,z_i\mid S+t)}$ (see
	\eqref{eq:G_function_big}). This provides the required power
	of $q$ in \eqref{eq:G_function_big_asymptotic}. The
	multiplicative bounds by $c_2$ and $c_2^{-1}$ follow in the
	same way as in \Cref{lemma:W_ratio_big_extraction}.
\end{proof}

\subsection{Minimizing over the \texorpdfstring{$(N-1)$}{(N-1)}-point configurations}
\label{sub:minimizing_over_configuration}

By \Cref{lemma:W_ratio_big_extraction,lemma:G_function_big_extraction},
to estimate the ratio
\eqref{eq:qRacah_ratio_prelimit}
of the $q$-Racah OPE probabilities, it suffices to consider the exponent
\begin{equation}
	\label{eq:W_G_exponent}
	\begin{split}
		\mathscr{E}_{\mathbf{k}}(x,y)&=
		\mathscr{E}_{\mathbf{k}}(x,y\mid T,S,N,t\mid \vec z)
		\\&\hspace{10pt}\coloneqq
		-\operatorname{\mathrm{sgn}}(x-y)
		\sum_{u=\min(x,y)}^{\max(x,y)-1}
		\mathscr{W}_{\mathbf{k}}(u\mid T,S,N,t)
		+2
		\sum_{i=1}^{N-1}
		\mathscr{G}_{\mathbf{k}}(x,y,z_i\mid S+t).
	\end{split}
\end{equation}
We aim to find pairs $(x,y)$ for which $\mathscr{E}_{\mathbf{k}}(x,y)$ is positive and grows with $L$.
If this is the case, then it is advantageous (under the $q$-Racah OPE measure) to replace the configuration $\vec z\cup \{x\}$ by $\vec z\cup \{y\}$.
For an integer $A$, denote
\begin{equation}
	\label{eq:d_x_y_notation}
	d_{\mathbf{k}}^{A}(x;y)\coloneqq
	\left|x-\tfrac12(A-\mathbf{k}-1)\right|-
	\left|y-\tfrac12(A-\mathbf{k}-1)\right|,
\end{equation}
where $\mathbf{k}$ is given by \eqref{eq:bold_k_notation}.
Let us begin to lower bound \eqref{eq:W_G_exponent} by first picking a configuration $\vec z$ which minimizes the second sum:

\begin{lemma}
	\label{lemma:G_function_minimization}
	Fix $A\in \mathbb{Z}$, and
	let $d_{\mathbf{k}}^{A}(x;y)>0$.
	Then for any
	$(N-1)$-point configuration $\vec z$ in $\mathbb{Z}$, we have
	\begin{equation}
		\label{eq:G_function_minimization}
		\sum_{i=1}^{N-1}\mathscr{G}_{\mathbf{k}}(x,y,z_i\mid A)
		\ge
		\min
		\Bigl\{
			\sum_{i\in I}\mathscr{G}_{\mathbf{k}}(x,y,i\mid A),
			\sum_{i\in I'}\mathscr{G}_{\mathbf{k}}(x,y,i\mid A)
		\Bigr\},
	\end{equation}
	where $I$ and $I'$ are $(N-1)$-point intervals in $\mathbb{Z}$ defined as
	\begin{equation}
		\label{eq:G_function_minimization_interval}
		\begin{split}
			I&\coloneqq
			\Bigl[
				\lfloor \tfrac12(A-\mathbf{k}-1) \rfloor -
				\lfloor \tfrac{N-1}2 \rfloor +1,
				\lfloor \tfrac12(A-\mathbf{k}-1) \rfloor
				+
				\lfloor \tfrac N2 \rfloor
			\Bigr]\cap \mathbb{Z},\\
			I'&\coloneqq
			\Bigl[
				\lfloor \tfrac12(A-\mathbf{k}-1) \rfloor -
				\lfloor \tfrac N2 \rfloor +1,
				\lfloor \tfrac12(A-\mathbf{k}-1) \rfloor
				+
				\lfloor \tfrac{N-1}2 \rfloor
			\Bigr]\cap \mathbb{Z}.
		\end{split}
	\end{equation}
\end{lemma}
See \Cref{rmk:lemma:G_function_minimization} below for a discussion of the
intervals $I$ and $I'$ \eqref{eq:G_function_minimization_interval}.
\begin{proof}[Proof of \Cref{lemma:G_function_minimization}]
	Under the assumptions on $x,y$,
	$\mathscr{G}_{\mathbf{k}}$ is nonpositive as a function of $z$.
	Moreover, it weakly decreases for $z\le \tfrac12(A-\mathbf{k}-1)$, weakly increases
	afterwards, and stays constant in an interval around $\tfrac12(A-\mathbf{k}-1)$.
	See \Cref{fig:graphs_g} for an illustration.\begin{figure}[htpb]
		\centering
		\includegraphics[width=.4\textwidth]{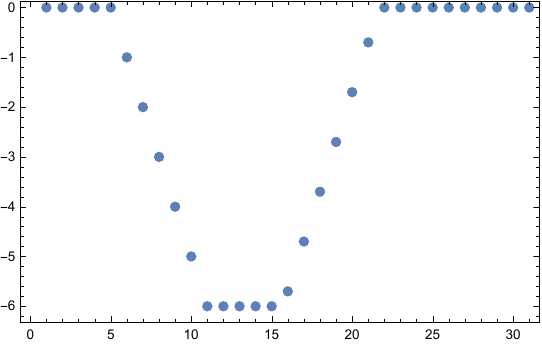}
		\caption{The plot of $\mathscr{G}_{\mathbf{k}}(x,y,z\mid A)$ as a function of $z$,
		where $x=4$, $y=10$, $A=15$, and $\mathbf{k}=4.3$.}
		\label{fig:graphs_g}
	\end{figure}

	This behavior of~$\mathscr{G}_{\mathbf{k}}$ implies that to minimize
	the sum \eqref{eq:G_function_minimization}, we should pick
	the $N-1$ points $z_i$ to be as close as possible to the
	midpoint $\tfrac12(A-\mathbf{k}-1)$. This is achieved by
	densely packing the $z_i$'s together, symmetrically around
	the middle point $\tfrac12(A-\mathbf{k}-1)$. Therefore,
	the configuration $\vec z$ must fill the interval $I$ or
	$I'$ \eqref{eq:G_function_minimization_interval}.

	The intervals $I$ and $I'$ are the same for odd $N$.
	For even $N$, we need to account for the (possibly
	non-integer) shift in one of the strictly monotone parts
	of the plot in \Cref{fig:graphs_g} which arises because of
	the presence of $\mathbf{k}$.  This completes the proof.
\end{proof}
\begin{remark}
	\label{rmk:lemma:G_function_minimization}
	The shift-by-one difference
	between the intervals $I$ and $I'$ in
	\Cref{lemma:G_function_minimization} matters only for
	even~$N$. Then the two sums in the right-hand side of
	\eqref{eq:G_function_minimization}
	differ by
	\begin{equation*}
		\Bigl|
			\mathscr{G}_{\mathbf{k}}
			(x,y,\lfloor \tfrac{1}{2} (A-\mathbf{k}-1)\rfloor+N/2\mid A)
			-
			\mathscr{G}_{\mathbf{k}}
			(x,y,\lfloor \tfrac{1}{2} (A-\mathbf{k}-1)\rfloor-N/2+1\mid A)
			\Bigr|.
	\end{equation*}
	One can check that this expression is always $\le 1$,
	and thus the difference between $I$ and $I'$ will be inessential for
	our asymptotic estimates in \Cref{sub:exponential_estimates} below.
\end{remark}

Notice that in \Cref{lemma:G_function_minimization}, we tacitly assumed
that the minimizing configuration of the $N-1$ particles $\vec z$ fits in an interval around
the middle point $\tfrac12(A-\mathbf{k}-1)=\tfrac12(S+t-\mathbf{k}-1)$.
However, in $\mathscr{E}_{\mathbf{k}}(x,y)$
\eqref{eq:W_G_exponent}, depending on $T,S,N,t$,
the middle point may be at a distance less than $\sim N/2$
from the boundary of the hexagon. In these cases, the
minimizing configuration $\vec z$ must still be densely
packed, but contained in an interval bordering the boundary
of the hexagon.
This corresponds to~$t$ being outside the waterfall bounds
$\lfloor L\ssp \mathsf{t}_l \rfloor \le
t
\le \lfloor L\ssp \mathsf{t}_r \rfloor$.
In more detail, we have:
\begin{lemma}
	\label{lemma:G_function_minimization_outside_waterfall}
	Fix $A\in \mathbb{Z}$, and
	let
	$d_{\mathbf{k}}^{A}(x;y)>0$.
	Assume that the configuration $\vec z$ is restricted to an interval
	$a\le z_1<\ldots<z_N \le b$, and
	\begin{equation}
		\label{eq:G_function_minimization_outside_waterfall}
		a\ge
		\lfloor \tfrac12(A-\mathbf{k}-1) \rfloor -
		\lfloor \tfrac{N-1}2 \rfloor +1
		\qquad
		\textnormal{or}
		\qquad
		b\le
		\lfloor \tfrac12(A-\mathbf{k}-1) \rfloor
		+ \lfloor \tfrac{N-1}2 \rfloor.
	\end{equation}
	Then for any such restricted $(N-1)$-point configuration $\vec z$, we have
	in the two cases in \eqref{eq:G_function_minimization_outside_waterfall}, respectively:
	\begin{equation*}
		\sum\nolimits_{i=1}^{N-1}\mathscr{G}_{\mathbf{k}}(x,y,z_i\mid A)
		\ge
		\sum\nolimits_{i=a}^{a+N-2}\mathscr{G}_{\mathbf{k}}(x,y,i\mid A)
		\qquad
		\textnormal{or}
		\qquad
		\ge
		\sum\nolimits_{i=b-N+2}^{b}\mathscr{G}_{\mathbf{k}}(x,y,i\mid A).
	\end{equation*}
\end{lemma}
\begin{proof}
	This is established similarly to the proof of \Cref{lemma:G_function_minimization}.
\end{proof}

\Cref{lemma:G_function_minimization,lemma:G_function_minimization_outside_waterfall}
establish a lower bound
for the second sum in
the leading $q$-power
$\mathscr{E}_{\mathbf{k}}(x,y)$
\eqref{eq:W_G_exponent}
over all configurations $\vec z$.
The minimizing configuration $\vec z$ is always densely packed.
For uniformity of notation, we denote it by
$\vec z_*=\{z_0,z_0+1,\ldots,z_0+N-2\}$.

In our estimates, the point $y$ (into which we move the particle from $x$)
is in the saturation band. This allows us to strengthen (increase) the
lower bounds of \Cref{lemma:G_function_minimization,lemma:G_function_minimization_outside_waterfall}
by recalling that
the configuration $\vec z$ must satisfy $y\ne z_i$ for all $i$.
Thus, instead of $y$, the minimizing configuration~$\vec z$
must include one of the particles at distance $1$ from the boundary of $\vec z_*$.
The next statement, which formalizes this observation, is straightforward.

\begin{lemma}
	\label{lemma:G_function_minimization_stronger_with_y}
	In the setting of \Cref{lemma:G_function_minimization} or
	\ref{lemma:G_function_minimization_outside_waterfall},
	assume that
	$y\in \vec z_*$, where $\vec z_*$ is
	the densely packed configuration defined above.
	Then for any $(N-1)$-point configuration $\vec z$ in $\mathbb{Z}$ satisfying $z_i\ne x,y$ for all $i$
	(and additional restrictions of \Cref{lemma:G_function_minimization_outside_waterfall},
	if applicable), we have
	\begin{equation}
		\label{eq:G_function_minimization_stronger_with_y}
		\begin{split}
			&
			\sum_{i=1}^{N-1}\mathscr{G}_{\mathbf{k}}(x,y,z_i\mid A)
			\ge
			\sum_{i=z_0}^{z_0+N-2}
			\mathscr{G}_{\mathbf{k}}(x,y,i\mid A)
			\\&
			\hspace{80pt}
			-\mathscr{G}_{\mathbf{k}}(x,y,y\mid A)
			+
			\min
			\bigl\{
				\mathscr{G}_{\mathbf{k}}(x,y,z_0-1\mid A),
				\mathscr{G}_{\mathbf{k}}(x,y,z_0+N-1\mid A)
			\bigr\}.
		\end{split}
	\end{equation}
\end{lemma}
The first line of \eqref{eq:G_function_minimization_stronger_with_y} coincides with the bound established in \Cref{lemma:G_function_minimization,lemma:G_function_minimization_outside_waterfall}. In the proof of \Cref{lemma:bound_waterfall_boundary_L} below, we will show that the terms appearing in the second line are non-negative; hence, the lower bound in \eqref{eq:G_function_minimization_stronger_with_y} is indeed stronger.

\subsection{Estimates of the leading power and the probability of a hole}
\label{sub:exponential_estimates}

If the configuration $\vec z = \vec z_*
=\{z_0,z_0+1,\ldots,z_0+N-2\}$ is densely packed, then we can rewrite
the leading $q$-power
$\mathscr{E}_{\mathbf{k}}(x,y\mid T,S,N,t\mid \vec z_*)$
\eqref{eq:W_G_exponent}
as a discrete
integral between $x$ and $y$:
\begin{lemma}
	\label{lemma:discrete_integral_representation_of_E}
	With the above notation, we have
	\begin{equation}
		\label{eq:discrete_integral_representation_of_E}
		\mathscr{E}_{\mathbf{k}}(x,y\mid T,S,N,t\mid \vec z_*)
		=
		\operatorname{\mathrm{sgn}}(y-x)\sum\nolimits_{u=\min(x,y)}^{\max(x,y)-1}
		\mathscr{H}_{\mathbf{k}}(u,z_0\mid T,S,N,t),
	\end{equation}
	where
	\begin{equation}
		\label{eq:discrete_integral_representation_of_E_star_function_H}
		\begin{split}
			\mathscr{H}_{\mathbf{k}}(u,z_0\mid T,S,N,t)
			&\coloneqq
			\mathscr{W}_{\mathbf{k}}(u\mid T,S,N,t)
			+2\bigl[
				(z_0-S-t+u+\mathbf{k}+1)^-
				-(z_0-u-1)^-
				\\&\hspace{60pt}
				-(z_0-S-t+u+N+\mathbf{k})^-
				+
				(z_0-u+N-2)^-
			\bigr].
		\end{split}
	\end{equation}
\end{lemma}
\begin{proof}
	This follows by two 
	telescopings.
	First, observe that
	\begin{equation*}
		\mathscr{G}_{\mathbf{k}}(x,y,z\mid A)
		=
		\operatorname{\mathrm{sgn}}(y-x)\sum\nolimits_{u=\min(x,y)}^{\max(x,y)-1}
		g_{\mathbf{k}}(u,z\mid A),
	\end{equation*}
	where
	\begin{equation*}
		g_{\mathbf{k}}(x,z\mid A)=
		(z-x)^-
		+
		(z-A+x+\mathbf{k}+1)^-
		-
		\bigl(
		(z-x-1)^-
		+
		(z-A+x+1+\mathbf{k}+1)^-\bigr).
	\end{equation*}
	Then, the sum
	of $\mathscr{G}_{\mathbf{k}}(x,y,i\mid A)$
	over $i\in \vec z=\vec z_*$ becomes
	\begin{equation*}
		2
		\operatorname{\mathrm{sgn}}(y-x)\sum\nolimits_{u=\min(x,y)}^{\max(x,y)-1}
		\sum\nolimits_{z=z_0}^{z_0+N-2}
		g_{\mathbf{k}}(u,z\mid S+t).
	\end{equation*}
	The inner sum telescopes to the last four terms in
	\eqref{eq:discrete_integral_representation_of_E_star_function_H},
	and we are done.
\end{proof}

Let us interpret
$\mathscr{H}_{\mathbf{k}}$
\eqref{eq:discrete_integral_representation_of_E_star_function_H}
viewed as a function of $u$
through cumulative distribution functions (cdfs) of Lebesgue
measures on certain segments of $\mathbb{R}$.
Denote the cdf of the density $1$ Lebesgue measure on $[A,B]$ by
$F_{A,B}(u)$. The next statement is checked directly:

\begin{lemma}
	\label{lemma:interpretation_of_H_via_cdf}
	We have
	\begin{equation}
		\label{eq:interpretation_of_H_via_cdf}
		\begin{split}
			&\mathscr{H}_{\mathbf{k}}(u,z_0\mid T,S,N,t)
			=
			S + t - 2 u-1
			-F_{S-N-\mathbf{k}-1,T-\mathbf{k}-1}(u)
			-F_{t-N-\mathbf{k}-1,t+S-\mathbf{k}-1}(u)
			\\&\hspace{15pt}
			+
			2F_{z_0-1,z_0+N-2}(u)
			+
			2F_{S+t-N-\mathbf{k}-z_0,S+t-1-\mathbf{k}-z_0}(u)
			+
			2F_{\frac12(S+t-\mathbf{k}-1)-1,\frac12(S+t-\mathbf{k}-1)}(u)
			.
		\end{split}
	\end{equation}
\end{lemma}

Recall that our
limit regime
\eqref{eq:limit_regime}
is
$T=\lfloor L\ssp \mathsf{T} \rfloor$,
$S=\lfloor L\ssp \mathsf{S} \rfloor$,
$N=\lfloor L\ssp \mathsf{N} \rfloor$,
$t=\lfloor L\ssp \mathsf{t} \rfloor$,
and $L\to +\infty$.
We are interested in the regime when $\mathsf{T}>\mathsf{N}$
and $\mathsf{t}_l<\mathsf{t}<\mathsf{t}_r$ (see
\eqref{eq:tl_tr} for the notation), that is, when there is
waterfall behavior at the $\mathsf{t}$-th vertical slice of
$\mathcal{H}$. In the other cases when there is no waterfall
on the $\mathsf{t}$-th vertical slice (either
$\mathsf{N}<\mathsf{T}$ and
$\mathsf{t}\notin [\mathsf{t}_l,\mathsf{t}_r]$, or
$\mathsf{N}>\mathsf{T}$), one can
use \Cref{lemma:interpretation_of_H_via_cdf,lemma:discrete_integral_representation_of_E,lemma:G_function_minimization,lemma:G_function_minimization_outside_waterfall,lemma:G_function_minimization_stronger_with_y} to
obtain exponential estimates similarly to \Cref{lemma:easy_bound_waterfall,lemma:bound_waterfall_boundary_L} below.
They will imply that the nonintersecting paths are clustered when they have slope~$0$ or~$1$,
which corresponds to the frozen behavior of the paths.
We do not explicitly formulate these similar estimates here, as they are not needed for the waterfall region.

We begin with a simpler exponential estimate of order $L^2$, which illustrates our approach when the point $y$ is inside the saturation band while $x$ is outside.

\begin{lemma}
	\label{lemma:easy_bound_waterfall}
	Let $\mathsf{t}_l<\mathsf{t}<\mathsf{t}_r$,
	\begin{equation*}
		\left|\mathsf{y}-\tfrac{\mathsf{S}+\mathsf{t}}{2}\right|\le \tfrac{1}{2}\mathsf{N},
		\qquad
		\left|\mathsf{x}-\tfrac{\mathsf{S}+\mathsf{t}}{2}\right|> \tfrac{1}{2}\mathsf{N},
	\end{equation*}
	such that the points $(\mathsf{t},\mathsf{x})$ and $(\mathsf{t},\mathsf{y})$ are inside
	$\mathcal{H}$.
	There exists a constant
	$c>0$ depending on $\mathsf{T},\mathsf{S},\mathsf{N},\mathsf{t},\mathbf{k},
	\left|\mathsf{x}-\tfrac{\mathsf{S}+\mathsf{t}}{2}\right|- \tfrac{1}{2}\mathsf{N}$,
	such that for any $\vec z$ with
	$x,y\ne z_i$ for all $i$,
	any fixed $\Delta x,\Delta y\in \mathbb{Z}$, and any~$L$ large enough, we have
	\begin{equation*}
		\mathscr{E}_{\mathbf{k}}(x,y\mid T,S,N,t\mid \vec z)
		> c \ssp L^2,
		\qquad x=\lfloor L\ssp \mathsf{x} \rfloor +\Delta x,
		\quad y=\lfloor L\ssp \mathsf{y} \rfloor +\Delta y.
	\end{equation*}
\end{lemma}
\begin{proof}
	We can use the estimate of \Cref{lemma:G_function_minimization}, since the condition
	$d_{\mathbf{k}}^{S+t}(x;y)>0$ \eqref{eq:d_x_y_notation} holds for large enough $L$.
	Let us
	lower bound the expression
	$\mathscr{E}_{\mathbf{k}}(x,y\mid T,S,N,t\mid \vec z_*)$
	given
	by~\eqref{eq:discrete_integral_representation_of_E} and~\eqref{eq:interpretation_of_H_via_cdf}.
	Since our slice contains the waterfall
	region, we have $L^{-1}z_0\to \frac12 (\mathsf{S}+\mathsf{t}-\mathsf{N})$ as $L\to +\infty$.
	The function
	$\mathscr{H}_{\mathbf{k}}$
	grows proportionally to $L$:
\begin{equation}
\label{eq:H_asymptotics}
\begin{split}
	&\mathscr{H}_{\mathbf{k}}\bigl(\lfloor L\ssp \mathsf{u} \rfloor , z_0\mid T,S,N,t\bigr)
	\\&\hspace{40pt}=
	L\bigl(
	\underbrace{\mathsf{S}+\mathsf{t}-2\ssp \mathsf{u}
	-
	F_{\mathsf{S}-\mathsf{N},\mathsf{T}}(\mathsf{u})
	-
	F_{\mathsf{t}-\mathsf{N},\mathsf{t}+\mathsf{S}}(\mathsf{u})
	+
	4F_{\frac12(\mathsf{S}+\mathsf{t}-\mathsf{N}),\frac12(\mathsf{S}+\mathsf{t}+\mathsf{N})}(\mathsf{u})}_{\eqcolon \mathsf{H}(\mathsf{u})}
	\bigr)
	+O(1).
\end{split}
\end{equation}
	To see that the error terms are indeed of order $O(1)$,
	observe that
	\begin{equation*}
		\begin{split}
			\bigl|F_{\lfloor L\ssp \mathsf{a} \rfloor +\Delta a,\lfloor L\ssp \mathsf{b} \rfloor+\Delta b }
			(\lfloor L\ssp \mathsf{u} \rfloor )
			-
			L\ssp
			F_{ \mathsf{a},\mathsf{b}}( \mathsf{u} )\bigr|
			&=
			L\ssp \bigl|F_{L^{-1}\lfloor L\ssp \mathsf{a} \rfloor +L^{-1}\Delta a,
			L^{-1}\lfloor L\ssp \mathsf{b} \rfloor+L^{-1}\Delta b }
			(L^{-1}\lfloor L\ssp \mathsf{u} \rfloor )
			-
			F_{ \mathsf{a},\mathsf{b}}( \mathsf{u} )\bigr|
			\\&\le |\Delta a|+|\Delta b|+3,
		\end{split}
	\end{equation*}
	uniformly in $\mathsf{u}$.

	The last cdf in \eqref{eq:interpretation_of_H_via_cdf}
	corresponds to a distribution supported on an interval of length $1$, and so it is of order $O(1)$
	and can be incorporated into the error term. Let us deal with the other four cdf terms in \eqref{eq:interpretation_of_H_via_cdf}.

	The discrete integral
	\eqref{eq:discrete_integral_representation_of_E}
	of $\mathscr{H}_{\mathbf{k}}$
	is the Riemann sum of a continuous integral from $\mathsf{x}$ to $\mathsf{y}$,
	which provides one more power of $L$ in the asymptotics.
	Namely, we have
	\begin{equation*}
		\mathscr{E}_{\mathbf{k}}(x,y\mid T,S,N,t\mid \vec z)
		\ge
		L^2 \int _{\mathsf{x}}^{\mathsf{y}} \mathsf{H}(\mathsf{u})\ssp d\mathsf{u}
		+O(L).
	\end{equation*}
	The integrand $\mathsf{H}(\mathsf{u})$ defined in
	\eqref{eq:H_asymptotics}
	is a piecewise linear function.
	One readily sees that
	\begin{equation*}
		\begin{cases}
			\mathsf{H}(\mathsf{u})=0,& \mathsf{u}\in
			\left[\tfrac{1}{2}(\mathsf{S}+\mathsf{t}-\mathsf{N}),\tfrac{1}{2}(\mathsf{S}+\mathsf{t}+\mathsf{N})\right],
			\\
			\mathsf{H}(\mathsf{u})>0,& \mathsf{u}<
			\tfrac{1}{2}(\mathsf{S}+\mathsf{t}-\mathsf{N}),
			\\
			\mathsf{H}(\mathsf{u})<0,& \mathsf{u}>\tfrac{1}{2}(\mathsf{S}+\mathsf{t}+\mathsf{N}),
		\end{cases}
	\end{equation*}
	see \Cref{fig:graphs_H}, right, for an illustration.
	Under our assumptions on $\mathsf{x},\mathsf{y}$, the integral
	from $\mathsf{x}$ to $\mathsf{y}$ is always positive, which produces the desired estimate.
\end{proof}

\begin{figure}[htpb]
	\centering
	\includegraphics[width=.4\textwidth]{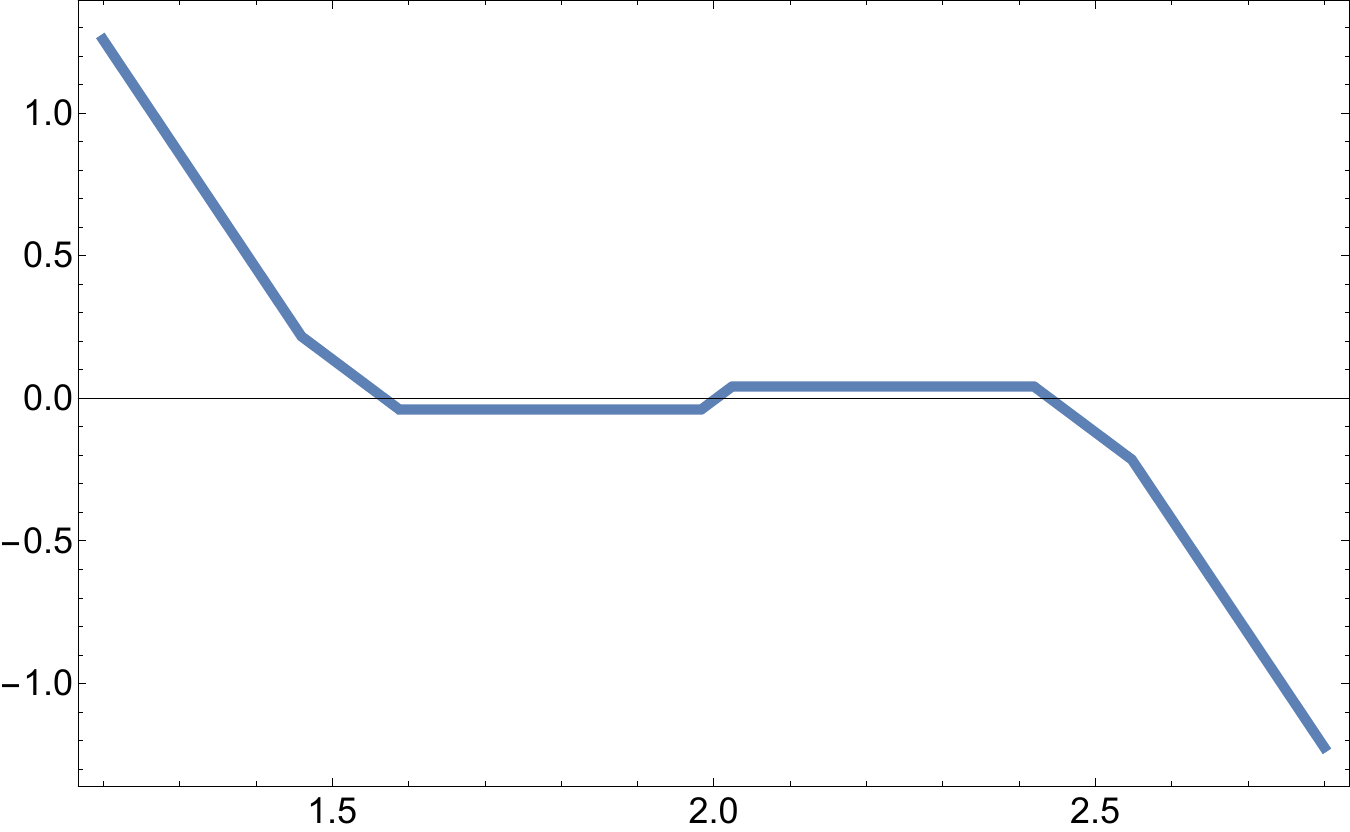}
	\qquad
	\includegraphics[width=.4\textwidth]{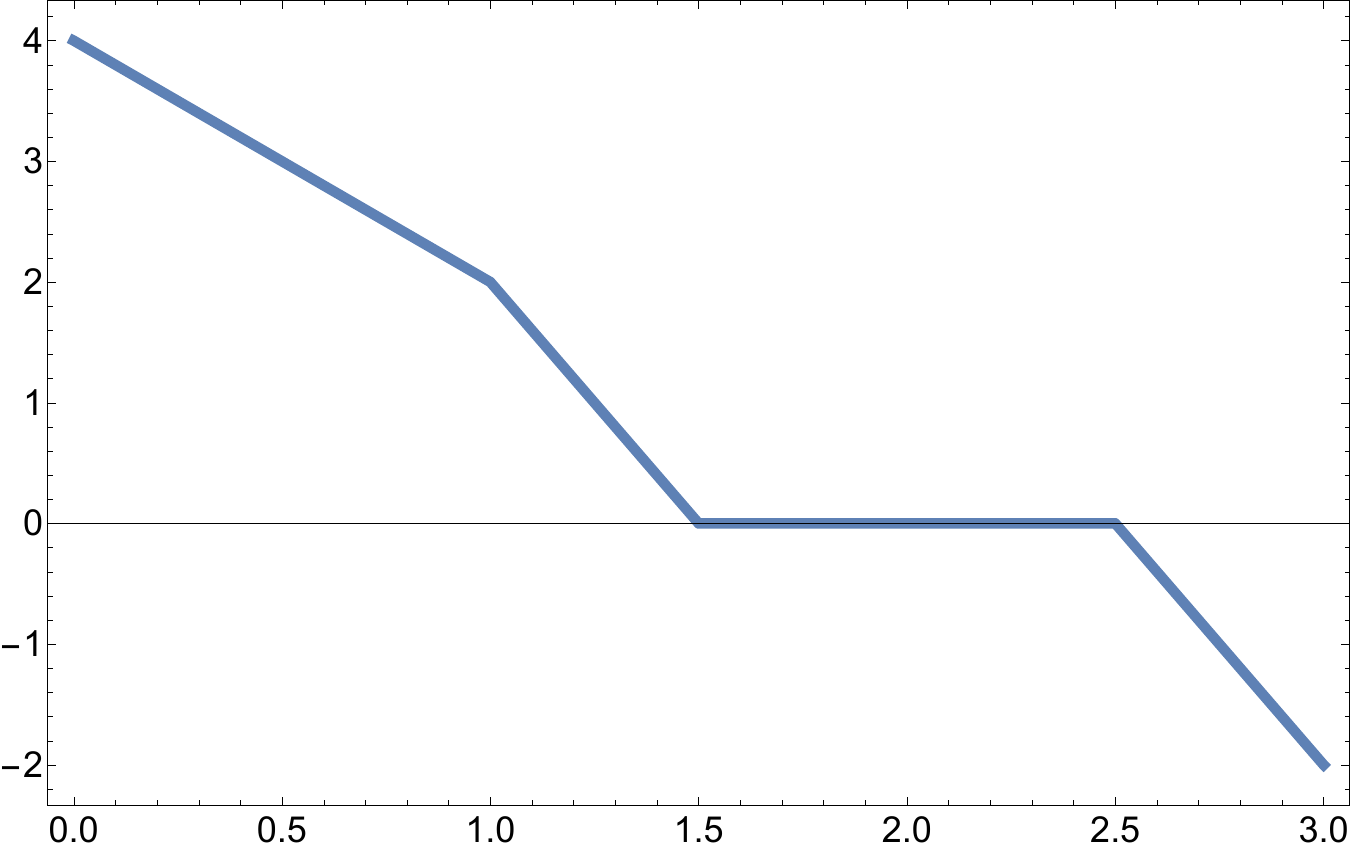}
	\caption{Left:
		The plot of $L^{-1}\mathscr{H}_{\mathbf{k}}(\lfloor L\ssp \mathsf{u} \rfloor ,z_0\mid T,S,N,t)$
		as a function of $\mathsf{u}$.
		Right: The plot of its continuous analogue $\mathsf{H}(\mathsf{u})$ on a larger interval.
		The parameters are
		$(\mathsf{T},\mathsf{S},\mathsf{N})=(4,2,1)$,
		$\mathsf{t}=2$, $\mathbf{k}=-2.2$, and $L=25$, and
		the saturation band is~$(1.5,2.5)$.
		In the left plot, the values close to zero are equal to $\pm L^{-1}$.}
	\label{fig:graphs_H}
\end{figure}

When $x$ can be close to the boundary of the saturation band, we need a more delicate
estimate.
\begin{lemma}
	\label{lemma:bound_waterfall_boundary_L}
	Assume that $\mathsf{t}_l<\mathsf{t}<\mathsf{t}_r$
	and $y=\lfloor L\ssp \mathsf{y} \rfloor+\Delta y$, where $\mathsf{y}$ is in the saturation band, that is,
	$\left|\mathsf{y}-\tfrac{\mathsf{S}+\mathsf{t}}{2}\right|\le \tfrac{1}{2}\mathsf{N}$.
	Let $x$ be such that
	$d_{\mathbf{k}}^{S+t}(x;y)>K$.
	For any $\delta>0$,
	sufficiently large fixed $K>0$, all
	$L$ large enough, and
	any configuration $\vec z$ with $x,y\ne z_i$ for all $i$,
	we have
	\begin{equation}
		\label{eq:best_bound_E}
		\mathscr{E}_{\mathbf{k}}(x,y\mid T,S,N,t\mid \vec z)
		> \min\bigl\{
			\tfrac18\ssp d_{\mathbf{k}}^{S+t}(x;y),
			\delta L
		\bigr\}.
	\end{equation}
\end{lemma}
In \Cref{lemma:bound_waterfall_boundary_L} we assume that the parameters $T,S,N,t$ scale with $L$ in the usual way, and that $y$, lying inside the saturation band, also scales with~$L$. For $x$ we require that it stay at distance at least $K$ from both $y$ and the point $S+t-\mathbf{k}-1-y$ (the mirror image of $y$ with respect to the midpoint of the saturation band), that is,
$d_{\mathbf{k}}^{S+t}(x;y)>K$
(recall the notation~\eqref{eq:d_x_y_notation}).
The constant $K$ may depend on
$\mathsf{T},\mathsf{S},\mathsf{N},\mathsf{t},\mathbf{k},\delta$
but not on $L$.
This allows $x$ to be close to the boundary of the saturation band, and even go inside the band.

Note that if
$x/L\to \mathsf{x}$ and
$\mathsf{x}\notin \{\mathsf{y},\mathsf{S}+\mathsf{t}-\mathsf{y}\}$,
then the right-hand side of \eqref{eq:best_bound_E} grows as~$O(L)$, which produces an upper bound
$\exp(-\mathrm{const}\cdot L)$
for the
ratio of the $q$-Racah probabilities \eqref{eq:qRacah_ratio_prelimit}.
\begin{proof}[Proof of \Cref{lemma:bound_waterfall_boundary_L}]
	Throughout the proof, we assume that $x$ is to the left of
	the midpoint of the saturation band
	$\tfrac12(S+t-\mathbf{k}-1)$.
	The other case is analogous.

	\medskip\noindent
	\textbf{Case 1.}
	When $x<(S+t-N)/2-\alpha\sqrt{L}$ for some sufficiently
	large constant $\alpha>0$ (independent of $L$), an argument
	analogous to the proof of \Cref{lemma:easy_bound_waterfall}
	immediately yields a lower bound of order $\delta L$.
	Indeed, consider the discrete integral of
	$\mathscr{H}_{\mathbf{k}}$ in
	\eqref{eq:discrete_integral_representation_of_E} taken from
	$(S+t-N)/2-\alpha_2\sqrt{L}$ to $(S+t-N)/2-\alpha_1\sqrt{L}$
	with $\alpha_2>\alpha_1>0$. Because
	$\mathscr{H}_{\mathbf{k}}$ is linearly decreasing on this
	interval with slope at most $-1$, this integral is of
	magnitude $\beta L$. By choosing $\alpha_1$ large enough
	(still independent of $L$) we can make the constant $\beta$
	arbitrarily large. The resulting positive contribution
	$\beta L$ offsets the discrete integral of
	$\mathscr{H}_{\mathbf{k}}(u)=\pm1$ over $u$ inside the
	saturation band, which can be negative and is of order
	$O(L)$. We analyze this negative contribution in the next
	case.

	\medskip\noindent
	\textbf{Case 2.}
	It remains to consider the case when
	$x>(S+t-N)/2-\alpha\sqrt L$ for some $\alpha>0$.
	We use
	\Cref{lemma:G_function_minimization_stronger_with_y}, and
	lower bound the right-hand side of
	\eqref{eq:G_function_minimization_stronger_with_y}.
	As in
	\Cref{lemma:discrete_integral_representation_of_E,lemma:interpretation_of_H_via_cdf},
	we can rewrite the first sum over $i=z_0,z_0+1,\ldots,z_0+N-2$
	(this is the sum over $\vec z_*$) as a discrete integral
	of
	$\mathscr{H}_{\mathbf{k}}$ between $x$ and $y$. Under our
	hypotheses, this discrete integral may be negative.
	However, we can lower bound it for large $L$ as
	\begin{equation}
		\label{eq:discrete_integral_representation_of_E_bound}
		\operatorname{\mathrm{sgn}}(y-x)\sum_{u=\min(x,y)}^{\max(x,y)-1}
		\mathscr{H}_{\mathbf{k}}(u,z_0\mid T,S,N,t)
		> -\ssp d_{\mathbf{k}}^{S+t}(x;y)-C(K,\mathbf{k}),\qquad C(K,\mathbf{k})>0.
	\end{equation}
	Indeed, $-d_{\mathbf{k}}^{S+t}(x;y)$ is (up to an additive constant) equal to
	\begin{equation*}
		\operatorname{\mathrm{sgn}}(y-x)\sum_{u=\min(x,y)}^{\max(x,y)-1}
		\bigl(  -\mathbf{1}_{u<(S+t-\mathbf{k}-1)/2}+\mathbf{1}_{u>(S+t-\mathbf{k}-1)/2} \bigr),
	\end{equation*}
	where the above summand mimics $\mathscr{H}_{\mathbf{k}}$ in the saturation band,
	see \Cref{fig:graphs_H}, left.
	The additive constant $C(K,\mathbf{k})$ in
	\eqref{eq:discrete_integral_representation_of_E_bound} is independent of $L$.
	This constant accounts for discrete effects coming from the presence of $\mathbf{k}$,
	and these effects are negligible for large $L$.
	When $x$ is in \textbf{Case 1},
	the right-hand side of \eqref{eq:discrete_integral_representation_of_E_bound} can be lower bounded
	by $-\beta' L$ for fixed $\beta'>0$. Thus, we obtain a bound of order $\delta L$ in \textbf{Case 1},
	which completes the proof of that case.

	Continuing with \textbf{Case 2}, notice that for
	$y$ in the saturation band, we have the extra second line in
	\eqref{eq:G_function_minimization_stronger_with_y}.
	This second line
	turns out to be nonnegative. More precisely, to complete the
	proof, it remains to show that for large $L$ and for $x,y$ satisfying our assumptions, we have
	\begin{equation}
		\label{eq:second_line_nonnegative_estimate}
		2
		\min
		\bigl\{
			\mathscr{G}_{\mathbf{k}}(x,y,z_0-1\mid S+t),
			\mathscr{G}_{\mathbf{k}}(x,y,z_0+N-1\mid S+t)
		\bigr\}
		-2\mathscr{G}_{\mathbf{k}}(x,y,y\mid S+t)
		>
		\tfrac{5}{4}\ssp
		d_{\mathbf{k}}^{S+t}(x;y).
	\end{equation}
	Let us derive the estimate \eqref{eq:second_line_nonnegative_estimate}
	for large $L$ from
	its limiting version.
	To obtain the latter,
	we replace $(x,y,S,t,N)$ by
	$(\mathsf{x},\mathsf{y},\mathsf{S},\mathsf{t},\mathsf{N})$,
	and set $\mathbf{k}=-1$ (which eliminates the additive constants $\mathbf{k}+1$ in
	$\mathscr{G}_{\mathbf{k}}$ and $d_{\mathbf{k}}^{\mathsf{S}+\mathsf{t}}$).
	We also replace $z_0-1$ and $z_0+N-1$ by $\tfrac12(S+t-\mathsf{N})$ and $\tfrac12(S+t+\mathsf{N})$, respectively.
	Observe that then
	$\mathscr{G}_{-1}(\mathsf{x},\mathsf{y},\tfrac12(S+t-\mathsf{N})\mid \mathsf{S}+\mathsf{t})=
	\mathscr{G}_{-1}(\mathsf{x},\mathsf{y},\tfrac12(S+t+\mathsf{N})\mid \mathsf{S}+\mathsf{t})$,
	so no minimum is required in~\eqref{eq:second_line_nonnegative_estimate}.

	The resulting difference of the two sides
	in the limiting version of \eqref{eq:second_line_nonnegative_estimate}
	has the form
	\begin{equation}
		\label{eq:second_line_nonnegative_estimate_proof}
		2
			\mathscr{G}_{-1}(\mathsf{x},\mathsf{y},\tfrac12(S+t-\mathsf{N})\mid \mathsf{S}+\mathsf{t})
		-2\mathscr{G}_{-1}(\mathsf{x},\mathsf{y},\mathsf{y}\mid \mathsf{S}+\mathsf{t})
		-
		\tfrac{5}{4}\ssp
		d_{-1}^{\mathsf{S}+\mathsf{t}}(\mathsf{x};\mathsf{y}).
	\end{equation}
	Observe that
	\eqref{eq:second_line_nonnegative_estimate_proof}
	is invariant under simultaneous shifts of $\mathsf{x},\mathsf{y}$ by $B$ and $\mathsf{S}+\mathsf{t}$ by $-2B$,
	so we can assume that $\mathsf{S}+\mathsf{t}=0$.
	Moreover, \eqref{eq:second_line_nonnegative_estimate_proof} is
	homogeneous under the simultaneous rescaling of $\mathsf{x},\mathsf{y},\mathsf{N}$, so we can set $\mathsf{N}=1$.
	The resulting expression is an even function
	separately in $\mathsf{x}$ and $\mathsf{y}$, so we can further assume that
	$\mathsf{x}\ge \mathsf{y}>0$
	(note that
	$d_{-1}^0(\mathsf{x};\mathsf{y})=|\mathsf{x}|-|\mathsf{y}|$
	must be nonnegative).
	Then,
	the
	resulting specialization of \eqref{eq:second_line_nonnegative_estimate_proof}
	has the form
	\begin{equation}
		\label{eq:second_line_nonnegative_estimate_proof_specialized}
		\tfrac34\ssp(\mathsf{x}-\mathsf{y})+(1-2\mathsf{x})^--(1-2\mathsf{y})^-.
	\end{equation}
	This explicit function is positive for
	$0< \mathsf{x}< \tfrac45$ and $0\le \mathsf{y}\le \min\{\mathsf{x},\tfrac{4-5\mathsf{x}}3\}$,
	see \Cref{fig:second_line_nonnegative_estimate_proof_graph}.
	\begin{figure}[htpb]
		\centering
		\includegraphics[width=.5\textwidth]{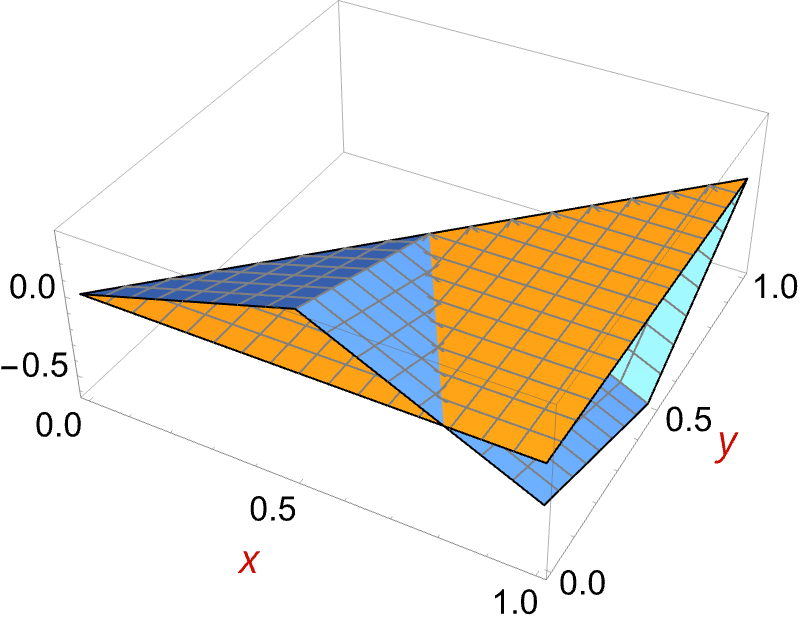}
		\caption{The plot of the function
			\eqref{eq:second_line_nonnegative_estimate_proof_specialized}
			for $\mathsf{x}\ge \mathsf{y}>0$,
			and its horizontal
			cross-section at height zero. We see that the positive part of the function lies above
			the triangle with vertices $(0,0),(\tfrac45,0)$, and $(\tfrac12,\tfrac12)$.}
		\label{fig:second_line_nonnegative_estimate_proof_graph}
	\end{figure}

	Since \eqref{eq:second_line_nonnegative_estimate_proof_specialized}
	is piecewise linear, we conclude that
	\eqref{eq:second_line_nonnegative_estimate} holds
	for $x>(S+t-N)/2-\gamma L$, where the constant $\gamma>0$
	depends only on the distance from $\mathsf{y}$ to the boundary of the saturation band.
	This range is larger than the assumption $x>(S+t-N)/2-\alpha\sqrt L$,
	so we are done with
	\textbf{Case 2}.
	This completes the proof.
\end{proof}

We can now formulate and prove the main result of this section:

\begin{theorem}
	\label{thm:no_holes_in_waterfall_region}
	Let the parameters
	$T,S,N,t$ grow proportionally to $L$ as in \eqref{eq:limit_regime}.
	Fix their scaled values $\mathsf{T},\mathsf{S},\mathsf{N},\mathsf{t}$.
	For any $\varepsilon>0$, there exists $c>0$ (depending on $\mathsf{T},\mathsf{S},\mathsf{N},\mathsf{t},\varepsilon$)
	such that
	\begin{equation*}
		\begin{split}
		&\operatorname{\mathbb{P}}
		\Bigl( \textnormal{there exists a hole }
			\begin{tikzpicture}[baseline = (current
			bounding box.south),scale=.25]
			\draw[thick] (0,0)--++(1,1)--++(1,0)--++(-1,-1)--cycle;
			\end{tikzpicture}
			\textnormal{ in the nonintersecting path ensemble at some point $(t,y)$,}
			\\&\hspace{140pt}
			\textnormal{where
			$
			\mathsf{t}_l+\varepsilon<
			\frac{t}{L}< \mathsf{t}_r-\varepsilon
			$
			and
				$\left|
				\frac{y}{L}-\frac{\mathsf{S}+\mathsf{t}}{2}
				\right|<\frac{\mathsf{N}-\varepsilon}{2}$
			}
		\Bigr)
		<e^{-cL},
		\end{split}
	\end{equation*}
	for all $L$ large enough.
\end{theorem}
\begin{proof}
	Fix $(t,y)$ satisfying the hypotheses of the theorem.
	We need to upper bound
	$\sum\limits_{\mathcal{C}\colon y\notin \mathcal{C}}\operatorname{\mathbb{P}}[\mathcal{C}]$,
	where the sum is over all $N$-particle configurations on the
	$t$-th vertical slice of the hexagon.
	Due to the assumption on $y$,
	for each such $\mathcal{C}$, there exists $x_{\mathcal{C}}$ such that,
	by
	\Cref{lemma:bound_waterfall_boundary_L},
	we have the bound
	\begin{equation}
		\label{eq:bound_on_prob_C_in_proof_of_no_holes_theorem}
		\operatorname{\mathbb{P}}[\mathcal{C}]< q^{\varepsilon' L}
		\operatorname{\mathbb{P}}[(\mathcal{C}\cup \{y\})\setminus \{x_{\mathcal{C}}\}],
	\end{equation}
	where $\varepsilon'>0$ depends on $\varepsilon$. Indeed,
	since $y$ is $(L\varepsilon/2)$-inside the waterfall
	region, we can pick a particle $x_{\mathcal{C}}$ with
	$d^{S+t}_{\mathbf{k}}(x_{\mathcal{C}};y)$ positive and of
	order $L$.  In other words, since not all particles can
	fit $(L\varepsilon/4)$-inside the waterfall region, we can
	pick a particle $x_{\mathcal{C}}$ which did not fit.
	Then, the estimate
	\eqref{eq:best_bound_E}
	from \Cref{lemma:bound_waterfall_boundary_L}
	immediately leads to the desired bound
	\eqref{eq:bound_on_prob_C_in_proof_of_no_holes_theorem}.

	To bound the sum over $\mathcal{C}$, let us rewrite it as
	\begin{equation*}
		\sum_{\mathcal{C}\colon y\notin \mathcal{C}}\operatorname{\mathbb{P}}[\mathcal{C}]
		<
		q^{\varepsilon' L}
		\sum_{x}\sum_{\substack{\mathcal{C}\colon y\notin \mathcal{C}\\x_{\mathcal{C}}=x}}
		\operatorname{\mathbb{P}}[(\mathcal{C}\cup \{y\})\setminus \{x_{\mathcal{C}}\}]
		<
		\mathrm{const}\cdot L\ssp q^{\varepsilon' L}
		\sum_{\mathcal{C}} \operatorname{\mathbb{P}}[\mathcal{C}]\le
		\mathrm{const}\cdot L\ssp q^{\varepsilon' L},
	\end{equation*}
	where the sum over $x$ on the $t$-th vertical slice produces a factor $O(L)$,
	and we bound the sum over some configurations by the sum over all configurations,
	which is at most $1$.

	Thus, we have obtained an exponential upper bound on the
	probability that a hole appears at $(t,y)$. This bound is
	uniform over all $(t,y)$ such that
	$\mathsf{t}_l+\varepsilon< t/L<\mathsf{t}_r-\varepsilon$ and
	$\left|y/L-(\mathsf{S}+\mathsf{t})/2\right|<(\mathsf{N}-\varepsilon)/2$.
	Applying a union bound over all such pairs $(t,y)$ yields
	the desired result.
\end{proof}

Using \Cref{thm:no_holes_in_waterfall_region}, we can enlarge the region where
the correlation kernel $K_t$ converges to the identity to the whole slice inside the waterfall region:

\begin{corollary}
	\label{cor:kernel_goes_to_identity}
	Let
	$\mathsf{N}<\mathsf{T}$,
	$\mathsf{t}_l<\mathsf{t}<\mathsf{t}_r$,
	and $\frac{1}{2}(\mathsf{S}+\mathsf{t}-\mathsf{N})<\mathsf{x}<\frac{1}{2}(\mathsf{S}+\mathsf{t}+\mathsf{N})$.
	Then
	\begin{equation}
		\label{eq:idenity_limit_of_kernel_on_the_full_slice}
		\lim_{L\to+\infty}
		K_{\lfloor L\ssp \mathsf{t} \rfloor }(\lfloor L\ssp \mathsf{x}\rfloor+\Delta x,
		\lfloor L\ssp \mathsf{x}\rfloor+\Delta y)=\mathbf{1}_{\Delta x=\Delta y}\quad \textnormal{for all
		$\Delta x,\Delta y\in \mathbb{Z}$}.
	\end{equation}
\end{corollary}
\Cref{cor:kernel_goes_to_identity} complements
\Cref{prop:limit_of_kernel} which dealt with the center line and the regions outside the waterfall.

\begin{proof}[Proof of \Cref{cor:kernel_goes_to_identity}]
	The clustering of the paths (\Cref{thm:no_holes_in_waterfall_region})
	immediately implies that all local correlations (in the vertical direction)
	inside the waterfall region go to one.
	This means that all principal minors of $K_{\lfloor L\ssp \mathsf{t} \rfloor }$
	converge to $1$. Since the one-dimensional kernel $K_t(x,y)$ living on a slice
	is symmetric in $(x,y)$, we conclude that all its off-diagonal elements
	converge to $0$. This completes the proof.
\end{proof}

We have completed the proof of \Cref{thm:waterfall_LLN_intro} from the Introduction.

\section{Correlations across vertical slices}
\label{sec:2d_correlations}

In this section, we recall the full two-dimensional correlation kernel
$K(s,x;t,y)$
of the $q$-Racah nonintersecting paths ensemble from \cite[Section~7]{borodin-gr2009q},
together with the inverse Kasteleyn matrix for the $q$-Racah
measure on lozenge tilings of the hexagon (equivalently, on dimer configurations on the underlying hexagonal grid).
We also present the pre-limit correlation kernel of the barcode process, which is the
one-dimensional random stepped interface appearing inside the waterfall region.
In the last part (\Cref{sub:2d_correlations_asymptotics_new}), we prove several partial asymptotic results that point toward the limiting
barcode kernel, which will be discussed in \Cref{sec:2d_correlations_asymptotics_numerics_conjectures} below.

\begin{remark}
	\label{rmk:zones}
	We focus in this section and the next on the first
	parameter zone of the hexagon, see
	\eqref{eq:hexagon_zone_1}. Doing so streamlines the
	notation and simplifies several computations. The
	remaining three zones,
	\eqref{eq:hexagon_zone_2}--\eqref{eq:hexagon_zone_4}, can
	be treated in exactly the same way, with only
	straightforward modifications.
\end{remark}

\subsection{Two-dimensional kernel and gauge transformations}
\label{sub:2d_correlations_kernel_new}

Recall from \Cref{sub:dpp_and_operator_interpretation}
the
two-dimensional correlation kernel
$K(s,x;t,y)$ which encodes the distribution of the
$q$-Racah nonintersecting path ensemble.
Here, we present a representation of $K$ in terms of the
$q$-Racah orthogonal
polynomials first obtained in
\cite[Section~7.1]{borodin-gr2009q}.
This representation extends the kernel $K_t(x,y)=K(t,x;t,y)$
from \Cref{prop:qRacah_kernel_OPE}.
This extension is an instance of a general principle
\cite[Section~4]{Borodin2009} for correlation kernels
of measures given by products of determinants.

Recall the notation of the \emph{orthonormal} $q$-Racah polynomials
$f_n(x)$ \eqref{eq:qRacah_orthonormal_basis_f_n},
and let $f_n^t(x)$ denote the specialization of $f_n(x)$
onto the $t$-th vertical slice of the hexagon. That is, in
$f_n^t(x)$, we have substituted the $q$-Racah parameters
$(\alpha,\beta,\gamma,\delta)$
in terms of the hexagon parameters $T,S,N,t$,
assuming that $t\le \min(S-1,T-S-1)$ since we work in the first zone \eqref{eq:hexagon_zone_1}.
Similarly, denote by $w_{t}(x)$ the weight function $w^{qR}(x)$~\eqref{eq:wqR}
specialized to the $t$-th slice.
From
\eqref{eq:wqR},
we have in the first zone \eqref{eq:hexagon_zone_1}:
\begin{equation}
	\label{eq:w_t_x_original}
	w_t(x)=
	\frac{q^{x(2N+T-1)}
	(q^{1-N-S};q)_{x}
	(q^{1-N-t};q)_{x}
	(q^{1-T} \kappa^{2};q)_{x}
	(q^{1-S-t} \kappa^{2};q)_{x}}
	{(q;q)_{x}
	(q^{1-S-t+T};q)_{x}
	(q^{N-S+1} \kappa^{2};q)_{x}
	(q^{N-t+1} \kappa^{2};q)_{x}
	}
	\frac{1-\kappa^{2} q^{1-S-t+2x}}
	{1-\kappa^{2} q^{1-S-t}}.
\end{equation}
To connect to formulas from \cite{borodin-gr2009q},
it is more convenient to renormalize the weight function $w_t$
as in \cite[Theorem~4.1]{borodin-gr2009q}.
Indeed,
while the functions $f_n^t(x)$ (orthonormal in the standard $\ell^2(\mathbb{R})$
space on the $t$-th slice) are defined canonically, the weight function $w_t(x)$
may be \emph{gauge transformed} into $\widetilde w_t(x)=\lambda(t)\ssp w_t(x)$, where
$\lambda(t)$ is any nonvanishing function of $t$ which does not depend on $x$,
but may depend on the hexagon parameters $T,S,N$. This gauge transformation does not affect the
fixed-slice kernel $K_t(x,y)$, but modifies $K(s,x;t,y)$ into
\begin{equation}
	\label{eq:new_2d_kernel_conjugation}
	\widetilde{K}(s,x;t,y)=K(s,x;t,y)\ssp \sqrt{\lambda(s)/\lambda(t)}.
\end{equation}
Clearly, the determinants \eqref{eq:dpp} of the new kernel
$\widetilde{K}$
are the same as those of the original kernel,
so this change does not affect the
determinantal process.

Denote the normalization of the $q$-Racah weight in \cite[Theorem~4.1]{borodin-gr2009q}
by $\widetilde{w}_t(x)$. Let us pick the gauge factor $\lambda(t)$ to be
\begin{equation}
	\label{eq:lambda_t_definition}
	\begin{split}
		\lambda(t)&\coloneqq
		(-1)^t
		\ssp \frac
		{\widetilde{w}_t(x)}
		{w_t(x)}
		=
		\frac{(-1)^t q^{\frac{1}{2} t (2N + t - 1) }
			(q^{N} \kappa^2; \frac{1}{q})_t
			(q^{T-S}; \frac{1}{q})_t}
			{(q^{N}; q)_t (q^{1-S} \kappa^2; \frac{1}{q})_t}
			\\&\hspace{100pt}\times\frac{
			(-1)^S
			(1 - \kappa^2 q^{1-S})
		}{
			(\frac{1}{q}; \frac{1}{q})_{N-1}
			(\frac{1}{q}; \frac{1}{q})_{N+S-1}
			(q; q)_{T-S}
			(q^{1-S} \kappa^2; q)_N
			(q^{1-T} \kappa^2; q)_{N+T}
		}.
	\end{split}
\end{equation}
These these factors are independent of $x$.
Note that the product in the second line in \eqref{eq:lambda_t_definition}
is also independent of $t$, so it does not play a role in
gauge transformations.
We included it only for completeness.

Denote
\begin{equation}
	\label{eq:a_0_a_1_definition}
	\begin{split}
		a_0^t(x)&\coloneqq
		(1-q^{x+T-t-S}) \ssp \frac{1-\kappa^2 q^{x+N-t}}{1-\kappa^2 q^{2 x-t-S+1}}; \\
		a_1^t(x)&\coloneqq
		q^{T+N-1-t}(q^{x-S-N+1}-1) \ssp \frac{1-\kappa^2 q^{x-T+1}}{1-\kappa^2 q^{2 x-t-S+1}}.
	\end{split}
\end{equation}
Define
\begin{equation}
	\label{eq:def:constants_C_n_t}
	\widetilde{\mathscr{C}}_n^t\coloneqq
	\left((q^{n-t-N}-1)(1-q^{T+N-t-n-1})
	\right)^{\frac12}.
\end{equation}
By
\cite[(24)]{borodin-gr2009q},
we have
\begin{equation}
	\label{eq:three_term_relation_operator_U}
	\widetilde{\mathscr{C}}_n^t f_n^{t+1}(x) =
	\sqrt{\frac{\widetilde{w}_t(x-1)}{\widetilde{w}_{t+1}(x)}}\ssp
	f_n^t(x-1)\ssp a_1^t(x-1)+
	\sqrt{\frac{\widetilde{w}_t(x)}{\widetilde{w}_{t+1}(x)}}\ssp
	f_n^t(x)\ssp a_0^t(x)
\end{equation}
for all $n$ and all $x$ on the $t$-th vertical slice.

\begin{remark}
	\label{rmk:signs_vs_BGR}
	We have changed the signs of
	$a_0^t(x),a_1^t(x)$ \eqref{eq:a_0_a_1_definition}
	compared to the coefficients in \cite[(24)]{borodin-gr2009q}
	(given by Lemma~7.4 in that paper),
	and also multiplied $\widetilde{\mathscr{C}}_n^t$ by an extra factor~$\mathbf{i}$
	to make the expression under the square root in \eqref{eq:def:constants_C_n_t}
	positive.
	These sign changes are convenient in our case of imaginary $\kappa$,
	and come from the sign $(-1)^t$ in \eqref{eq:lambda_t_definition}.
\end{remark}

With the notation developed above in this subsection,
we can now write down the two-dimensional correlation kernel:
\begin{proposition}[{\cite[Theorem~7.5]{borodin-gr2009q}}]
\label{prop:2d_kernel_BGR}
	The two-dimensional correlation kernel $\widetilde{K}(s,x;t,y)$
	of the $q$-Racah nonintersecting path ensemble is given by
	\begin{equation}
		\label{eq:K_kernel_two_dimensional}
		\widetilde{K}(s,x;t,y)=
		\begin{cases}
			\displaystyle
			\sum_{n=0}^{N-1}
			(\widetilde{\mathscr{C}}_n^{t} \widetilde{\mathscr{C}}_n^{t+1}\ldots \widetilde{\mathscr{C}}_n^{s-1})^{-1}\ssp
			f_n^s(x)\ssp f_n^t(y),&\qquad  s\ge t;\\[12pt]
			\displaystyle
			-\sum_{n\ge N}
			(\widetilde{\mathscr{C}}_n^s \widetilde{\mathscr{C}}_n^{s+1}\ldots \widetilde{\mathscr{C}}_n^{t-1})\ssp
			f_n^s(x)\ssp f_n^t(y)
			,&\qquad  s<t.
		\end{cases}
	\end{equation}
\end{proposition}
Clearly, for $s=t$, expression \eqref{eq:K_kernel_two_dimensional}
reduces to the fixed-slice kernel $K_t$
\eqref{eq:qRacah_kernel}.

\subsection{Operator for transitions between vertical slices}
\label{sub:2d_correlations_transitions_new}

Using \eqref{eq:three_term_relation_operator_U},
define
the operator $\mathfrak{U}_t$
by
\begin{equation}
	\label{eq:operator_U}
	\bigl(\mathfrak{U}_t f\bigr)(x)\coloneqq
	\sqrt{\frac{\widetilde{w}_t(x-1)}{\widetilde{w}_{t+1}(x)}}\ssp
	a_1^t(x-1)\ssp
	f(x-1)+
	\sqrt{\frac{\widetilde{w}_t(x)}{\widetilde{w}_{t+1}(x)}}\ssp
	a_0^t(x)
	\ssp
	f(x)
	,
\end{equation}
where the coefficients $a_0^t(x),a_1^t(x)$ are given by
\eqref{eq:a_0_a_1_definition}.
It is useful to view $\mathfrak{U}_t$ as an operator that maps functions on the $t$-th vertical slice to functions on the $(t+1)$-st slice. Recall that the $t$-th slice in the first zone \eqref{eq:hexagon_zone_1} is $\{0,1,\ldots,N+t-1\}$.
In particular, identity \eqref{eq:three_term_relation_operator_U}
implies that
\begin{equation}
	\label{eq:operator_U_action_on_f_n}
	(\mathfrak{U}_tf_n^t)(y)=\widetilde{\mathscr{C}}_n^t f_n^{t+1}(y),
	\qquad
	0\le n\le N+t-1,\quad 0\le y\le N+t.
\end{equation}
The $(t+1)$-st slice has one more point than the $t$-th slice. Let us define the
(partial) inverse of $\mathfrak{U}_t$ as follows:
\begin{equation}
	\label{eq:inverse_operator_U}
	\mathfrak{U}_t^{-1}f_n^{t+1}=
	\begin{cases}
		(\widetilde{\mathscr{C}}^t_n)^{-1}f_n^t,& 0\le n\le N+t-1;\\
		0,& n=N+t.
	\end{cases}
\end{equation}
Note also that
$\widetilde{\mathscr{C}}^t_{N+t}=0$, while the formula for $f_{N+t}^t(y)$
\eqref{eq:qRacah_orthonormal_basis_f_n}
also yields zero for $0\le y\le N+t$.
The composition $\mathfrak{U}_t^{-1}\mathfrak{U}_t$ is the identity operator on the $t$-th slice, while
the composition $\mathfrak{U}_t\mathfrak{U}_t^{-1}$ in the opposite order maps
$f_{N+t}^{t+1}$ to zero, and all the other basis functions
$f_n^{t+1}$, $0\le n\le N+t-1$, to themselves.

The operators \eqref{eq:operator_U} and \eqref{eq:inverse_operator_U}
provide the following operator interpretation of the
kernel:
\begin{proposition}[{\cite[Section~8]{borodin-gr2009q}}]
	\label{prop:2d_kernel_operator_interpretation}
	We have 
	\begin{equation}
		\label{eq:K_kernel_two_dimensional_operator_interpretation}
		\widetilde{K}(s,x;t,y)=
		\begin{cases}
			\displaystyle
			\left(
				\mathfrak{U}_t^{-1}
				\mathfrak{U}_{t+1}^{-1}
				\ldots
				\mathfrak{U}_{s-1}^{-1}
			\right)_y
			K_s(x,y)
			,&\qquad  s> t;\\[4pt]
			K_t(x,y),&\qquad s=t;\\[2pt]
			\displaystyle
			\left(
				\mathfrak{U}_{t-1}
				\ldots
				\mathfrak{U}_{s+1}
					\ssp
				\mathfrak{U}_{s}
			\right)_y
			\left( -\mathbf{1}_{x=y}+K_s(x,y) \right)
			,&\qquad  s<t.
		\end{cases}
	\end{equation}
	Here the subscript $y$ indicates that the operators act on
	functions in the variable $y$.
\end{proposition}
\begin{proof}
	We have for $s<t$:
	\begin{equation*}
		\begin{split}
			\widetilde{K}(s,x;t,y)&=-\sum_{n\ge N}
			(\widetilde{\mathscr{C}}_n^s \widetilde{\mathscr{C}}_n^{s+1}\ldots \widetilde{\mathscr{C}}_n^{t-1})\ssp
			f_n^s(x)\ssp f_n^t(y)
			\\&=-
			\left(
			\mathfrak{U}_{t-1}
			\ldots
			\mathfrak{U}_{s+1}
				\ssp
			\mathfrak{U}_{s}
			\right)_y
			\sum_{n\ge N}
			f_n^s(x)\ssp f_n^s(y)
			\\&=
			\left(
			\mathfrak{U}_{t-1}
			\ldots
			\mathfrak{U}_{s+1}
				\ssp
			\mathfrak{U}_{s}
			\right)_y
			\left( -\mathbf{1}_{x=y}+K_s(x,y) \right),
		\end{split}
	\end{equation*}
	where we used the fact that the sum of $f_n^s(x)f_n^s(y)$
	over all $n\ge0$ produces the identity operator, since
	the functions $f_n^s(x)$ are orthonormal.
	Similarly, for $s>t$, we have:
	\begin{equation}
		\label{eq:K_kernel_two_dimensional_operator_interpretation_proof}
		\widetilde{K}(s,x;t,y)=
		\sum_{n=0}^{N-1}
		(\widetilde{\mathscr{C}}_n^{t} \widetilde{\mathscr{C}}_n^{t+1}\ldots \widetilde{\mathscr{C}}_n^{s-1})^{-1}\ssp
		f_n^s(x)\ssp f_n^t(y)
		=
			\left(
				\mathfrak{U}_t^{-1}
				\mathfrak{U}_{t+1}^{-1}
				\ldots
				\mathfrak{U}_{s-1}^{-1}
			\right)_y
			K_s(x,y),
	\end{equation}
	as desired.
	Note that the application of the inverse operators
	in \eqref{eq:K_kernel_two_dimensional_operator_interpretation_proof}
	is valid since in the sum we have $n\le N-1<N+t$.
\end{proof}

\subsection{Inverse Kasteleyn matrix}
\label{sub:2d_correlations_inverse_kasteleyn_new}

The two-dimensional correlation kernel
$\widetilde{K}(s,x;t,y)$ \eqref{eq:K_kernel_two_dimensional}
describes the determinantal structure of the nonintersecting
paths. As is evident from the correspondence between
nonintersecting paths and lozenge tilings (see
\Cref{fig:tiling_and_path_picture}), $\widetilde{K}$ lets us
access only the non-horizontal lozenges
\begin{tikzpicture}[baseline = (current
bounding box.south),scale=.25]
\draw[thick] (0,0)--++(0,1)--++(1,0)--++(0,-1)--cycle;
\end{tikzpicture}
and
\begin{tikzpicture}[baseline = (current
bounding box.south),scale=.25]
\draw[thick] (0,0)--++(0,1)--++(1,1)--++(0,-1)--cycle;
\end{tikzpicture},
and it does not distinguish between these two types. These
are precisely the lozenges that appear in the waterfall
region (see \Cref{fig:waterfall_region,fig:barcode}), and we would like to tell them apart. Therefore, we
must upgrade the correlation kernel to the \emph{inverse
Kasteleyn matrix} which can access all three types of lozenges.
For a general overview of Kasteleyn theory in
the context of random dimer coverings (and, in particular,
lozenge tilings), see \cite[Chapter~3]{Kenyon2007Lecture}.
An expression for the inverse Kasteleyn matrix
through the kernel $\widetilde{K}(s,x;t,y)$
is given in
\cite[Section~7.2]{borodin-gr2009q}.

Represent each
lozenge in a tiling
as an occupied edge of a dimer covering of the
underlying hexagonal lattice. It is useful to coordinatize this
lattice as follows.
Divide the two-dimensional triangular lattice (dual to the hexagonal lattice)
into black and white triangles
(the black triangles point left,
and the white one point right),
such that each lozenge is a union of two neighboring triangles
as follows:
\begin{equation*}
	\begin{tikzpicture}[baseline = (current
	bounding box.south)]
	\draw[ultra thick, fill=gray] (0,0)--++(1,1)--++(0,-1)--cycle;
	\draw[ultra thick] (-2,1)--++(1,0)--++(-1,-1)--cycle;
	\draw[fill=blue,circle] (1,.5) circle (.1) node [right] {$(s,x)$};
	\draw[fill=red,circle] (-2,.5) circle (.1) node [left] {$(t,y)$};
	\end{tikzpicture}
	\hspace{60pt}
	\begin{tikzpicture}[baseline = (current
	bounding box.south)]
	\draw[fill=gray] (0,0)--++(1,1)--++(0,-1)--cycle;
	\draw[fill=white] (1.0,1.0)--++(1,0)--++(-1,-1)--cycle;
	\draw[ultra thick] (0,0)--++(1,1)--++(1,0)--++(-1,-1)--cycle;
	\begin{scope}[shift={(3,1)}]
		\draw[fill=gray] (0,0)--++(1,1)--++(0,-1)--cycle;
		\draw[fill=white] (.0,0)--++(1,0)--++(-1,-1)--cycle;
		\draw[ultra thick] (0,0)--++(1,1)--++(0,-1)--++(-1,-1)--cycle;
	\end{scope}
	\begin{scope}[shift={(5,0)}]
		\draw[fill=gray] (0,0)--++(1,1)--++(0,-1)--cycle;
		\draw[fill=white] (0,0)--++(0,1)--++(1,0)--cycle;
		\draw[ultra thick] (0,0)--++(0,1)--++(1,0)--++(0,-1)--cycle;
	\end{scope}
	\end{tikzpicture}
\end{equation*}
Let the lattice coordinates of a triangle be the coordinates of the midpoint of its vertical side.
These triangles represent vertices in the hexagonal grid, and
lozenges (unions of two neighboring triangles) correspond to edges included in a dimer covering.
To each lozenge, we associate its white and black triangle's coordinates. The Kasteleyn
matrix has rows and columns indexed by the white and black triangles, respectively, and is defined by
\begin{equation}
	\label{eq:Kasteleyn_matrix}
	\mathrm{Kast}\left( t,y;s,x \right)
	\coloneqq
	\begin{cases}
		\mathsf{w}_{q, \kappa}\bigl( y-\tfrac t2+1\bigr), & s = t, \ x                             = y;\\
		1,             & s                                    = t+1,                           \ x = y+1;\\
		1,             & s                                    = t+1,                           \ x = y;\\
		0,             & \textnormal{otherwise}.
	\end{cases}
\end{equation}
where $\mathsf{w}_{q,\kappa}(\cdot )$ is the weight of a horizontal lozenge
\begin{tikzpicture}[baseline = (current
bounding box.south),scale=.25]
\draw[thick] (0,0)--++(1,1)--++(1,0)--++(-1,-1)--cycle;
\end{tikzpicture}
\eqref{eq:q_kappa_weight_of_ensemble}.

The role of the inverse of the matrix \eqref{eq:Kasteleyn_matrix} is
that for any $(t_i,y_i)$ and $(s_i,x_i)$, $i=1,\ldots,m $, the probability
that a random lozenge tiling contains all $m$ lozenges
$[ (t_i,y_i);(s_i,x_i)]$
(the coordinates refer to the white and black triangles forming the lozenge)
is given by
\cite[Corollary~3]{Kenyon2007Lecture}:
\begin{equation}
	\label{eq:probability_of_lozenges_det_via_inverse_kasteleyn}
	\prod_{i=1}^m \mathrm{Kast}(t_i,y_i;s_i,x_i)
	\cdot
	\det\left[ \mathrm{Kast}^{-1}(s_i,x_i;t_j,y_j) \right]_{i,j=1}^{m}.
\end{equation}

\begin{proposition}[{\cite[Theorem~7.6]{borodin-gr2009q}}]
	\label{prop:inverse_kasteleyn_matrix}
	The inverse Kasteleyn matrix is expressed through the two-dimensional correlation kernel
	$\widetilde{K}$ \eqref{eq:K_kernel_two_dimensional}
	as follows:
	\begin{equation}
		\label{eq:inverse_kasteleyn_matrix}
		\mathrm{Kast}^{-1}\left( s,x;t,y \right)
		=
		\frac{G(s,x)}{G(t,y)}
		\frac{\mathbf{1}_{s=t}\mathbf{1}_{x=y}- \widetilde{K}(s,x;t,y)}{\mathsf{w}_{q,\kappa}(x-\tfrac s2+1)},
	\end{equation}
	where
	$(s,x)$ and $(t,y)$ are the coordinates of a black (resp., white) triangle.
	The factor $G$ is
	\begin{equation}
		\label{eq:G_factor_conjugation}
		G(t,x)\coloneqq
		\frac{1}{\sqrt{\widetilde{w}_t(x)}}
		\frac{(-1)^{x} \kappa^{-t}
		q^{x(T+N-t-1)+t(S/2-1/2)+t(t+1)/4} \left(1-\kappa^2 q^{2x-t-S+1}\right)}
		{(\frac{1}{q}; \frac{1}{q})_{S+N-1-x} (q;q)_{T-S+x-t} (\kappa^2 q^{x-T+1}; q)_{T+N-t}}.
	\end{equation}
\end{proposition}

\begin{remark}
	\label{rmk:signs_vs_BGR_Kasteleyn}
	The function $G(t,x)$ \eqref{eq:G_factor_conjugation}
	differs from the one in \cite[(25)]{borodin-gr2009q} by the sign
	$(-1)^{t}$, which reflects the sign gauge in the definition of the correlation kernel,
	see \Cref{rmk:signs_vs_BGR}.
\end{remark}

\subsection{Pre-limit barcode kernel for a specific hexagon}
\label{sub:2d_correlations_prelimit_barcode_kernel_new}

Let us now focus on the waterfall region
of the $q$-Racah random lozenge tiling ensemble, see
\Cref{fig:barcode} for an illustration.
\begin{figure}[htpb]
	\centering
	\includegraphics[width=0.9\textwidth]{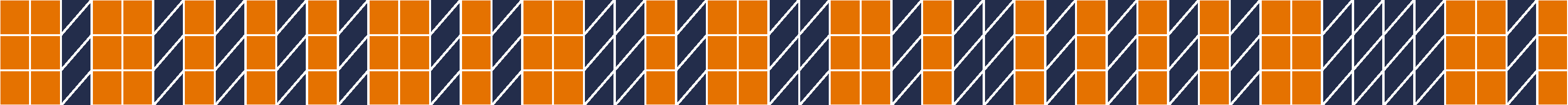}
	\caption{Local structure of the one-dimensional random stepped interface
		in the waterfall region (see \Cref{fig:waterfall_region} for the global picture).
		This barcode configuration
		is encoded as a sequence of particles (square lozenges, the lighter ones in the figure)
		and holes (vertical lozenges, the darker ones in the figure).
		The binary sequence (encoding particles and holes) for this barcode configuration is
		$110110101010110101100101100110100101010101100001101$.}
	\label{fig:barcode}
\end{figure}%
The
one-dimensional random stepped interface (the \emph{barcode})
we see in the waterfall region
can be represented as a random sequence of square and vertical lozenges,
\begin{tikzpicture}[baseline = (current
bounding box.south),scale=.25]
\draw[thick] (0,0)--++(1,0)--++(0,1)--++(-1,0)--cycle;
\end{tikzpicture}
and
\begin{tikzpicture}[baseline = (current
bounding box.south),scale=.25]
\draw[thick] (0,0)--++(0,1)--++(1,1)--++(0,-1)--cycle;
\end{tikzpicture}.
Let us interpret each square lozenge
$[(t-1,x);(t,x)]$
as a particle at $t\in \mathbb{Z}$.
We may pick the vertical coordinate $x$ arbitrarily so that
$(t,x)$ asymptotically lies in the waterfall region~$\mathcal{W}$ (recall \Cref{def:waterfall_region}).
Indeed, this is possible
thanks to the clustering
of the nonintersecting
paths in $\mathcal{W}$ (which is equivalent to the absence of
the horizontal lozenges
$\begin{tikzpicture}[baseline = (current
bounding box.south),scale=.25]
\draw[thick] (0,0)--++(1,1)--++(1,0)--++(-1,-1)--cycle;
\end{tikzpicture}$
in $\mathcal{W}$), see \Cref{thm:no_holes_in_waterfall_region}.
Thus, we arrive at the (pre-limit) \emph{barcode point process} on $\mathbb{Z}$.

To access the (conjectural) barcode kernel in the
limit, throughout the rest of the paper, we fix a hexagon with the specific scaled side
lengths $T=8L$ and $N=S=4L$.\footnote{While we believe that
the limiting barcode kernel is the same for any hexagon, a
specific choice of side lengths simplifies many of the
formulas in the remainder of this \Cref{sec:2d_correlations}
and in the next \Cref{sec:2d_correlations_asymptotics_numerics_conjectures}. See
also examples in \Cref{subsub:density_function,subsub:s7_density_and_correlations} below.} This leads to the following
definition.

\begin{definition}[Pre-limit barcode kernel]
	\label{def:prelimit_barcode_kernel}
	Let $L$ be a positive integer.
	Consider the hexagon with parameters $T=8L$, $N=S=4L$.
	Denote $T_0\coloneqq2L$, and $X_0\coloneqq3L$. The point $(T_0,X_0)$
	is in the first zone of the hexagon
	\eqref{eq:hexagon_zone_1},
	and asymptotically lies
	on the
	center line $\mathsf{x}=\frac12(\mathsf{t}+\mathsf{S})$ of the waterfall region.
	Let $s,t\in \mathbb{Z}$ be such that $|s|,|t|<L$.
	We define the \emph{pre-limit barcode kernel} by
	\begin{equation}
		\label{eq:prelimit_barcode_kernel}
		K^{\mathrm{barcode}}_{(L)}(s,t)\coloneqq
		\mathrm{Kast}^{-1}\left( T_0+s,X_0;T_0+t-1,X_0 \right),
	\end{equation}
	where $\mathrm{Kast}^{-1}$ is the inverse Kasteleyn matrix
	for the $q$-Racah measure on lozenge tilings
	of our specific hexagon.
	In other words, the pre-limit barcode kernel
	describes the joint distribution of the
	square lozenges
	$\begin{tikzpicture}[baseline = (current
	bounding box.south),scale=.25]
	\draw[thick] (0,0)--++(1,0)--++(0,1)--++(-1,0)--cycle;
	\end{tikzpicture}$
	around the point $(T_0,X_0)$
	in the waterfall region.
\end{definition}

\begin{proposition}
	\label{prop:squares_joint_distributions}
	For the hexagon with $L$-dependent sides as in \Cref{def:prelimit_barcode_kernel},
	for any $m$ and any distinct $t_1,\ldots,t_m\in \mathbb{Z}$,
	$|t_i|<L$, we have
	\begin{equation}
		\label{eq:probability_of_square_lozenges_det}
		\operatorname{\mathbb{P}}\bigl\{
			\textnormal{there are square lozenges at
			$(t_i+T_0,X_0)$, $i=1,\ldots,m$}
		\bigr\}
		=
		\det
		\bigl[
			K^{\mathrm{barcode}}_{(L)}(t_i,t_j)
		\bigr]_{i,j=1}^m.
	\end{equation}
\end{proposition}
\begin{proof}
	This follows from
	the general formula \eqref{eq:probability_of_lozenges_det_via_inverse_kasteleyn}
	for the probability of finding a given set of lozenges
	in a random tiling. Note that the prefactor
	involving the Kasteleyn matrix is equal to $1$, since we are
	considering only square lozenges for which the edge weight is $1$, see \eqref{eq:Kasteleyn_matrix}.
\end{proof}

\subsection{Asymptotics of inter-slice coefficients and some prefactors}
\label{sub:2d_correlations_asymptotics_new}

In this subsection, we prove a number of lemmas on
the asymptotic behavior of various components entering the
pre-limit barcode kernel \eqref{eq:prelimit_barcode_kernel}.
Recall that for the rest of the paper we assume that the
hexagon has the specific side lengths
$T=8L$ and $N=S=4L$. Moreover, as in \eqref{eq:prelimit_barcode_kernel},
we set $t=T_0+t=2L+t$ and $x=X_0=3L$.

First, we need the following preliminary result.
We use the notation and properties of the Jacobi
theta function $\theta_q(z)$ from \Cref{app:a1_defn}.

\begin{lemma}
	\label{lemma:inf_qPochhammer_to_theta}
	Let $\alpha,\beta>0$ be fixed. Then
	\begin{equation*}
		\lim_{L\to\infty}
		\frac{(\alpha q^L;q)_\infty}{(\beta q^L;q)_\infty}=1,
	\end{equation*}
	and
	\begin{equation*}
		\lim_{L\to-\infty}
		\left( \frac{\beta}{\alpha} \right)^{L}
		\frac{(-\alpha q^{-L};q)_\infty}{(-\beta q^{-L};q)_\infty}=\frac{\theta_q(-\alpha)}{\theta_q(-\beta)}.
	\end{equation*}
\end{lemma}
\begin{proof}
	The first statement is evident, since both the numerator and the denominator
	rapidly approach $1$ as $L\to\infty$. The second statement follows from
	standard properties of the $q$-Pochhammer symbols,
	for example, see \cite[Exercise~1.1]{GasperRahman}.
\end{proof}

Let us denote
\begin{equation}
	\label{eq:mathfrak_a_definition}
	\mathfrak{a}(x)\coloneqq
	\biggl[
		\frac{-\kappa^2 q^{2x-1}}
		{ (1-\kappa ^2 q^{2x})(1-\kappa ^2 q^{2x-1})}
	\biggr]^{\frac12},\qquad x\in \tfrac12 \mathbb{Z}.
\end{equation}
Note that this expression rapidly decays as $x\to\pm\infty$.

\begin{lemma}
	\label{lemma:limit_of_operator_U_coefficients_a_0_a_1}
	The coefficients of the operator $\mathfrak{U}_t$ \eqref{eq:operator_U}
	have the following asymptotic behavior at coordinates $(T_0+t,X_0+x)$,
	where $t,x\in \mathbb{Z}$ are fixed:
	\begin{align*}
		\lim_{L\to\infty}q^{3L}\sqrt{\frac{\widetilde{w}_{T_0+t}(X_0+x-1)}{\widetilde{w}_{T_0+t+1}(X_0+x)}}\ssp
		a_1^{T_0+t}(X_0+x-1) &= q^{-\frac{t}{2}} \ssp \mathfrak{a}(x-\tfrac{t}{2});\\
		\lim_{L\to\infty}
		q^{3L}
		\sqrt{\frac{\widetilde{w}_{T_0+t}(X_0+x)}{\widetilde{w}_{T_0+t+1}(X_0+x)}}\ssp
		a_0^{T_0+t}(X_0+x)
		&=q^{-\frac{t}{2}} \ssp \mathfrak{a}(x-\tfrac{t}{2}+\tfrac12).
	\end{align*}
\end{lemma}
\begin{proof}
	We have
	\begin{align*}
	\frac{\widetilde{w}_{T_0+t}(X_0+x-1)}{\widetilde{w}_{T_0+t+1}(X_0+x)}&=
	\frac{q^{-11 L} (1-q^{3 L+x}) (1-\kappa^2 q^{3 L+x}) (q^{t+1}-\kappa^2q^{2x})
	(q^{1-5 L+x} \kappa^2;q)_{\infty}}{(q^x-q^{5 L}) (q^t-\kappa^2q^{2x})
	(q^{-5 L+x} \kappa^2;q)_{\infty}}
	;
	\\
	\frac{\widetilde{w}_{T_0+t}(X_0+x)}{\widetilde{w}_{T_0+t+1}(X_0+x)}&=
	\frac{
	(q^{-3L-t+x}-1)(q^t-q^{2x+1}\kappa^2)(q^{-3L-t+x}\kappa^2;q)_\infty
	}
	{
	(1-q^{5L-t+x})(q^t-q^{2x}\kappa^2)(1-q^{5L-t+x}\kappa^2)(q^{1-3L-t+x}\kappa^2;q)_\infty
	};
	\\
	a_1^{T_0+t}(X_0+x-1)
	&=
	\frac{(q^x-q^{5 L}) (q^{5 L}-\kappa^2q^x)}{q^{t+1}-\kappa^2q^{2x}}
	;
	\\
	a_0^{T_0+t}(X_0+x)&=
	\frac{(1-q^{5 L-t+x}) (1-\kappa^2 q^{5 L-t+x})}{1-\kappa^2 q^{1-t+2x}}.
	\end{align*}
	Using \Cref{lemma:inf_qPochhammer_to_theta}, we immediately
	get the desired limits.
\end{proof}

\begin{lemma}
	\label{lemma:G_ratio_asymptotics}
	The ratio of the prefactors $G$
	\eqref{eq:G_factor_conjugation}
	in the pre-limit barcode kernel
	has the following asymptotic behavior:
	\begin{equation}
		\label{eq:G_ratio_asymptotics}
		\begin{split}
			&\lim_{L\to\infty}
			q^{-3L(t-1-s)}\ssp\frac{G(T_0+s,X_0)}{G(T_0+t-1,X_0)
			\mathsf{w}_{q,\kappa}(X_0-\tfrac {T_0+s}2+1)}
			\\&
			\hspace{120pt}=
			-
			\mathbf{i}^{t-s}
			q^{\frac12\binom {t-1}2-\frac{1}{2}\binom {s-1}2+1-s}
			\sqrt{
			\frac{-\kappa^2}{(1-\kappa ^2 q^{2-t})(1-\kappa ^2 q^{1-s} )}
			}
			.
		\end{split}
	\end{equation}
\end{lemma}
\begin{proof}
	One can check that the ratio
	$\frac{G(T_0+s,X_0)}{G(T_0+t-1,X_0)\ssp\mathbf{i}^{t-1-s}}$
	is positive for all $s,t$. Its square is given by
	\begin{align*}
		&
		\left( \frac{G(T_0+s,X_0)}{G(T_0+t-1,X_0)\ssp\mathbf{i}^{t-1-s}} \right)^2
		=
		\frac{(-1)^{t+s}(q^s-\kappa ^2 q) \kappa ^{-2s+2t-2} q^{-3L(s-t+1)+t-s} }{\kappa ^2 q^2-q^t}
		\\&\hspace{150pt}\times
		\frac{(q^{5L-s+1};q)_{\infty } (q^{3L+t-1};q)_{\infty } (q^{-3L-s+1} \kappa ^2;q)_{\infty } (q^{5L-s+1} \kappa ^2;q)_{\infty }}{ (q^{3L+s};q)_{\infty } (q^{5L-t+2};q)_{\infty } (q^{-3L-t+2} \kappa ^2;q)_{\infty } (q^{5L-t+2} \kappa ^2;q)_{\infty }}.
	\end{align*}
	By \Cref{lemma:inf_qPochhammer_to_theta}, the expression in the second line behaves as
	$\sim \left( \frac{q^{1-s}\kappa^2}{q^{2-t}\kappa^2} \right)^{3L}\frac{\theta_q(q^{1-s}\kappa^2)}{\theta_q(q^{2-t}\kappa^2)}$.
	Using the quasi-periodicity of the theta function and simplifying, we arrive at
	the desired result.
\end{proof}

\begin{remark}
	When the expression \eqref{eq:G_ratio_asymptotics} appears inside the
	barcode kernel, certain factors on the right-hand side of \eqref{eq:G_ratio_asymptotics}
	(specifically $\mathbf{i}^{t-s} q^{\frac12\binom{t-1}{2}-\frac12\binom{s-1}{2}}$)
	can be removed through a conjugation
	(gauge transformation) of the correlation kernel.
	In addition, the factor $q^{-3L(t-s)}$ on the left-hand side of \eqref{eq:G_ratio_asymptotics}
	is likewise a conjugation factor applied before taking the limit.
	The latter factor is necessary for the convergence of the
	ratio of the $G$ prefactors, but in the full barcode kernel,
	it can be adjusted (by a conjugation) to
	ensure the convergence of the kernel (see
	\Cref{conj:conjectural_limit_of_full_barcode_kernel}
	below).
\end{remark}

\section{Barcode kernel. Conjectures and numerical evidence}
\label{sec:2d_correlations_asymptotics_numerics_conjectures}

\subsection{Obstacles posed by the limiting inter-slice operator}
\label{sub:s7_setup}

In \Cref{sec:2d_correlations} we expressed the \emph{pre-limit} barcode kernel
\(K^{\mathrm{barcode}}_{(L)}(s,t)\)
(encoding the one-dimen\-sional
random stepped interface that emerges in the
waterfall region)
in terms of the two-dimen\-sional correlation
kernel \(\widetilde K(T_{0}+s,X_{0};T_{0}+t-1,X_{0})\) of the \(q\)-Racah
nonintersecting path ensemble.
Recall that we focus on the hexagon with the specific side lengths
$T=8L$ and $N=S=4L$, and we set $T_0=2L$ and $X_0=3L$.
Thanks to
\Cref{prop:2d_kernel_operator_interpretation},
\(\widetilde K\) admits an operator representation via the family of linear
operators \(\mathfrak U_{T_{0}+t}\) given by \eqref{eq:operator_U}.
Recall that
\Cref{lemma:G_ratio_asymptotics} provides additional explicit scaling and normalization
for $\widetilde{K}$ to get to $K^{\mathrm{barcode}}_{(L)}$.

Along the
center line of the waterfall region, the coefficients of
\(\mathfrak U_{T_{0}+t}\) converge to expressions
involving the function \(\mathfrak{a}(x)\) \eqref{eq:mathfrak_a_definition}, see
\Cref{lemma:limit_of_operator_U_coefficients_a_0_a_1}.
This enables us to define the limiting linear operators:
\begin{equation}
	\label{eq:limiting_operator_U}
	\bigl(\mathfrak{U}^{\mathrm{barcode}}_t f\bigr)(x)
	\coloneqq
	\mathfrak{a}(x-\tfrac{t}{2})\ssp f(x-1)+\mathfrak{a}(x-\tfrac{t}{2}+\tfrac12)\ssp f(x),
	\qquad
	x\in\mathbb Z\ \text{or}\ \mathbb Z+\tfrac12,
\end{equation}
so that
\begin{equation}
	\label{eq:limiting_operator_U_convergence}
	q^{3L+\frac{t}{2}}\ssp \bigr(\mathfrak{U}_{T_0+t}f\bigr)(X_0+x)\sim
	\bigl(\mathfrak{U}^{\mathrm{barcode}}_t f\bigr)(x), \qquad L\to\infty.
\end{equation}
The functions $\mathcal{F}_n(x)$ defined by \eqref{eq:F_n_functions_app} in \Cref{appendix}
play the same role for the limiting operators
as the orthonormal $q$-Racah polynomials $f_n^t(x)$ do for the
pre-limit ones. Namely, define
\begin{equation}
	\label{eq:F_n_t_limiting_functions}
	\mathcal{F}_n^t(x)\coloneqq \mathcal{F}_n(x-\tfrac{t-1}{2}),
\end{equation}
then
by \Cref{prop:three_term_relation},
we have (compare with \eqref{eq:operator_U_action_on_f_n}):
\begin{equation}
	\label{eq:U_barcode_interslice_functions}
	\bigl(\mathfrak{U}_t^{\mathrm{barcode}}\mathcal{F}_n^t\bigr)(x)=-q^n \mathcal{F}_n^{t+1}(x),
	\qquad
	t,x\in \mathbb{Z},
	\quad
	n\in \tfrac{1}{2}\mathbb{Z}_{\ge0}.
\end{equation}

A natural idea would be to replace \(\mathfrak U_{T_{0}+t}\)
with \(\mathfrak U^{\mathrm{barcode}}_t\) in the operator
representation of \(\widetilde K\) and expect that the
resulting expression (rewritten in terms of the $\mathcal{F}^t_n(x)$'s)
yields the limit of
\(K^{\mathrm{barcode}}_{(L)}\).
Let us
explain why this
seemingly straightforward step in fact requires a
considerably more delicate analysis,
which will be carried out in another work.

\medskip
When \(q\to1\) while \(q^{L}\) remains bounded (the regime analyzed in
\cite[Section~8]{borodin-gr2009q}), the operators \(\mathfrak U_{T_{0}+t}\) converge to
a bounded two-diagonal operator with constant coefficients. In Fourier space, both the limiting operator
and its inverse act by multiplication by a function and its reciprocal, respectively.
This makes it possible to obtain the limiting local lattice kernel simply by
inverting the Fourier transform.
The kernel in the limit as $q\to 1$ is given by the same expressions as before the limit \Cref{prop:2d_kernel_operator_interpretation},
In this way one recovers the celebrated
\emph{incomplete Beta kernel}
\cite{okounkov2003correlation},
which describes correlations in the Gibbs pure phase (see \Cref{sub:traditional_scaling_regime}
for further discussion and references).
We see that for bounded lattice operators, the operator method (in particular, the
results from \cite{reed1972} quoted in \Cref{thm:funkan}) provides a rigorous derivation of the local limiting kernel.

In our case of fixed $q$, the situation is dramatically different.
While the two-diagonal operators
$\mathfrak{U}^{\mathrm{barcode}}_t$ are bounded, their coefficients decay exponentially at $\pm\infty$.
Therefore, their inverses are unbounded, and, for example, applying
$\bigl(\mathfrak{U}^{\mathrm{barcode}}_t\bigr)^{-1}$ in the variable $y$ to the limit of $K_s(x,y)$
(which is the delta function at $y=x$) does not yield a
function in $\ell^2$.

\medskip

Throughout the remainder of this section we outline the main
computations (which, in fact, reveal further obstacles to
the operator approach based on the functions
$\mathcal{F}^t_n$), derive the conjectural limiting density of the barcode process,
and compare it with both
the numerical pre-limit barcode kernel and with
probabilistic simulations of the $q$-Racah ensemble. We also
discuss a possible approach towards the full limiting barcode kernel.
In computations in \Cref{sub:s7_density_limit,sub:s7_general_limit}, we will
explicitly indicate nonrigorous steps.

\subsection{Nonrigorous limit of the barcode kernel. Density}
\label{sub:s7_density_limit}

We first consider the diagonal elements of the
pre-limit barcode kernel, which define the density function
\begin{equation}
	\label{eq:density_function_barcode_prelimit}
	\rho^{\mathrm{barcode}}_{(L)}(t)\coloneqq K^{\mathrm{barcode}}_{(L)}(t,t),\qquad t\in \mathbb{Z}.
\end{equation}
From \eqref{eq:prelimit_barcode_kernel} and
\Cref{lemma:G_ratio_asymptotics}, we have the asymptotic equivalence as \(L\to\infty\):
\begin{equation}
	\label{eq:barcode_density_prelimit_asymptotics}
	\rho^{\mathrm{barcode}}_{(L)}(t)
	\sim
	q^{-3L-\frac{t-1}{2}}
	\widetilde K(T_0+t,X_0;T_0+t-1,X_0)
	\ssp
	\underbrace{\sqrt{
	\frac{-\kappa^2 q^{1-t}}{(1-\kappa ^2 q^{2-t})(1-\kappa ^2 q^{1-t} )}
	}}_{\mathfrak{a}(1-\frac{t}{2})}.
\end{equation}
The ``time'' arguments in $\widetilde{K}$
are ordered as in the first case in
\eqref{eq:K_kernel_two_dimensional_operator_interpretation}.
Thus, we need to compute the limit of
\begin{equation}
	\label{eq:limit_of_barcode_density_tilde_K_asymptotics}
	q^{-3L-\frac{t-1}{2}}\ssp \bigl( \mathfrak{U}_{T_0+t-1}^{-1} \bigr)_y\ssp
	K_{T_0+t}(X_0,X_0),
\end{equation}
where we apply the inverse operator to the function
$y\mapsto K_{T_0+t}(X_0,X_0+y)$, and then set $y=0$.
Let us now
replace this function in $y$ by its limit $\mathbf{1}_{y=0}$ (that follows from \Cref{prop:limit_of_kernel}),
and also use the asymptotic equivalence \eqref{eq:limiting_operator_U_convergence} for the
inverse operator $\mathfrak{U}_{T_0+t-1}^{-1}$, which reads
\begin{equation}
\label{eq:asymptotic_equivalence_inverse_operator}
	q^{-3L-\frac{t}{2}}\ssp \bigl( \mathfrak{U}_{T_0+t}^{-1} f\bigr)(X_0+x)\sim
	\bigl( (\mathfrak{U}^{\mathrm{barcode}}_t )^{-1}f\bigr)(x),
	\qquad
	L\to\infty.
\end{equation}
Therefore, we have
\begin{equation}
	\label{eq:limit_of_barcode_density_tilde_K_asymptotics_2}
	\eqref{eq:limit_of_barcode_density_tilde_K_asymptotics}
	\sim
	\bigl(\mathfrak{U}^{\mathrm{barcode}}_t\bigr)^{-1}\ssp \mathbf{1}_{y=0}\Big\vert_{y=0},
\end{equation}
where the operator acts on the variable $y$.
Using the representation of the identity in \Cref{prop:completeness_of_F_n_functions_app},
we have
\begin{equation}
	\label{eq:density_nonrigorous_final}
	\rho^{\mathrm{barcode}}_{(L)}(t)\sim
	\mathfrak{a}(1-\tfrac{t}{2})\ssp
	\bigl(\mathfrak{U}^{\mathrm{barcode}}_t\bigr)^{-1}
	\sum_{n\in \frac12\mathbb{Z}_{\ge0}}
	\frac{\mathcal{F}_n^t(0)\mathcal{F}_n^t(y)}{\|\mathcal{F}_n\|^2_{\ell^2(\mathbb{Z})}}
	\bigg\vert_{y=0}
	=
	\mathfrak{a}(1-\tfrac{t}{2})
	\sum_{n\in \frac12\mathbb{Z}_{\ge0}}
	\frac{-q^{-n}
	\mathcal{F}_n^t(0)\mathcal{F}_n^{t-1}(0)
	}{\|\mathcal{F}_n\|^2_{\ell^2(\mathbb{Z})}}.
\end{equation}
\begin{remark}
	\label{rmk:nonrigorous_limit_of_barcode_density_step1}
	Passing to the limit separately in $K_{T_0+t}$ and $\mathfrak{U}_{T_0+t-1}^{-1}$
	in \eqref{eq:limit_of_barcode_density_tilde_K_asymptotics}
	is nonrigorous.
	Moreover,
	we did not rigorously define $(\mathfrak{U}^{\mathrm{barcode}}_t)^{-1}$, either,
	so the right-hand side of \eqref{eq:limit_of_barcode_density_tilde_K_asymptotics_2}
	does not formally make sense.
\end{remark}

Let us examine the
final formula for the density more closely.
First, numerically one can check that the series in $n$
in the right-hand side
of
\eqref{eq:density_nonrigorous_final}
\textbf{diverges} for any $t\in \mathbb{Z}$.
However, this series diverges only ``mildly'',
in the sense that the terms are approaching the same nonzero
constant as $n\to\infty$, but they have alternating signs.
Therefore, the series in
\eqref{eq:density_nonrigorous_final}
may be thought of as having ``\textbf{two different sums}'',
coming from two possible ways of pairing the neighboring terms.
We make a precise conjecture:
\begin{conjecture}
	\label{conj:barcode_density_limit}
	The following limits
	\begin{equation*}
	\rho_0(t)\coloneqq
		\lim_{M\to\infty}
		\mathfrak{a}(1-\tfrac{t}{2})
		\sum_{n=0}^{M+1/2}
		\frac{-q^{-n}
		\mathcal{F}_n^t(0)\mathcal{F}_n^{t-1}(0)
		}{\|\mathcal{F}_n\|^2_{\ell^2(\mathbb{Z})}}
		,\quad
		\rho_1(t)\coloneqq
		\lim_{M\to\infty}
		\mathfrak{a}(1-\tfrac{t}{2})
		\sum_{n=0}^{M}
		\frac{-q^{-n}
		\mathcal{F}_n^t(0)\mathcal{F}_n^{t-1}(0)
		}{\|\mathcal{F}_n\|^2_{\ell^2(\mathbb{Z})}}
	\end{equation*}
	exist for all $t\in \mathbb{Z}$,
	where $M\in \mathbb{Z}$, and the sums are over $n\in \frac12\mathbb{Z}_{\ge0}$.
	Moreover, they take only two different values, depending on the parity of $t$:
	\begin{equation*}
		\rho^{\mathrm{barcode}}_{\mathrm{even}}\coloneqq \rho_0(2k)=\rho_1(2k+1),\qquad
		\rho^{\mathrm{barcode}}_{\mathrm{odd}}\coloneqq \rho_0(2k+1)=\rho_1(2k+2),
	\end{equation*}
	and $\rho^{\mathrm{barcode}}_{\mathrm{even}}+\rho^{\mathrm{barcode}}_{\mathrm{odd}}=1$.

	The pre-limit barcode density has the following asymptotic behavior
	for all $t\in \mathbb{Z}$:
	\begin{equation*}
		\lim_{L\to\infty}
		\rho^{\mathrm{barcode}}_{(L)}(t)
		=
		\begin{cases}
			\rho^{\mathrm{barcode}}_{\mathrm{even}}, & \text{if } t \text{ is even};\\
			\rho^{\mathrm{barcode}}_{\mathrm{odd}}, & \text{if } t \text{ is odd}.
		\end{cases}
	\end{equation*}
\end{conjecture}

Let us make a few comments on this conjecture.
Recall that $\rho_{(L)}^{\mathrm{barcode}}(t)$ is the density of the
square lozenges
$\begin{tikzpicture}[baseline = (current bounding box.south),scale=.25]
\draw[thick] (0,0)--++(1,0)--++(0,1)--++(-1,0)--cycle;
\end{tikzpicture}$\ssp.

\begin{remark}
The conjectural limits $\rho_0(t)$ and $\rho_1(t)$
may be equivalently written as infinite
sums, for example,
\begin{equation*}
	\rho_0(t)=
	-\mathfrak{a}(1-\tfrac{t}{2})
	\sum_{k=0}^{\infty}
	\biggl(
		\frac{q^{-k}
		\mathcal{F}_k^t(0)\mathcal{F}_k^{t-1}(0)
		}{\|\mathcal{F}_k\|^2_{\ell^2(\mathbb{Z})}}
		+
		\frac{q^{-(k+1/2)}
		\mathcal{F}_{k+1/2}^t(0)\mathcal{F}_{k+1/2}^{t-1}(0)
		}{\|\mathcal{F}_{k+1/2}\|^2_{\ell^2(\mathbb{Z})}}
	\biggr)
	.
\end{equation*}
Here the sum is over integer $k$.
A similar series representation can be written down for
$\rho_1(t)$; in that case one first isolates the term
corresponding to $k=0$ and then groups together the terms
with indices $k-\tfrac12$ and $k$ for $k=1,2,\ldots$.
The convergence of these infinite series
is equivalent to the first part of \Cref{conj:barcode_density_limit}.
\end{remark}

\begin{remark}
	\label{rmk:periodicity}
	The conjectural limit of the barcode density
	becomes \emph{periodic} in $t$ with period two.
	This periodicity was not expected a priori, since the
	$q$-Racah measure on lozenge tilings of the hexagon
	did not have any such periodicity.
	A posteriori,
	this periodicity may be explained as follows.

The number of horizontal lozenges
	$\begin{tikzpicture}[baseline = (current
	bounding box.south),scale=.25]
	\draw[thick] (0,0)--++(1,1)--++(1,0)--++(-1,-1)--cycle;
	\end{tikzpicture}$
	in the $t$-th vertical slice of the hexagon
	is determined by the shape of the hexagon
	and grows linearly with $t$.
	Thus, increasing $t$ by one adds exactly one horizontal lozenge to that slice.
	In the scaling limit, with probability exponentially close to one, these horizontal lozenges
	concentrate in the two regions $\mathcal{W}^{\pm}$ lying outside the waterfall region
	(see \Cref{def:above_below_waterfall}).
	When $t$ is incremented by one, the newly added horizontal lozenge
	can be placed only in one of the two regions, either $\mathcal{W}^+$ or $\mathcal{W}^-$.
	This asymmetry
	introduces an imbalance in the limiting density.
	By contrast, when $t$ increases by an even amount, the additional horizontal lozenges
	can be evenly distributed between $\mathcal{W}^+$ and $\mathcal{W}^-$,
	which leads to the observed period-two translation invariance of the limiting density.
\end{remark}

\begin{remark}
	\label{rmk:density_one_half}
	The fact that $\rho^{\mathrm{barcode}}_{\mathrm{even}}+\rho^{\mathrm{barcode}}_{\mathrm{odd}}=1$
	means that on larger scales, the proportion of the
	square lozenges $\begin{tikzpicture}[baseline = (current bounding box.south),scale=.25]
	\draw[thick] (0,0)--++(1,0)--++(0,1)--++(-1,0)--cycle;
	\end{tikzpicture}$
	(and similarly the vertical lozenges
	\begin{tikzpicture}[baseline = (current
	bounding box.south),scale=.25]
	\draw[thick] (0,0)--++(0,1)--++(1,1)--++(0,-1)--cycle;
	\end{tikzpicture}\ssp)
	tends to $\frac12$.
	This is consistent with the fact that the
	waterfall region has slope $\frac12$,
	see \Cref{fig:waterfall_region} for illustrations.
\end{remark}

\subsection{Nonrigorous limit of the barcode kernel. General case}
\label{sub:s7_general_limit}

First, based on numerics (detailed in
\Cref{sub:s7_numerics} below), we are able to conjecture the right gauge of the
pre-limit barcode kernel which leads to a convergent expression.

\begin{conjecture}
	\label{conj:conjectural_limit_of_full_barcode_kernel}
	The following limit exists for all $t,s\in \mathbb{Z}$:
	\begin{equation}
		\label{eq:conjectural_limit_of_full_barcode_kernel}
		\lim_{L\to\infty}
		\mathbf{i}^{t-s} q^{2L(s-t)}\ssp
		K^{\mathrm{barcode}}_{(L)}(s,t)
		\eqcolon
		\mathcal{K}^{\mathrm{barcode}}(s,t).
	\end{equation}
	The limit kernel $\mathcal{K}^{\mathrm{barcode}}(s,t)$
	is symmetric in $s$ and $t$,
	and is $2\times 2$ block Toeplitz:
	\begin{equation*}
	\begin{pmatrix}
		\mathcal{K}^{\mathrm{barcode}}(s,t)& \mathcal{K}^{\mathrm{barcode}}(s,t+1)\\
		\mathcal{K}^{\mathrm{barcode}}(s+1,t) & \mathcal{K}^{\mathrm{barcode}}(s+1,t+1)
	\end{pmatrix}
	=
		\begin{pmatrix}
			\mathcal{K}^{\mathrm{barcode}}_{00}(s-t) & \mathcal{K}^{\mathrm{barcode}}_{01}(s-t)\\
			\mathcal{K}^{\mathrm{barcode}}_{10}(s-t) & \mathcal{K}^{\mathrm{barcode}}_{11}(s-t)
		\end{pmatrix},
		\qquad s,t\in 2\mathbb{Z}.
	\end{equation*}
\end{conjecture}
In particular, $\mathcal{K}^{\mathrm{barcode}}_{00}(0)=\rho^{\mathrm{barcode}}_{\mathrm{even}}$
and
$\mathcal{K}^{\mathrm{barcode}}_{11}(0)=\rho^{\mathrm{barcode}}_{\mathrm{odd}}$.

\medskip

We can express the limiting kernel
$\mathcal{K}^{\mathrm{barcode}}(s,t)$ for $s\ge t$
in terms of the orthogonal functions $\mathcal{F}_n^{t}(x)$
\eqref{eq:F_n_t_limiting_functions}.
Since the kernel is conjecturally symmetric, specifying it
for $s\ge t$ uniquely determines
$\mathcal{K}^{\mathrm{barcode}}(s,t)$ for all
$s,t\in\mathbb{Z}$.

To access $\mathcal{K}^{\mathrm{barcode}}(s,t)$ for $s\ge t$,
we use the operator representation of the
two-dimensional correlation kernel
$\widetilde{K}(T_0+s,X_0;T_0+t-1,X_0)$,
which involves the inverse operators
$\mathfrak{U}_{T_0+j}^{-1}$; see \Cref{prop:2d_kernel_operator_interpretation}.
Arguing as in \Cref{sub:s7_density_limit}, we obtain an expression
involving an infinite series of the form:
\begin{equation}
	\label{eq:badly_diverging_series}
	(-1)^{t-s}
	\sqrt{\frac{-\kappa^2q^{1-s}}{(1-\kappa^2 q^{2-t})(1-\kappa^2 q^{1-s})}}
	\sum_{n\in \frac12\mathbb{Z}_{\ge0}}
	\frac{\bigl(-q^{n}\bigr)^{t-1-s}\mathcal{F}_n^{s}(0)\mathcal{F}_n^{t-1}(0)}
	{\|\mathcal{F}_n\|^2_{\ell^2(\mathbb{Z})}}.
\end{equation}
As we discussed in \Cref{sub:s7_density_limit}, for $s=t$,
the series in \eqref{eq:badly_diverging_series}
diverges ``mildly'', as its tail behaves like $+c-c+c-c+\ldots $.
For $s\ge t+1$, however, the presence of the negative powers of $q^n$ (starting from $q^{-2n}$)
in the summands
make them rapidly growing, and so the series diverges ``worse'' than in the case $s=t$.
However, it turns out that if we regularize the partial sums of the series,
we get a well-defined limit which (numerically) coincides
with the limit in \Cref{conj:conjectural_limit_of_full_barcode_kernel}:

\begin{conjecture}
	\label{conj:barcode_limit_via_F}
	The following limit exists for all
	$s,t\in \mathbb{Z}$,
	$s\ge t$:
	\begin{equation}
		\label{eq:sum_regularization_of_badly_diverging_series}
			\lim_{M\to\infty}
			q^{(M+1)(s-t)}\times
			(-1)^{t-s}
			\sqrt{\frac{-\kappa^2q^{1-s}}{(1-\kappa^2 q^{2-t})(1-\kappa^2 q^{1-s})}}
			\sum_{n=0}^{M+1/2}
			\frac{\bigl(-q^{n}\bigr)^{t-1-s}\mathcal{F}_n^{s}(0)\mathcal{F}_n^{t-1}(0)}
			{\|\mathcal{F}_n\|^2_{\ell^2(\mathbb{Z})}},
	\end{equation}
	where $M\in \mathbb{Z}$, and the sum is over $n\in \frac12\mathbb{Z}_{\ge0}$.
	The limit \eqref{eq:sum_regularization_of_badly_diverging_series}
	coincides with $\mathcal{K}^{\mathrm{barcode}}(s,t)$.

	Moreover, if we take the sum until $M$ instead of $M+1/2$,
	and correspondingly replace the regularization factor $q^{(M+1)(s-t)}$
	by $q^{(M+\frac{1}{2})(s-t)}$, then the limit
	also exists, and coincides with
	the shifted kernel $\mathcal{K}^{\mathrm{barcode}}(s+1,t+1)$.
\end{conjecture}

\subsection{Numerics}
\label{sub:s7_numerics}

Here we present numerical evidence in support of our
\Cref{conj:barcode_density_limit,conj:conjectural_limit_of_full_barcode_kernel,conj:barcode_limit_via_F}
on the asymptotic behavior of the barcode kernel.
We use \texttt{Mathematica} for symbolic and precise numerical computations.
The code defining all the functions and kernels
is reproduced in \Cref{sec:Mathematica_code_appendix}.
Throughout the rest of this section, we denote by
$\mathcal{K}^{\mathrm{barcode}}_{(M)}(s,t)$
the renormalized partial sum of the series in
\eqref{eq:sum_regularization_of_badly_diverging_series}
up to $M+\frac{1}{2}$. We also set
\begin{equation}
	\widehat{K}_{(L)}^{\mathrm{barcode}}(s,t)
	\coloneqq
	\mathbf{i}^{t-s}q^{2L(s-t)}\ssp K^{\mathrm{barcode}}_{(L)}(s,t)
\end{equation}
for the renormalized (conjugated) pre-limit kernel.

\subsubsection{Rate of convergence}

The pre-limit kernel
$\widehat{K}^{\mathrm{barcode}}_{(L)}(s,t)$
and the
partial sums
$\mathcal{K}^{\mathrm{barcode}}_{(M)}(s,t)$
both converge at a geometric rate $q$ in $L$ and $M$, respectively, see
\Cref{tab:L_convergence}.

\begin{table}[ht]
  \centering
	\scriptsize
  \begin{tabular}{|c|c|c|c|c|}
    \hline
    $L$ & $q=2/3$ & ratios & $q=1/3$ & ratios \\
    \hline
    10 & 0.46557005 & 0.69146 & 0.43591993 & 0.3333278 \\
    11 & 0.46659216 & 0.68248 & 0.43592101 & 0.3333315 \\
    12 & 0.46729892 & 0.67691 & 0.43592137 & 0.3333327 \\
    13 & 0.46778127 & 0.67336 & 0.43592149 & 0.3333331 \\
    14 & 0.46810778 & 0.67107 & 0.43592153 & 0.3333332 \\
    15 & 0.46832764 & 0.66958 & 0.43592154 & 0.3333333 \\
    16 & 0.46847518 & 0.66860 & 0.43592155 & 0.3333333 \\
    17 & 0.46857397 & 0.66795 & 0.43592155 & 0.3333333 \\
    18 & 0.46864002 & 0.66752 & 0.43592155 & 0.3333333 \\
    19 & 0.46868414 &         & 0.43592155 &           \\
    20 & 0.46871359 &         & 0.43592155 &           \\
    \hline
  \end{tabular}
  \qquad
	\begin{tabular}{|c|c|c|c|c|}
    \hline
    $L$ & $q=2/3$ & ratios & $q=1/3$ & ratios \\
    \hline
    10 & 0.05154306 & 0.64841 & 0.01901038 & 0.3331699 \\
    11 & 0.05233498 & 0.65338 & 0.01901074 & 0.3332788 \\
    12 & 0.05284846 & 0.65728 & 0.01901086 & 0.3333152 \\
    13 & 0.05318396 & 0.66016 & 0.01901090 & 0.3333273 \\
    14 & 0.05340448 & 0.66222 & 0.01901091 & 0.3333313 \\
    15 & 0.05355006 & 0.66365 & 0.01901092 & 0.3333327 \\
    16 & 0.05364646 & 0.66463 & 0.01901092 & 0.3333331 \\
    17 & 0.05371044 & 0.66530 & 0.01901092 & 0.3333333 \\
    18 & 0.05375296 & 0.66575 & 0.01901092 & 0.3333333 \\
    19 & 0.05378125 &         & 0.01901092 &           \\
    20 & 0.05380008 &         & 0.01901092 &           \\
    \hline
  \end{tabular}
  \caption{Convergence of the pre-limit barcode kernel for various $q$ and for $\kappa=2\mathbf{i}$.
	Left:
	density
	$\widehat{K}^{\mathrm{barcode}}_{(L)}(0,0)$.
	Right:
	elements
	$\widehat{K}^{\mathrm{barcode}}_{(L)}(6,1)$ of the kernel.
	Here and in \Cref{tab:M_convergence},
	we take ratios of a sequence $a_k$ defined as $(a_{k+1}-a_{k+2})/(a_{k}-a_{k+1})$.
	We see that the ratios become close to $q$.}
  \label{tab:L_convergence}
\end{table}

\begin{table}[ht]
  \centering
	\scriptsize
  \begin{tabular}{|c|c|c|c|c|}
    \hline
    $M$ & $q=2/3$ & ratios & $q=1/3$ & ratios \\
    \hline
    12 & 0.46719254 & 0.66801 & 0.43592134 & 0.3333333 \\
    13 & 0.46771795 & 0.66757 & 0.43592148 & 0.3333333 \\
    14 & 0.46806893 & 0.66727 & 0.43592153 & 0.3333333 \\
    15 & 0.46830323 & 0.66707 & 0.43592154 & 0.3333333 \\
    16 & 0.46845957 & 0.66693 & 0.43592155 & 0.3333333 \\
    17 & 0.46856386 & 0.66684 & 0.43592155 & 0.3333333 \\
    18 & 0.46863342 & 0.66678 & 0.43592155 & 0.3333333 \\
    19 & 0.46867980 & 0.66675 & 0.43592155 & 0.3333333 \\
    20 & 0.46871072 & 0.66672 & 0.43592155 & 0.3333333 \\
		21 & 0.46873134 &         & 0.43592155 &           \\
		22 & 0.46874509 &         & 0.43592155 &           \\
    \hline
  \end{tabular}
  \qquad
	\begin{tabular}{|c|c|c|c|c|}
    \hline
    $M$ & $q=2/3$ & ratios & $q=1/3$ & ratios \\
    \hline
    12 & 0.05141315 & 0.64925 & 0.01901026 & 0.3333306 \\
    13 & 0.05224775 & 0.65456 & 0.01901070 & 0.3333324 \\
    14 & 0.05278962 & 0.65836 & 0.01901085 & 0.3333330 \\
    15 & 0.05314430 & 0.66102 & 0.01901089 & 0.3333332 \\
    16 & 0.05337781 & 0.66286 & 0.01901091 & 0.3333333 \\
    17 & 0.05353216 & 0.66411 & 0.01901092 & 0.3333333 \\
    18 & 0.05363448 & 0.66495 & 0.01901092 & 0.3333333 \\
    19 & 0.05370242 & 0.66552 & 0.01901092 & 0.3333333 \\
    20 & 0.05374761 & 0.66590 & 0.01901092 & 0.3333333 \\
		21 & 0.05377768 &         & 0.01901092 &           \\
		22 & 0.05379770 &         & 0.01901092 &           \\
    \hline
  \end{tabular}
  \caption{Convergence of the renormalized series
	$\mathcal{K}^{\mathrm{barcode}}_{(M)}(0,0)$
	(left)
	and
	$\mathcal{K}^{\mathrm{barcode}}_{(M)}(6,1)$
	(right) for various $q$ and for $\kappa=2\mathbf{i}$.}
  \label{tab:M_convergence}
\end{table}

\subsubsection{Density function}
\label{subsub:density_function}

The diagonal values (i.e., the density function) of the pre-limit barcode kernel, together with those of its conjectural limit, are collected in \Cref{tab:diag_kernel_values}.
We chose smaller $L$ in the former to show some dependence on $t$, while
also took larger $M$ in the latter to demonstrate that the limiting density depends on $t$
only through its parity.
In particular,
the values
$\mathcal{K}^{\mathrm{barcode}}_{(M)}(t,t)$
and
$\mathcal{K}^{\mathrm{barcode}}_{(M)}(t+2,t+2)$
in \Cref{tab:diag_kernel_values} match up to ten decimal places.
One can see that
\begin{equation*}
	\bigl|
	\mathcal{K}^{\mathrm{barcode}}_{(M)}(7,7)-\mathcal{K}^{\mathrm{barcode}}_{(M)}(-7,-7)
	\bigr|
	\approx 2.7\cdot 10^{-16}.
\end{equation*}

We can also evaluate the inverse
Kasteleyn matrix
at other macroscopic coordinates
$(T_0,X_0)\neq(2L,3L)$,
to confirm that the limiting barcode
kernel is independent of the global location, and
depends only on the parameters $(q,\kappa)$.
For example, for $q=\frac{1}{7}$, $\kappa=3\mathbf{i}$,
$L=16$, and the global locations
$(T_0,X_0)=(\frac{L}{2},\frac{9L}{4})$ (on the center line)
and
$(T_0,X_0)=(\frac{L}{2},3L)$ (off the center line),
we have
the values given in \Cref{tab:diag_kernel_values_other_locations}.

In both \Cref{tab:diag_kernel_values,tab:diag_kernel_values_other_locations},
we
set $q=\tfrac{1}{7}$, a relatively small value chosen to accelerate convergence and get the limiting barcode kernel with higher precision.

\begin{table}[ht]
	\centering
	\scriptsize
	\begin{tabular}{|c|c||c|c|}
	\hline
		$t$ & \rule{0pt}{10pt}$\widehat{K}^{\mathrm{barcode}}_{(L)}(t,t)$ & $t$ & $\mathcal{K}^{\mathrm{barcode}}_{(M)}(t,t)$\\
		\hline
		$-7$ & 0.5523372977 & $-7$ & 0.5611174103\\
		$-6$ & 0.4387439934 & $-6$ & 0.4388825896\\
		$-5$ & 0.5598645425 & $-5$ & 0.5611174103\\
		$-4$ & 0.4388625943 & $-4$ & 0.4388825896\\
		$-3$ & 0.5609384574 & $-3$ & 0.5611174103\\
		$-2$ & 0.4388797292 & $-2$ & 0.4388825896\\
		$-1$ & 0.5610918461 & $-1$ & 0.5611174103\\
		$0$  & 0.4388821809 & $0$  & 0.4388825896\\
		$1$  & 0.5611137582 & $1$  & 0.5611174103\\
		$2$  & 0.4388825313 & $2$  & 0.4388825896\\
		$3$  & 0.5611168883 & $3$  & 0.5611174103\\
		$4$  & 0.4388825813 & $4$  & 0.4388825896\\
		$5$  & 0.5611173341 & $5$  & 0.5611174103\\
		$6$  & 0.4388825883 & $6$  & 0.4388825896\\
		$7$  & 0.5611173877 & $7$  & 0.5611174103\\
		\hline
	\end{tabular}
	\caption{Diagonal values $\widehat{K}^{\mathrm{barcode}}_{(L)}(t,t)$ and $\mathcal{K}^{\mathrm{barcode}}_{(M)}(t,t)$ for $-7\le t\le 7$. Parameters: $L=6$, $M=20$, $q=\tfrac17$, $\kappa=3\mathbf{i}$. All numbers are truncated to $10$ decimal digits.}
	\label{tab:diag_kernel_values}
\end{table}

%
%

\begin{table}[ht]
	\centering
	\scriptsize
	\begin{tabular}{|c|c|c|}
		\hline
		$t$
		&
		\rule{0pt}{10pt}
		$\mathrm{Kast}^{-1}(\frac{L}{2}+t, \frac{9L}{4}; \frac{L}{2}+t-1, \frac{9L}{4})$
		&
		$\mathrm{Kast}^{-1}(\frac{L}{2}+t, 3L; \frac{L}{2}+t-1, 3L)$
		\\[3pt]
		\hline
		$-2$ & 0.438744002701398029726330344950846261091117 & 0.438744002701397996670004043029846077414593\\
		$-1$ & 0.559864542689336428941964543890893079446439 & 0.559864542689336428425425841013168294080068\\
		$0$  & 0.438862595645192998493904028142472069596241 & 0.438862595645192993780780403343843713850882\\
		$1$  & 0.560938457463989285851404171643209418839516 & 0.560938457463989285776297216171385214228610\\
		$2$  & 0.438879729391405022754197543436512517729097 & 0.438879729391405022081083055826078202044213\\
		\hline
	\end{tabular}
	\caption{Entries of the inverse Kasteleyn matrix at two different macroscopic locations,
	where $L=16$, $q=\tfrac17$, and $\kappa=3\mathbf{i}$. We see that the values
	for the same $T_0=L/2$ and different $X_0$ agree extremely closely, which reflects the
	exponential concentration established in \Cref{thm:no_holes_in_waterfall_region}.
	These values also
	agree with the ones in \Cref{tab:diag_kernel_values}.}
	\label{tab:diag_kernel_values_other_locations}
\end{table}

We plot the density at $(0,0)$, i.e.\ the quantity
$\rho_{\mathrm{even}}^{\mathrm{barcode}}$ from
\Cref{conj:barcode_density_limit}, as a function of the
parameters $q$ and $\kappa$.  The resulting surface,
displayed in \Cref{fig:density_rho_even_3d}, exhibits a
striking system of waves whose frequency increases as
$q\to1$ or as $\kappa/\mathbf{i}\to0$.  Several
one-dimensional cross-sections of this surface are shown in
\Cref{fig:density_rho_even_cross_sections}.

\begin{remark}
	\label{rmk:q_to_zero}
	\Cref{fig:density_rho_even_3d,fig:density_rho_even_cross_sections}
	suggest the existence of a well-defined limit
	of $\rho_{\mathrm{even}}^{\mathrm{barcode}}$
	as $q\to0$. In this limit,
	many of our formulas should simplify, while the
	waterfall phenomenon in lozenge tilings persists
	for arbitrarily small $q$. We do not pursue this limit here.
\end{remark}

\begin{figure}[htpb]
	\centering
	\includegraphics[width=.7\textwidth]{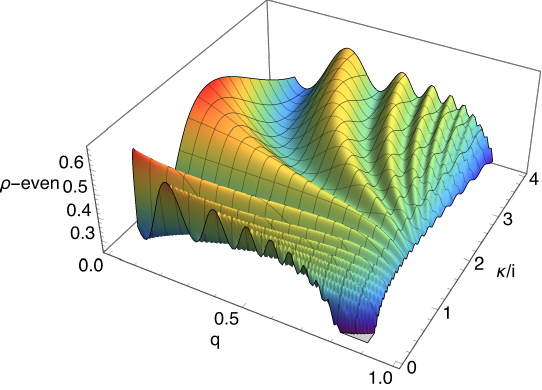}
	\caption{The even density function $\rho_{\mathrm{even}}^{\mathrm{barcode}}
	=\mathcal{K}^{\mathrm{barcode}}_{(M)}(0,0)$
	as a function of the parameters $q$ and $\kappa/\mathbf{i}$.
	Here we take $M=10$, $\frac{1}{10}\le q\le \frac{9}{10}$,
	$\frac{1}{13}\le \kappa/\mathbf{i}\le 4$,
	and the discretization in both parameters is $1/100$.
	Note that the theta functions in the denominator of \eqref{eq:F_n_functions_app}
	lead to singularities when $q$ and $\kappa/\mathbf{i}$ are mutual inverses.}
	\label{fig:density_rho_even_3d}
\end{figure}

\begin{figure}[htpb]
	\centering
	\begin{tabular}{cc}
		\makebox[.37\textwidth][c]{$\kappa=\tfrac{\mathbf{i}}{10}$}&
		\makebox[.37\textwidth][c]{$q=\tfrac12$}\\[2pt]
		\includegraphics[width=.37\textwidth]{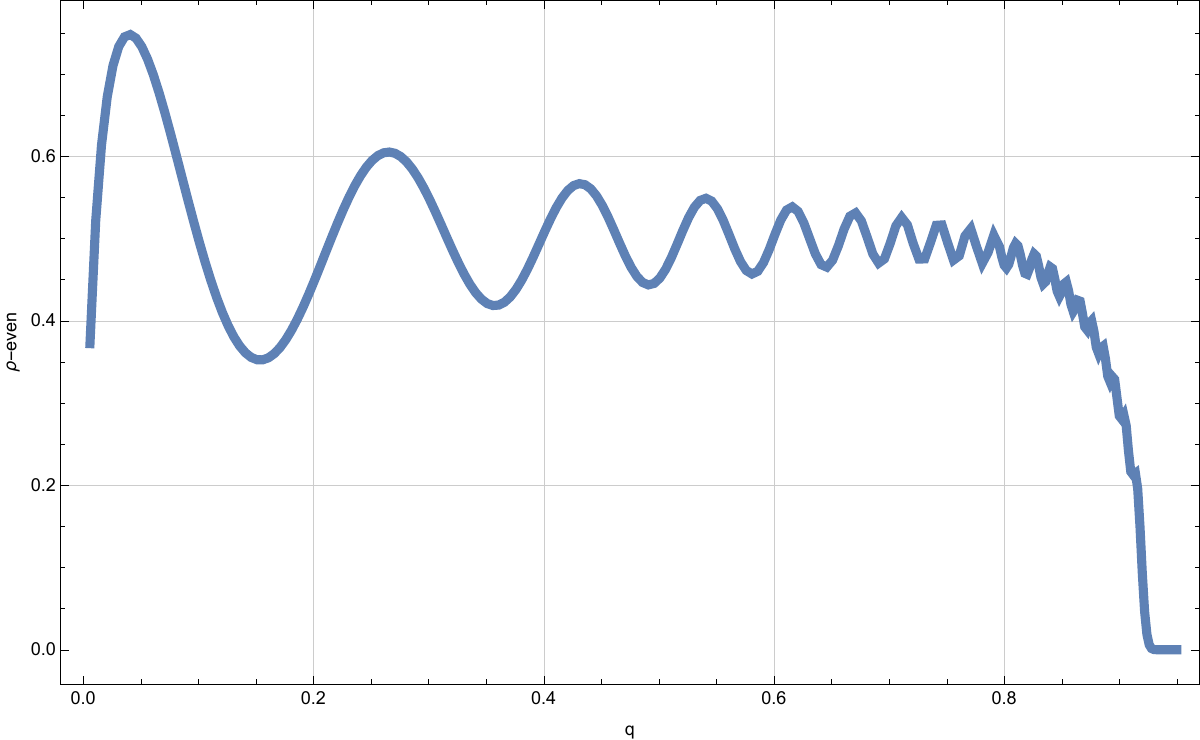}&
		\includegraphics[width=.37\textwidth]{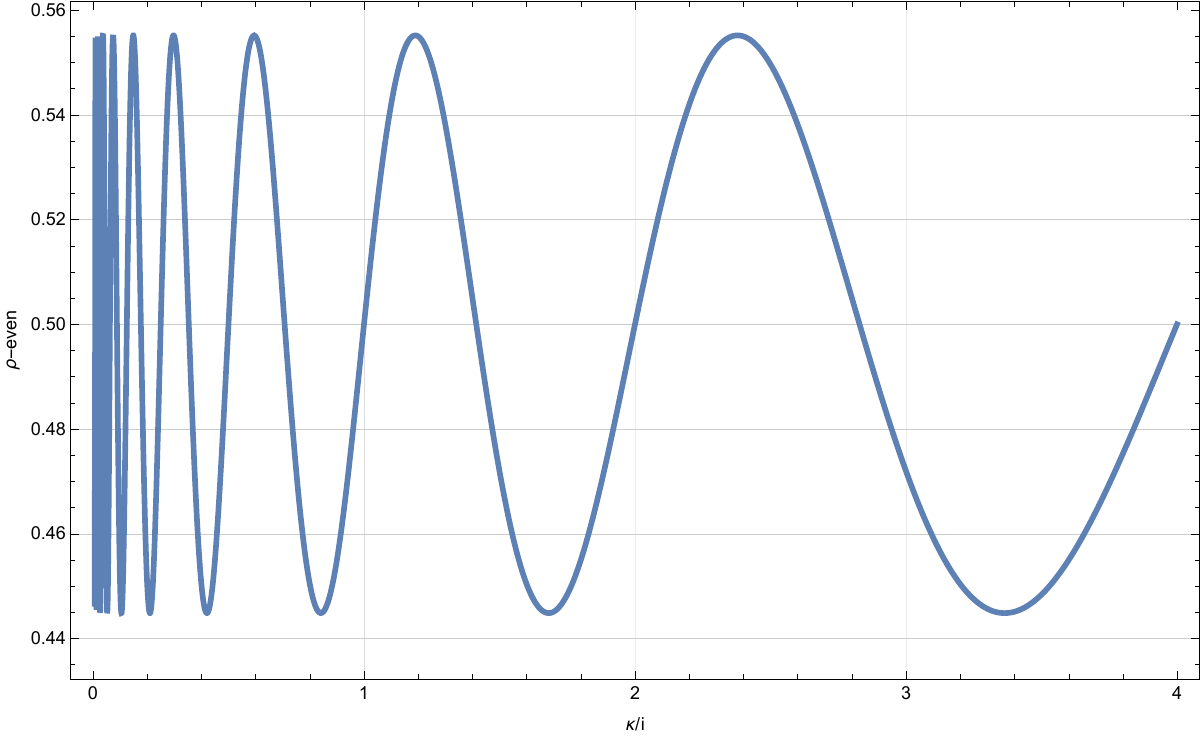}\\[20pt]
		\makebox[.37\textwidth][c]{$q=\tfrac1{1000}$}&
		\makebox[.37\textwidth][c]{$q=\tfrac{85}{100}$}\\[2pt]
		\includegraphics[width=.37\textwidth]{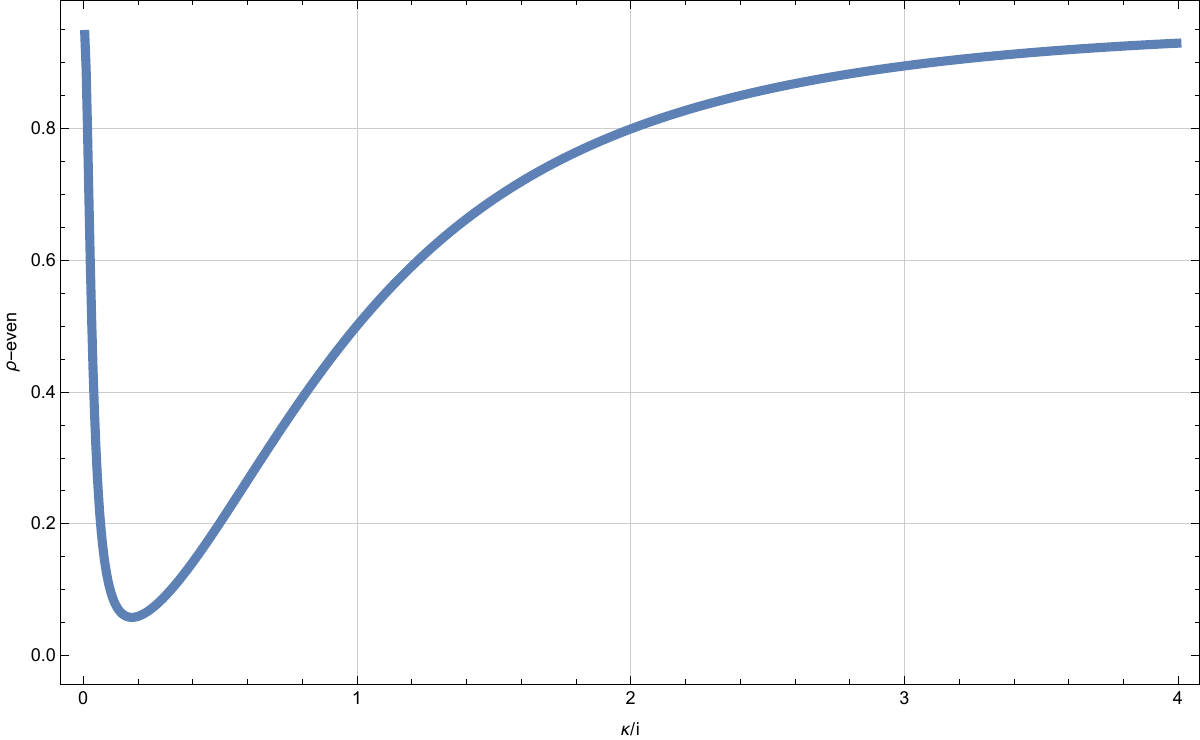}&
		\includegraphics[width=.37\textwidth]{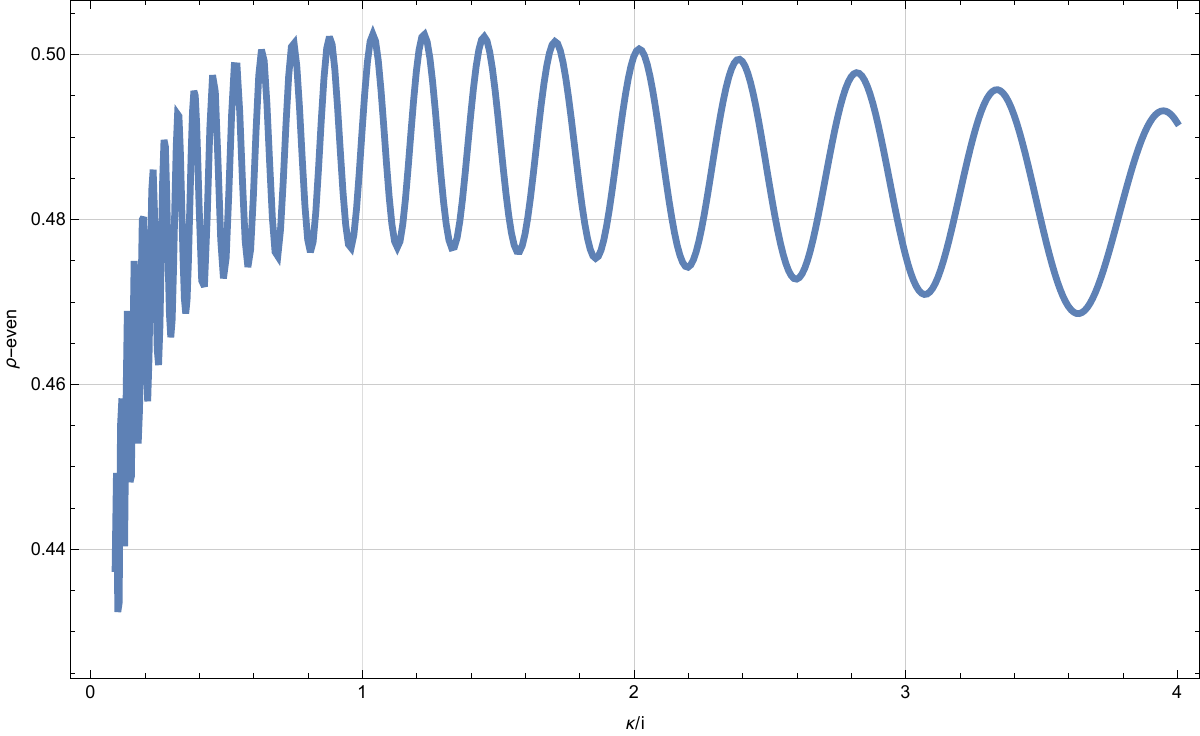}
	\end{tabular}
	\caption{Cross-sections of the surface $\rho_{\mathrm{even}}^{\mathrm{barcode}}$
	as one of the parameters $q$ or $\kappa/\mathbf{i}$ is fixed.}
	\label{fig:density_rho_even_cross_sections}
\end{figure}

\subsubsection{Correlation decay}

The covariance between the events
to find a particle in the barcode process
at $0$ and $t$ is given by
\begin{equation*}
	\operatorname{Cov}\left(
		\mathbf{1}_{\mathrm{barcode}}(0),
		\mathbf{1}_{\mathrm{barcode}}(t)
	\right)=
	-\mathcal{K}^{\mathrm{barcode}}_{(M)}(0,t)^2.
\end{equation*}
As for all determinantal point processes with a symmetric kernel,
the covariance is negative, which indicates repulsion.
The plot of the function
\begin{equation*}
	t\mapsto
	\frac{\log \bigl|\mathcal{K}^{\mathrm{barcode}}_{(M)}(0,t)\bigr|}
	{\log q}, \qquad t=1,2,\ldots,24,
\end{equation*}
for various values of $q$ is shown in \Cref{fig:correlations}.
This plot suggests a conjecture that correlations decay exponentially
in $t$.

\begin{figure}[htpb]
\centering
\includegraphics[width=.6\textwidth]{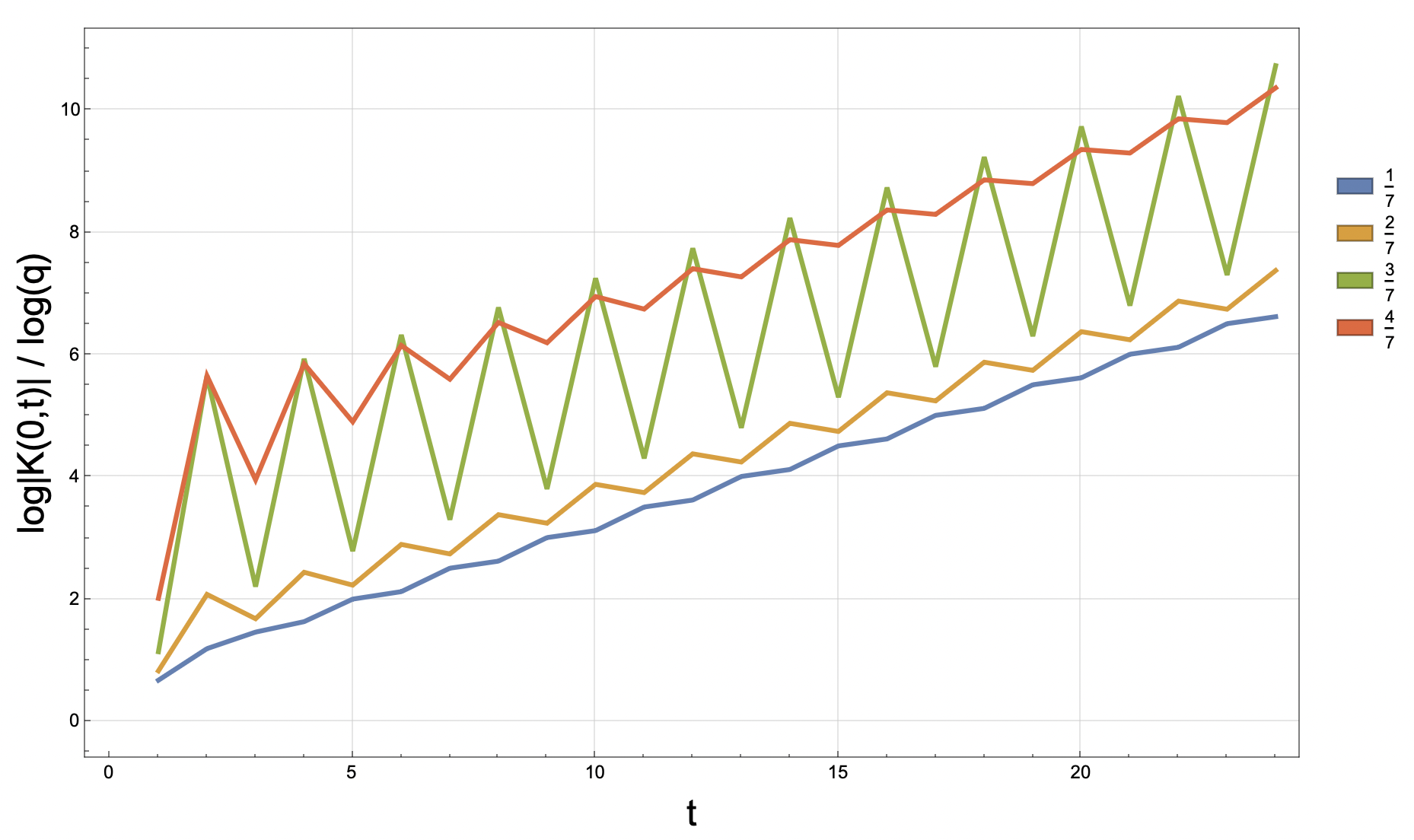}
\caption{Two-point correlations in the barcode process as a
function of $t$, for $q=\frac{1}{7},\frac{2}{7},\frac{3}{7}$, and~$\frac{4}{7}$.
The other parameters are $M=20$ and
$\kappa=\tfrac{3\mathbf{i}}{2}$.
The two-periodicity in $t$ is apparent, and
the approximately linear growth
corresponds to the hypothetical
exponential decay of correlations.}
\label{fig:correlations}
\end{figure}

The exponential decay of correlations stands in stark
contrast to the polynomial decay exhibited under ergodic
Gibbs measures on lozenge tilings with uniform resampling;
see \cite{KOS2006} and the further references in
\Cref{sub:traditional_scaling_regime} above. This numerical
observation distinguishes the barcode process from the sine
process (a cross-section of an ergodic Gibbs configuration),
whose correlations also decay polynomially. For the
continuous sine process this decay was already observed in
\cite{Dyson1962_part_3_of_statistical_theory_energy_levels}.

\subsection{Probabilistic simulations}
\label{sub:s7_probabilistic_simulations}

\subsubsection{Porting the perfect sampling algorithm}
\label{subsub:s7_porting_perfect_sampling}

Apart from numerics based on explicit formulas
(\Cref{sub:s7_numerics}), we can also test our
\Cref{conj:barcode_limit_via_F,conj:barcode_density_limit,conj:conjectural_limit_of_full_barcode_kernel}
with the help of the exact sampling algorithm for the $q$-Racah random lozenge tilings
\cite{borodin-gr2009q}.
The key observation underlying the algorithm is that one can
construct explicit Markov transition matrices on the space
of lozenge tilings; these transitions transport the
$q$-Racah measure for a hexagon with parameters $(T,S,N)$ to
the corresponding measures with parameters $(T,S+1,N)$ or
$(T,S-1,N)$.
For
nearly twenty years, the only available implementation of
this algorithm was Vadim Gorin's original \texttt{Compaq
Visual Fortran} code, whose source he kindly shared with us.
With
the assistance of modern AI-based coding tools we have
ported the relevant parts of this program to
\texttt{Python}; the resulting script is provided as
an ancillary file to the arXiv version of the paper.
We have
also created an interactive web-based simulator
\cite{petrov2025_3d_lozenge}; the latter samples the measure
$q^{\mathsf{vol}}$, i.e.\ the specialization $\kappa=0$,
which does not display the waterfall phenomenon.

The new \texttt{Python} implementation is particularly well
suited for batch generation of random lozenge tilings.
As in \cite{borodin-gr2009q}, the algorithm starts with
the unique tiling of the hexagon with
parameters $(T,0,N)$ (that is, a parallelogram), and then performs the sequence of
Markov moves that increment $S$ by one, thereby producing an
exact sample from the $q$-Racah distribution with parameters
$(T,S,N)$. For the final tiling, we also record
the barcode process, encoding it as a binary sequence.
See \Cref{fig:python_sim} for an example.\footnote{Examples in \Cref{fig:waterfall_region,fig:waterfall_region_frozen}
were generated by the original \texttt{Fortran} program.}
A key advantage of the script is that it can operate in a
``head-less'' mode: it need not render the tiling and thus can
produce large collections of barcode process samples in a short
time.

\begin{figure}[htpb]
\centering
\includegraphics[width=0.4\textwidth]{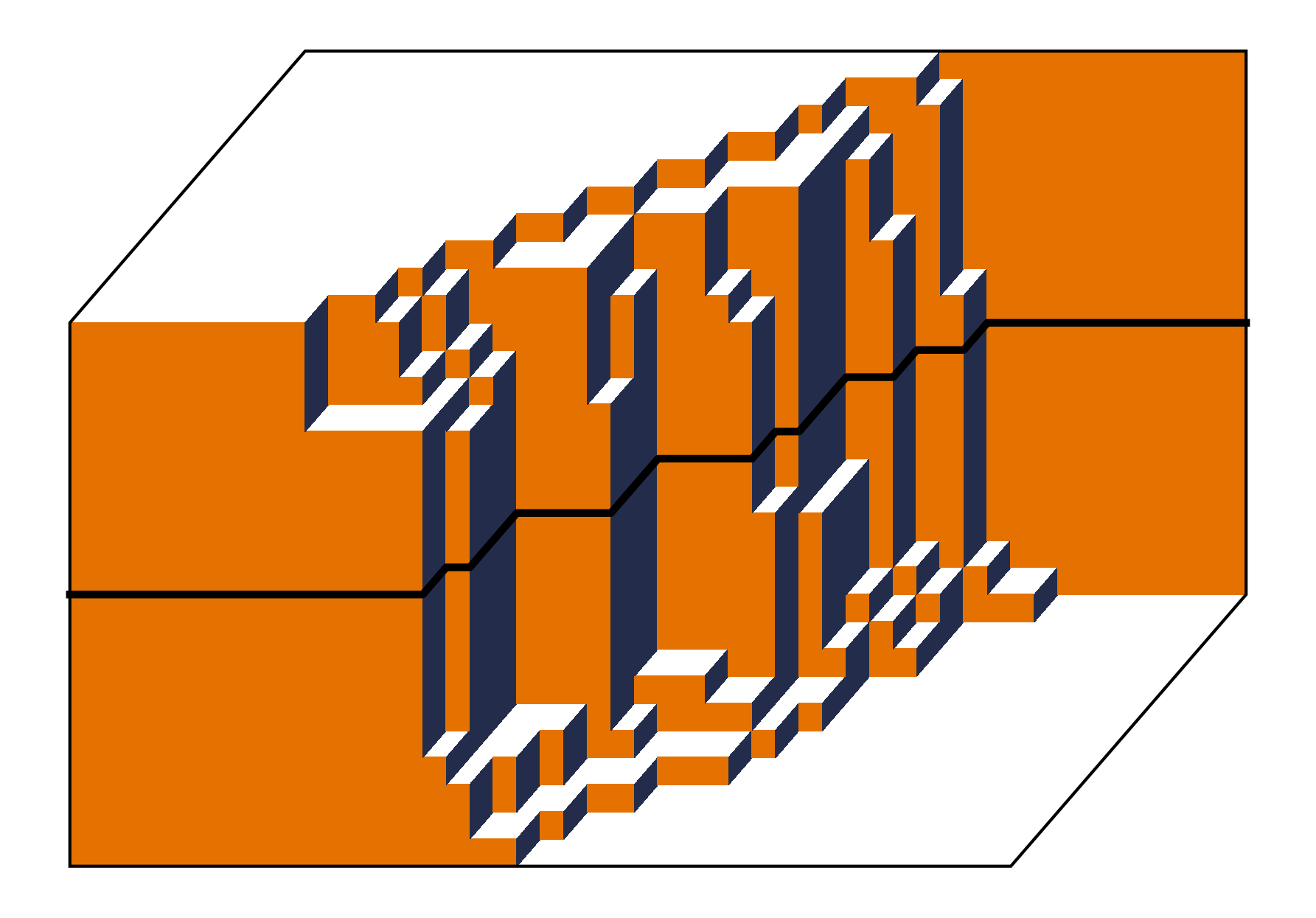}
\caption{Graphical output of the \texttt{Python} script. Here $T=50$, $S=10$, $N=20$, $q=\tfrac{7}{10}$, and $\kappa=3\mathbf{i}$.
The barcode process realization is
11111111111111101001111001111010011011011111111111.}
\label{fig:python_sim}
\end{figure}

In the rest of this subsection, we compare empirical pattern counts in the
barcode process with the predictions of
\Cref{conj:barcode_density_limit,conj:conjectural_limit_of_full_barcode_kernel},
and also demonstrate fluctuations of the height function of
the barcode process.

\subsubsection{Density and correlations}
\label{subsub:s7_density_and_correlations}

In \Cref{tab:simulation_density} we present empirical
density counts for various choices of $(T,S,N)$ (including the
base case $(8L,4L,4L)$, but also with other choices of the size),
with parameters
$q=\tfrac{1}{7}$ and $\kappa=4.3\mathbf{i}$. This
particular value of $\kappa$ is chosen since it induces a
comparatively strong bias between the densities on even and odd sites,
namely,
\begin{equation}
\label{eq:K_M_theoretical}
	\mathcal{K}_{(M)}^{\mathrm{barcode}}(0,0)=
	0.345174\ldots,
	\qquad
	\mathcal{K}_{(M)}^{\mathrm{barcode}}(1,1)=
	0.654826\ldots,
	\qquad M=20.
\end{equation}
For \Cref{tab:simulation_density}, we compute the frequency of ones on
even and odd sites, inside the waterfall region.
Indeed, note that when the hexagon is not symmetric (i.e.,
outside the base case $(8L,4L,4L)$; see \Cref{fig:python_sim} for an example),
the barcode process may start and end with a frozen string of zeroes or ones,
and we needed to cut these out.

\begin{table}[htbp]
	\centering
	\scriptsize
	\begin{tabular}{|ccc|c|c|c|c|c|}
		\hline
		$T$ & $S$ & $N$ &
		Even 1's count & Odd 1's count & Normalization & Even frequency & Odd frequency
		\\\hline
		400 & 200 & 200 &6,730&12,764& 19,500&0.345128&0.654564
		\\
		300 & 180& 140& 3,626 & 6,886 & 10,500 & 0.345333 & 0.65581
		\\
		400 & 20 & 20 & 6,68& 1,238& 1,900& 0.351579& 0.651579
		\\
		400 & 200 & 50 &  1,716 & 3,367 & 5,000 & 0.3432 & 0.6734
		\\\hline
	\end{tabular}
	\caption{Each row is based on 100 independently generated
	samples.  For example, in the first row we have 100
	realizations of a barcode sequence of length 400; removing
	five sites from each end to eliminate the boundary effects
	leaves 390 positions per sample.
	Thus the number of even (and, separately, odd) sites
	contributing to the empirical counts equals
	$390\times100/2=$19,500. The parameters are
	$q=\tfrac{1}{7}$ and $\kappa=4.3,\mathbf{i}$.
	In all cases, we observe a strong agreement of the sample
	densities with the theoretical ones \eqref{eq:K_M_theoretical}.}
	\label{tab:simulation_density}
\end{table}

Similarly, we can count the number of patterns of the form $11$ (length $2$) and $1\!\star\!1$ (length $3$) in the realizations of the barcode process (where $\star$ is either $0$ or $1$), and confirm the two-periodicity of these correlations, as well as the agreement with the theoretical predictions
\begin{equation*}
	\begin{split}
		&
		\det\begin{bmatrix}
			\mathcal{K}^{\mathrm{barcode}}_{(M)}(t,t) & \mathcal{K}^{\mathrm{barcode}}_{(M)}(t,t+1) \\
			\mathcal{K}^{\mathrm{barcode}}_{(M)}(t+1,t) & \mathcal{K}^{\mathrm{barcode}}_{(M)}(t+1,t+1)
		\end{bmatrix}=
		0.132142\ldots
		;\\
		&
		\hspace{40pt}
		\det\begin{bmatrix}
			\mathcal{K}^{\mathrm{barcode}}_{(M)}(t,t) & \mathcal{K}^{\mathrm{barcode}}_{(M)}(t,t+2) \\
			\mathcal{K}^{\mathrm{barcode}}_{(M)}(t+2,t) & \mathcal{K}^{\mathrm{barcode}}_{(M)}(t+2,t+2)
		\end{bmatrix}=
		\begin{cases}
			0.108658\ldots , & \text{if } t \text{ is even};\\
			0.418309\ldots , & \text{if } t \text{ is odd},
		\end{cases}
	\end{split}
\end{equation*}
for $q=\frac{1}{7}$, $\kappa=4.3\mathbf{i}$, and $M=20$.
Note that the first two-point correlation is fully translation-invariant.
For the data used for the first line in \Cref{tab:simulation_density}, we observe the frequency
$0.133933$ of the pattern~$11$ (over \emph{all} offsets, even or odd), and
$0.10653$
and
$0.413419$
for the patterns $1\!\star\!1$ with even and odd offsets, respectively.

\subsubsection{Height function of the barcode process}
\label{subsub:s7_height_function}

Finally, let us consider fluctuations of the height function
of the barcode process in the bulk of the waterfall region.
For a barcode process realization
$\vec b=(b_1,b_2,\ldots,b_T )$, $b_j\in \{0,1\}$,
we define the height function by
\begin{equation*}
	h(t)\coloneqq \sum\nolimits_{j=K}^t b_j,
\end{equation*}
where $K$ is a fixed offset chosen to reduce the effect of the left boundary
of the waterfall region.
Since the density of particles in the barcode process is
$\frac{1}{2}$, we consider the fluctuations
$h(t)-\frac{t}{2}$.
Plots
of fluctuations
are given in \Cref{fig:height_function}.
Based on them, we conjecture that the
fluctuations of the height function
grow slower than any power of $L$.

\begin{figure}[htb]
	\centering
	\begin{tabular}{m{.5\textwidth}m{.3\textwidth}}
		\includegraphics[width=\linewidth]{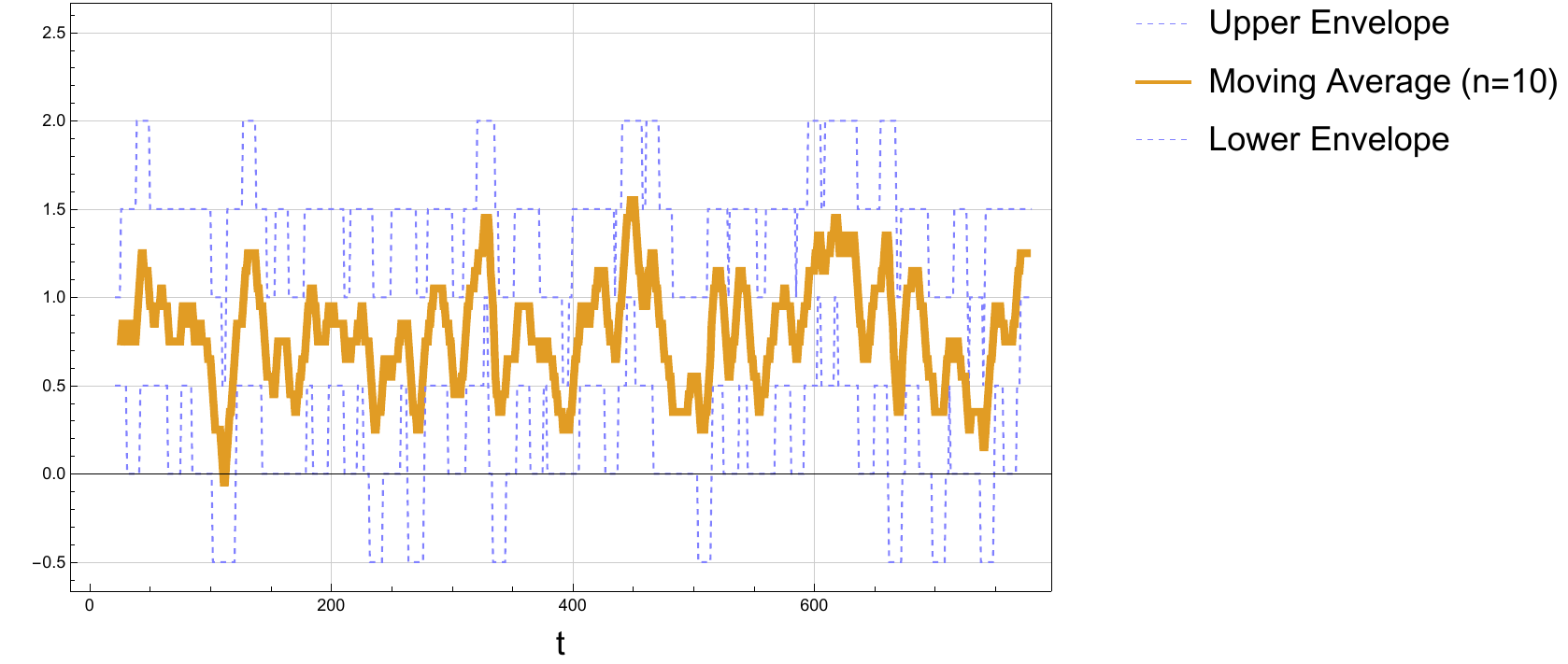} &
		\includegraphics[width=\linewidth]{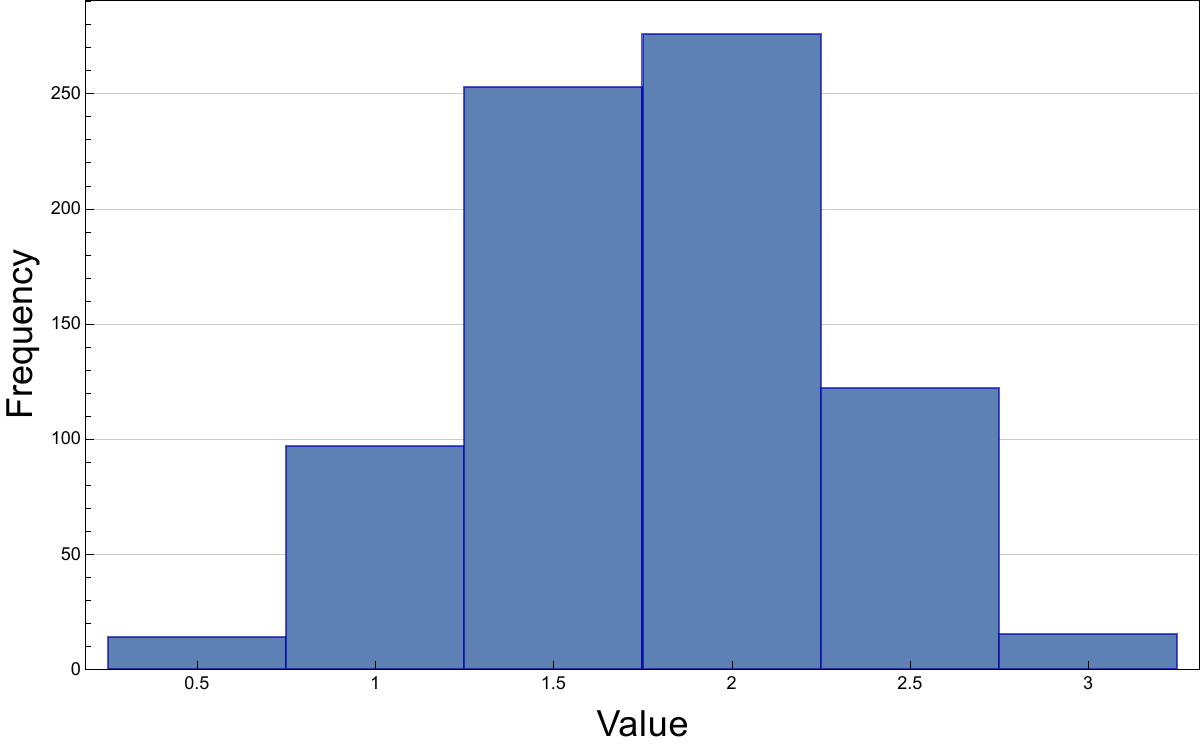}\\
		\includegraphics[width=\linewidth]{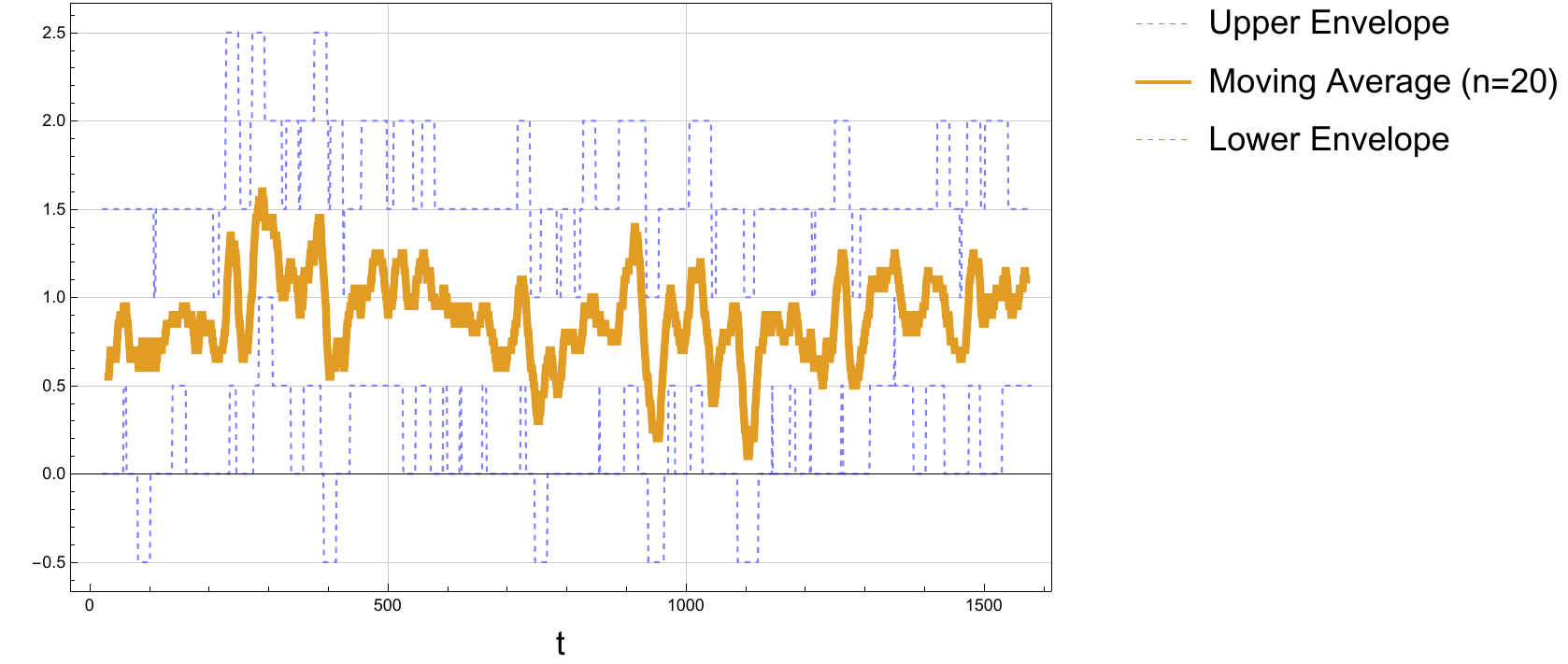} &
		\includegraphics[width=\linewidth]{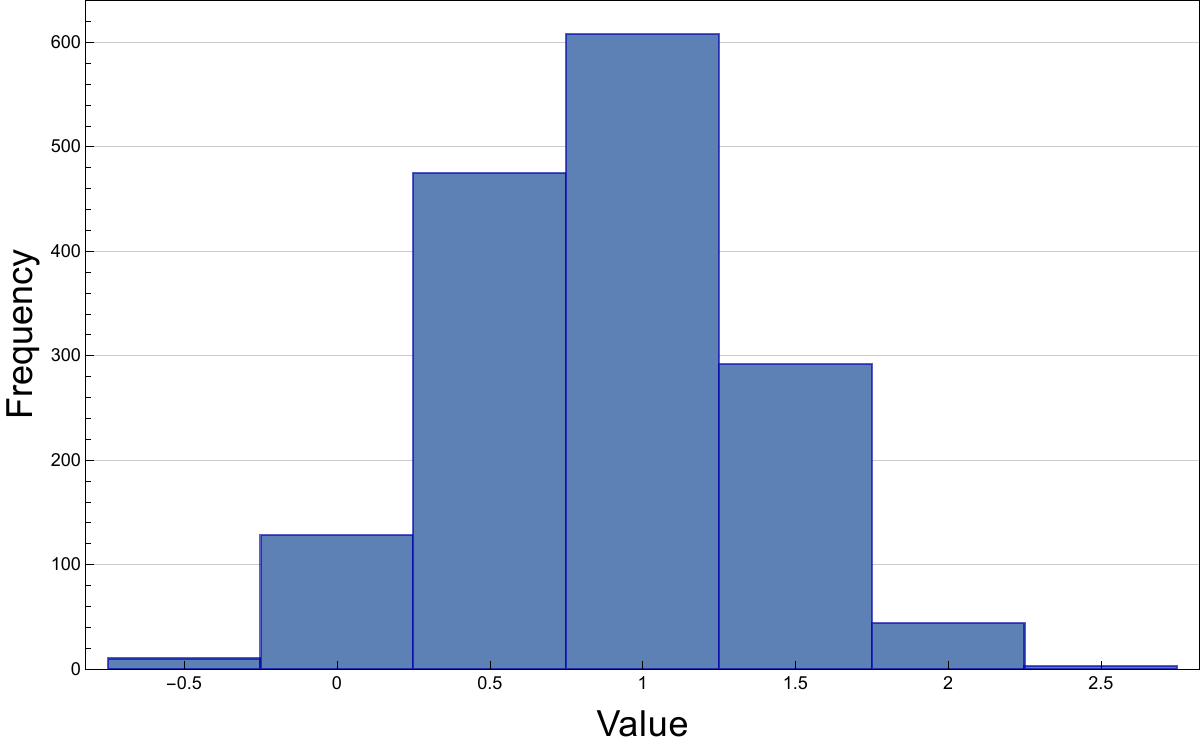}\\
	\end{tabular}
	\caption{Fluctuations $h(t)-\frac{t}{2}$ of the height function
	of the barcode process
	arising in the hexagon with parameters
	$(T,S,N)=(8L,4L,4L)$ for $L=100$ and $200$, where
	$q=\frac{4}{5}$, $\kappa=3\mathbf{i}$. The offset is $K=L/10$.
	On the left, the middle (bold) graph is the moving average, and there are also upper and lower
	envelopes (dashed lines). On the right, we show the histogram of the fluctuation values.}
	\label{fig:height_function}
\end{figure}

\subsection{Conclusion}
\label{subsec:7_conclusion}

The nonrigorous analysis carried out
in the present \Cref{sec:2d_correlations_asymptotics_numerics_conjectures}
paints the following conjectural picture.
There exists a family of determinantal point processes on $\mathbb{Z}$
with the correlation kernels
\[
  \mathcal{K}^{\mathrm{barcode}}(s,t\mid q,\kappa)
\]
depending on the parameters $q\in(0,1)$ and $\kappa\in\mathbf{i}\mathbb{R}_{>0}$.
These processes
governs the local statistics in the waterfall regime
of the $q$-Racah random lozenge tilings of the hexagon.
The limiting local process is universal --- it depends only on the parameters
$(q,\kappa)$, and not on the side lengths $(T,S,N)$ of the hexagon
(as long as the waterfall phase is present, i.e., $N\ll T$),
or the macroscopic point of observation around which we consider the local statistics.

This conjecture is supported by high-precision numerical computations
with the determinantal kernel of the $q$-Racah measure on lozenge tilings,
as well as by analyzing the results of the perfect sampling algorithm
for the measure.

\appendix

\section{Limiting orthogonal functions}
\label{appendix}

In this appendix, we establish several
properties of the orthogonal functions $\mathcal{F}_n(x)$
that enter \Cref{conj:barcode_density_limit,conj:barcode_limit_via_F} on the asymptotic
behavior of the barcode kernel.
This part of the paper is self-contained, its proofs are rigorous, and properties of
the functions $\mathcal{F}_n(x)$ may be of independent interest.

\subsection{Definition of the functions}
\label{app:a1_defn}

We need the Jacobi theta function, its quasi-periodicity, and the triple product identity:
\begin{align}
	\theta_q(z)&\coloneqq (z;q)_{\infty}(q/z;q)_{\infty},
	\nonumber
	\\
	\theta_q\left(q^m z\right)&=(-1)^m q^{-\frac{m(m-1)}{2}} z^{-m} \theta_q(z), \qquad  m \in \mathbb{Z},
	\label{eq:theta_function_and_defns}
	\\
	\nonumber
	\theta_q(z)&=\frac{1}{(q;q)_\infty}\sum_{m\in \mathbb{Z}} (-1)^m q^{\frac{m(m-1)}{2}} z^m.
\end{align}

Let $\frac12\mathbb{Z}$ be the set of all integers and half-integers.
We abbreviate $\mathbb{Z}'=\mathbb{Z}+\tfrac12$,
so $\tfrac12\mathbb{Z}=\mathbb{Z}\cup \mathbb{Z}'$.
If $i-j\in \mathbb{Z}$, we say that $i$ and $j$ have the same \emph{half-parity}.

Define
\begin{equation}
	\label{eq:F_n_functions_app}
	\mathcal{F}_n(x)\coloneqq
	\sqrt{1-\kappa^{-2}q^{-2x}}
	\sum_{i=-n}^{n}
	\frac{(-1)^{n+i} q^{-i/2}}{(q;q)_{n-i}(q;q)_{n+i}}
	\ssp
	\frac{1}{\theta_{\sqrt q}\left( |\kappa|q^{x+i+\frac12} \right)},
\end{equation}
where $0<q<1$ and $\kappa\in \mathbf{i}\mathbb{R}_{>0}$ are the
$q$-Racah parameters, $n,x\in \frac{1}{2}\mathbb{Z}$, and $n\ge0$.
By convention, in \eqref{eq:F_n_functions_app} and
throughout this appendix, the sum runs only over those $i$
that have the same half-parity as $n$. Any sum over all
half-integers will be indicated explicitly.

\begin{remark}
	\label{rmk:theta_function_hypergeometric_form_cont_hermite}
	The functions $\mathcal{F}_n$
	can be expressed in a $q$-hypergeometric form
	as
	\begin{equation}
		\label{eq:F_n_hypergeometric_form}
		\mathcal{F}_n(x)=
		\frac{(-1)^{2n+2x+1}q^{(n-x)^2+x/2}|\kappa|^{2x-2n+1}
		\sqrt{1-\kappa^{-2}q^{-2x}}}
		{(q;q)_{2n}\theta_{\sqrt q}(|\kappa|)}\ssp
		{}_2\phi_0\left(
			\begin{array}{c}
				q^{2n},0\\
				-
			\end{array}
			\middle|\ssp q^{-1};\kappa^{2}q^{2(x-n)}
		\right).
	\end{equation}
	The series ${}_2\phi_0$ is terminating, and can
	be identified with the
	continuous $q$-Hermite orthogonal polynomial
	$H_{2n}(\kappa q^x\mid q^{-1})$, where
	\cite[Chapter~3.26]{Koekoek1996}:
	\begin{equation}
		\label{eq:continuous_q_hermite}
		H_n(\lambda\mid q)= \lambda^{-n}
		{}_2\phi_0\left(
			\begin{array}{c}
				q^{-n},0\\
				-
			\end{array}
			\middle|\ssp q;\lambda^2 q^n
		\right).
	\end{equation}
	Since $0<q<1$ and
	the polynomials involve the inverse parameter $q^{-1}$, they
	are often referred to as the
	\emph{continuous $q^{-1}$-Hermite polynomials}.

	The occurrence of the orthogonal polynomial $H_{2n}$, with
	the index doubled from $n$ to $2n$, reflects the fact that
	the parameter $n$ in $\mathcal{F}_n$ may be a (proper) half-integer.
\end{remark}

In the remainder of this appendix, we establish the following properties of the functions $\mathcal{F}_n(x)$:
\begin{enumerate}[$\bullet$]
	\item
		(\Cref{prop:three_term_relation} in \Cref{sub:three_term_relation})
		Three-term relation in $x$;
	\item
		(\Cref{prop:orthogonality_F_n_functions_app} in \Cref{sub:orthogonality})
		Orthogonality of the functions $\mathcal{F}_n(x)$, $n\in \frac{1}{2}\mathbb{Z}_{\ge0}$,
		in each of the three spaces $\ell^2(\mathbb{Z})$,
		$\ell^2(\mathbb{Z}')$, and $\ell^2(\tfrac12\mathbb{Z})$;
	\item
		(\Cref{prop:norms_of_F_n_functions_app} in \Cref{sub:norms})
		Explicit expression for
		$\ell^2$-norms of $\mathcal{F}_n$;
	\item
		(\Cref{prop:completeness_of_F_n_functions_app} in \Cref{sub:representation_of_identity})
		Completeness of the
		functions in $\ell^2(\mathbb{Z})$ and $\ell^2(\mathbb{Z}')$,
		and the corresponding representation of the identity operator.
\end{enumerate}
All of these properties, except the last one, are proved
directly using the
definition~\eqref{eq:F_n_functions_app}
and the properties of the Jacobi theta function
\eqref{eq:theta_function_and_defns}.
To establish
completeness, however, we employ the $q$-hypergeometric
representation from
\Cref{rmk:theta_function_hypergeometric_form_cont_hermite},
together with the $q$-Mehler formula (an explicit Poisson kernel) for
the continuous $q^{-1}$-Hermite polynomials obtained in
\cite{IsmailMasson1994}.

\subsection{Three-term relation}
\label{sub:three_term_relation}

\begin{proposition}
	\label{prop:three_term_relation}
	We have
	\begin{equation}
		\label{eq:master_function_three_term_relation}
		\mathfrak{a}(x)\ssp
		\mathcal{F}_n(x-\tfrac12)
		+
		\mathfrak{a}(x+\tfrac12)\ssp
		\mathcal{F}_n(x+\tfrac12)
		=
		-q^n\ssp \mathcal{F}_n(x),
		\qquad n,x\in \tfrac12 \mathbb{Z},\quad n\ge0,
	\end{equation}
	where $\mathfrak{a}$
	is given by
	\begin{equation*}
		\mathfrak{a}(x)\coloneqq
		\biggl[
			\frac{-\kappa^2 q^{2x-1}}
			{ (1-\kappa ^2 q^{2x})(1-\kappa ^2 q^{2x-1})}
		\biggr]^{\frac12},\qquad x\in \tfrac12 \mathbb{Z}.
	\end{equation*}
\end{proposition}
\begin{proof}
	Using the quasi-periodicity \eqref{eq:theta_function_and_defns}
	and after the necessary simplifications, the three-term relation
	becomes
	\begin{multline*}
		\sum_{i=-n}^n
		\frac{(-1)^{n+i} q^{-i/2}}{(q;q)_{n-i}(q;q)_{n+i}}
		\Biggl(
			\frac{q^{i^2-\frac12i}|\kappa|^{2i}q^{2ix}}{\sqrt{1-\kappa^2 q^{2x}}}
			+
			\frac{q^{i^2+\frac32 i}|\kappa|^{2i+2}q^{(2i+2)x}}{\sqrt{1-\kappa^2 q^{2x}}}
		\Biggr)\frac{1}{\theta_{\sqrt q}(|\kappa|q^x)}
		\\=-q^n
		\sum_{i=-n}^{n}
		\sqrt{1-\kappa^{-2}q^{-2x}}
		\frac{(-1)^{n+i} q^{-i/2}}{(q;q)_{n-i}(q;q)_{n+i}}
		\ssp
		\frac{-q^{i^2+\frac12i}|\kappa|^{2i+1}q^{(2i+1)x}}{\theta_{\sqrt q}\left( |\kappa|q^{x} \right)}
	\end{multline*}
	(up to sign $(-1)^{2i}=\pm1$ which is the same on both sides).
	This identity reduces to
	\begin{equation}
		\label{eq:theta_function_eigenfunction_proof_1}
		\sum_{i=-n}^{n}
		\frac{q^{i^{2}}
			\left( z^iq^{-i}-z^{i+1}q^{i}-z^iq^{n}+z^{i+1}q^{n} \right)
		}
		{(q;q)_{n-i}(q;q)_{n+i}}=0,
	\end{equation}
	where
	$z=\kappa^2 q^{2x}$.
	Using the fact that $(q;q)_{n-i}(q;q)_{n+i}=q^{i^2+n^2+n}(q^{-1};q^{-1})_{n-i}(q^{-1};q^{-1})_{n+i}$,
	we can remove the factor
	$q^{i^2}$ from the numerator in \eqref{eq:theta_function_eigenfunction_proof_1}
	by changing $q$ to $\hat q\coloneqq 1/q$. Moreover, by symmetry, we can replace $i$ by $(-i)$
	in the first and the third summands in the numerator in~\eqref{eq:theta_function_eigenfunction_proof_1}.
	Thus, we need to show that
	\begin{equation}
		\label{eq:theta_function_eigenfunction_proof_2}
		\sum_{i=-n}^{n}
		\frac{
			z^{-i}\hat q^{-i}-z^{i+1}\hat q^{-i}-z^{-i}\hat q^{-n}+z^{i+1}\hat q^{-n}
		}
		{(\hat q;\hat q)_{n-i}(\hat q;\hat q)_{n+i}}=0.
	\end{equation}
	Factorizing
	$z^{-i}\hat q^{-i}-z^{i+1}\hat q^{-i}-z^{-i}\hat q^{-n}+z^{i+1}\hat q^{-n}=
	-(1-\hat q^{n-i})\ssp\hat q^{-n}(z^{-i}-z^{i+1})$, we see that
	\eqref{eq:theta_function_eigenfunction_proof_2} is equivalent to
	\begin{equation*}
		\sum_{i=-n}^{n-1}\frac{z^{-i}-z^{i+1}}{(\hat q;\hat q)_{n-i-1}(\hat q;\hat q)_{n+i}}=0.
	\end{equation*}
	The latter identity follows directly from the symmetry under $i\mapsto -\,i-1$.
	Consequently, we obtain the desired three-term relation \eqref{eq:master_function_three_term_relation}.
	Moreover, this relation is valid for all $x$, because the
	argument reduces to an identity of rational functions in
	the variable $z=\kappa^{2}q^{2x}$.
\end{proof}

\begin{remark}
	Alternatively, \Cref{prop:three_term_relation}
	can also be derived directly from the difference
	equation \cite[(3.26.5)]{Koekoek1996} satisfied by the
	continuous $q^{-1}$-Hermite orthogonal polynomials, which are
	related to $\mathcal{F}_n(x)$ (see
	Remark~\ref{rmk:theta_function_hypergeometric_form_cont_hermite}).
\end{remark}

\subsection{Orthogonality}
\label{sub:orthogonality}

\begin{proposition}
	\label{prop:orthogonality_F_n_functions_app}
	The functions $\mathcal{F}_n(x)$, $n\in \tfrac12 \mathbb{Z}_{\ge0}$,
	are orthogonal in the variable $x$
	in each of the three spaces
	$\ell^2(\mathbb{Z})$,
	$\ell^2(\mathbb{Z}')$,
	and
	$\ell^2(\tfrac12\mathbb{Z})$.
\end{proposition}
\begin{proof}
	Once we have the three-term relation
	(\Cref{prop:three_term_relation}),
	the orthogonality
	of $\mathcal{F}_n(x)$
	in $\ell^2(\tfrac12\mathbb{Z})$
	(as functions of $x$) readily follows.
	Indeed,
	the left-hand side
	of \eqref{eq:master_function_three_term_relation}
	defines a symmetric operator
	\begin{equation}
		\label{eq:T_operator_no_name}
		f(x)\mapsto \mathfrak{a}(x)f(x-\tfrac12)+\mathfrak{a}(x+\tfrac12)f(x+\tfrac12)
	\end{equation}
	in $\ell^2(\tfrac12\mathbb{Z})$.
	This operator is bounded because its coefficients $\mathfrak{a}(\cdot)$ rapidly decay at $\pm\infty$.
	Therefore, it is self-adjoint in $\ell^2(\tfrac12\mathbb{Z})$.

	The functions $\mathcal{F}_n$ are eigenfunctions of \eqref{eq:T_operator_no_name}
	with distinct eigenvalues $-q^n$, $n\in \tfrac12\mathbb{Z}_{\ge0}$,
	and hence they are orthogonal in $\ell^2(\tfrac12\mathbb{Z})$.

	Next, the square of the operator \eqref{eq:T_operator_no_name}
	is also bounded and self-adjoint, but preserves each of the spaces
	$\ell^2(\mathbb{Z})$ and $\ell^2(\mathbb{Z}')$. The functions
	$\mathcal{F}_n$ are eigenfunctions of the square of \eqref{eq:T_operator_no_name}
	with distinct eigenvalues $q^{2n}$, and hence they are orthogonal in $\ell^2(\mathbb{Z})$ and $\ell^2(\mathbb{Z}')$.
	This completes the proof.
\end{proof}

\subsection{Norms}
\label{sub:norms}

\begin{proposition}
\label{prop:norms_of_F_n_functions_app}
	For all $n\in \tfrac12 \mathbb{Z}_{\ge0}$, we have
	\begin{equation}
		\label{eq:norms_of_F_n}
		\|\mathcal{F}_n\|_{\ell^2(\mathbb{Z})}^2=
		\|\mathcal{F}_n\|_{\ell^2(\mathbb{Z'})}^2=
		\frac12\|\mathcal{F}_n\|_{\ell^2(\!{\scriptscriptstyle\frac12}\mathbb{Z})}^2=
		\frac{
		(q;q)_\infty(-\kappa^2)\ssp\theta_q(\kappa^{-2})}
		{q^n(q;q)_{2n}\ssp\theta_{\sqrt{q}}(|\kappa|)^2}.
	\end{equation}
\end{proposition}
\begin{proof}
	The functions $\mathcal{F}_n(x)$ are real-valued (as $n$ and $i$ are of the same half-parity),
	so we can sum their squares to compute the norms. We have
	\begin{equation}
		\label{eq:theta_function_eigenfunction_proof_3}
		\begin{split}
			(\mathcal{F}_n(x))^2&=
			\frac{1}{\theta_{\sqrt{q}}(|\kappa|)^2}\ssp
			(1-\kappa^{-2}q^{-2x})\sum_{i,j}
			\frac{(-1)^{2n+i+j}q^{-(i+j)/2}}
			{
			(q;q)_{n-i} (q;q)_{n+i} (q;q)_{n-j} (q;q)_{n+j}}
			\\&\hspace{160pt}\times
			|\kappa|^{4x+2+2i+2j}
			q^{(i+x)(i+x+1/2)+
			(j+x)(j+x+1/2)}.
		\end{split}
	\end{equation}
	Here the sum is over $-n\le i,j\le n$ of the same half-parity as $n$.
	Next, let us sum the $x$-dependent part
	of \eqref{eq:theta_function_eigenfunction_proof_3}
	over $x\in\mathbb{Z}$. We have, using the triple product identity \eqref{eq:theta_function_and_defns}:
	\begin{equation}
		\label{eq:theta_4_for_summing_over_Z}
		\sum_{x\in \mathbb{Z}}(1-\kappa^{-2}q^{-2x})(-\kappa^2)^{2x}q^{2x^2+(2i+2j+1)x}=
		(q^4;q^4)_{\infty}
		\Bigl[
			\theta_{q^4}( -\kappa^4 q^{2i+2j+3})-
			\kappa^{-2}
			\theta_{q^4}( -\kappa^4 q^{2i+2j+1})
		\Bigr].
	\end{equation}
	We can simplify this expression depending on the parity of $i+j$.
	If $i+j$ is even, we have
	\begin{equation*}
		\eqref{eq:theta_4_for_summing_over_Z}=
		(q^4;q^4)_{\infty}
		(-\kappa^{-2})^{i+j}
		\bigl[
			\theta_{q^4}(-\kappa^4 q^3)
			q^{-(i+j)(i+j+1)/2}
			-\kappa^{-2}
			\theta_{q^4}(-\kappa^4 q)
			q^{-(i+j)(i+j-1)/2}
		\bigr].
	\end{equation*}
	If $i+j$ is odd, we have
	\begin{equation*}
		\eqref{eq:theta_4_for_summing_over_Z}=
		(q^4;q^4)_{\infty}
		(-\kappa^{-2})^{i+j-1}
		\bigl[
			\kappa^{-4}\theta_{q^4}(-\kappa^4 q)
			q^{-(i+j)(i+j+1)/2}
			-\kappa^{-2}
			\theta_{q^4}(-\kappa^4q^{3} )
			q^{-(i+j)(i+j-1)/2}
		\bigr].
	\end{equation*}
	Multiplying this by the remaining terms in \eqref{eq:theta_function_eigenfunction_proof_3},
	we have for $i+j$ even:
	\begin{equation*}
		\begin{split}
			&
			-(-1)^{2n+i+j}\kappa^2 q^{(i-j)^2/2} (q^4;q^4)_{\infty}
			\bigl[
				\theta_{q^4}(-\kappa^4 q^3)q^{-(i+j)/2}
				-\kappa^{-2}
				\theta_{q^4}(-\kappa^4 q)q^{(i+j)/2}
			\bigr]
			\\&\hspace{50pt}
			=(-1)^{2n+i+j}
			q^{(i-j)^2/2} (q^4;q^4)_{\infty}
			\bigl[
				\theta_{q^4}(-\kappa^4 q)q^{(i+j)/2}
				-\kappa^{2}\theta_{q^4}(-\kappa^4 q^3)q^{-(i+j)/2}
			\bigr],
		\end{split}
	\end{equation*}
	and for $i+j$ odd:
	\begin{equation*}
		(-1)^{2n+i+j}
		q^{(i-j)^2/2}
		(q^4;q^4)_{\infty}
		\bigl[
			\theta_{q^4}(-\kappa^4 q)q^{-(i+j)/2}
			-
			\kappa^{2}
			\theta_{q^4}(-\kappa^4 q^3)q^{(i+j)/2}
		\bigr].
	\end{equation*}
	Utilizing the symmetry of the summation intervals for $i,j$, we see that
	the power $q^{-(i+j)/2}$ can always be replaced by $q^{(i+j)/2}$.
	This eliminates the dependence on the parity of $i+j$,
	and thus we have:
	\begin{equation*}
		\sum_{x\in \mathbb{Z}}(\mathcal{F}_n(x))^2=
		\frac{(q^4;q^4)_{\infty}\left(
				\theta_{q^4}(-\kappa^4 q)
				-\kappa^{2}\theta_{q^4}(-\kappa^4 q^3)
		\right)}
		{\theta_{\sqrt{q}}(|\kappa|)^2}
		\sum_{i,j}
		\frac{(-1)^{2n+i+j}q^{(i-j)^2/2+(i+j)/2}}
		{(q;q)_{n-i} (q;q)_{n+i} (q;q)_{n-j} (q;q)_{n+j}}.
	\end{equation*}
	By \Cref{lemma:theta_identity} below,
	the prefactor is equal to
	$\frac{
	(-\kappa^2)(q;q)_\infty\ssp \theta_q(\kappa^{-2})}{\theta_{\sqrt q}(|\kappa|)^2}$.
	The remaining sum over $i,j$ is simplified
	to $q^{-n}(q;q)_{2n}^{-1}$
	in \Cref{lemma:a_binomial_identity} below.
	Thus, we arrive at the desired norm in
	$\ell^2(\mathbb{Z})$.

	Let us obtain the norm in $\ell^2(\mathbb{Z}')$.
	The sum in \eqref{eq:theta_4_for_summing_over_Z}
	over $\mathbb{Z}'$
	is equal to
	\begin{equation*}
		\begin{split}
			&
			\sum_{x\in \mathbb{Z}'}
			(1-\kappa^{-2}q^{-2x})(-\kappa^2)^{2x}q^{2x^2+(2i+2j+1)x}=
			\sum_{y\in \mathbb{Z}}
			(1-\kappa^{-2}q^{-2y-1})(-\kappa^2)^{2y+1}q^{2y^2+(2i+2j+3)y+i+j+1}
			\\&\hspace{120pt}=
			-\kappa^2 q^{i+j+1}
			(q^4;q^4)_{\infty}
			\bigl[
				\theta_{q^4}(-\kappa^4 q^{2i+2j+3})
				-
				\kappa^{-2}q^{-1}\theta_{q^4}(-\kappa^4 q^{2i+2j+1})
			\bigr].
		\end{split}
	\end{equation*}
	Summing this expression (together with the remaining factors
	from \eqref{eq:theta_function_eigenfunction_proof_3})
	over $i,j$ is analogous to the case of $\ell^2(\mathbb{Z})$,
	and we omit the details.
	In the end, one can check that the norms in
	$\ell^2(\mathbb{Z})$ and $\ell^2(\mathbb{Z}')$ are equal to each other.
	This completes the proof of \Cref{prop:norms_of_F_n_functions_app}
	modulo the two lemmas below.
\end{proof}

\begin{lemma}
	\label{lemma:theta_identity}
	We have
	\begin{equation*}
		(q^4;q^4)_{\infty}\left(
			\theta_{q^4}(-\kappa^4 q)
			-\kappa^{2}\theta_{q^4}(-\kappa^4 q^3)
			\right)=
			(-\kappa^2)(q;q)_\infty\ssp \theta_q(\kappa^{-2}).
	\end{equation*}
\end{lemma}
\begin{proof}
	Using the Jacobi triple product identity \eqref{eq:theta_function_and_defns}, we have
	\begin{align*}
		(q^4;q^4)_{\infty}\left(
		\theta_{q^4}(-\kappa^4 q)
		-\kappa^{2}\theta_{q^4}(-\kappa^4 q^3)
		\right)&=
		\sum_{m\in \mathbb{Z}}
		(q^4)^{m(m-1)/2}\left( \kappa^{4m} q^m-\kappa^2 \kappa^{4m}q^{3m} \right).
	\end{align*}
	We aim to show that this sum is equal to
	$\sum_{j\in \mathbb{Z}}(-1)^{j+1} \kappa ^{2-2 j} q^{j(j-1)/2}$.
	Indeed, for $j=2m+1$, the $j$-th summand is equal to
	$(q^4)^{(-m)(-m-1)/2}\kappa^{-4m} q^{-m}$,
	and for $j=2m$, the $j$-th summand is
	$-(q^4)^{(-m)(-m-1)/2}\kappa^2 \kappa^{-4m}q^{-3m}$.
	This completes the proof.
\end{proof}

\begin{lemma}
	\label{lemma:a_binomial_identity}
	We have for all $n\in \tfrac12\mathbb{Z}_{\ge0}$:
	\begin{equation}
		\label{eq:a_binomial_identity}
		\sum_{i=-n}^n\sum_{j=-n}^n
		\frac{(-1)^{2n+i+j}q^{(i-j)^2/2+(i+j)/2}}
		{(q;q)_{n-i} (q;q)_{n+i} (q;q)_{n-j} (q;q)_{n+j}}=\frac{1}{q^n(q;q)_{2n}},
	\end{equation}
	where the sums are over $i,j$ of the same half-parity as $n$.
\end{lemma}
\begin{proof}
	Let us show that, for a fixed $i$, the sum over $j$
	in \eqref{eq:a_binomial_identity}
	is zero unless $i=-n$.
	Indeed,
	\begin{equation}
		\label{eq:a_binomial_identity_proof_1}
		\sum_{j=-n}^n
		\frac{(-1)^{n+j}q^{j(j+1)/2}z^j}
		{ (q;q)_{n-j} (q;q)_{n+j}}=
		\frac{z^{-n}q^{n(n-1)/2}}{(q;q)_{2n}}\prod_{r=-n+1}^{n}(1-zq^{r}),
	\end{equation}
	which is a direct consequence of the $q$-binomial theorem
	(it allows one to
	extract the coefficient of $z^{j}$ on the right-hand side).
	The polynomial on the right-hand side of
	\eqref{eq:a_binomial_identity_proof_1}
	vanishes at $z=q^{-i}$ for every $i$ with $-n+1\le i\le n$.
	Therefore, the sum over $j$ in
	\eqref{eq:a_binomial_identity} is zero unless $i=-n$.
	When $i=-n$, substituting $z=q^{-n}$ yields the desired result.
\end{proof}

\subsection{Representation of the identity and completeness}
\label{sub:representation_of_identity}

\begin{proposition}
	\label{prop:completeness_of_F_n_functions_app}
	For all $x,y\in \tfrac12\mathbb{Z}$
	with $x-y\in \mathbb{Z}$, we have
	\begin{equation}
		\label{eq:representation_of_identity}
		\sum_{n\in \frac12\mathbb{Z}_{\ge0}}
		\frac{\mathcal{F}_n(x)\mathcal{F}_n(y)}
		{\|\mathcal{F}_n\|^2
		_{\ell^2(\mathbb{Z})}}
		=\mathbf{1}_{x=y}.
	\end{equation}
	This implies that the orthonormal functions $\mathcal{F}_n/\|\mathcal{F}_n\|_{\ell^2(\mathbb{Z})}$,
	$n\in \tfrac12\mathbb{Z}_{\ge0}$,
	are complete in each of the spaces $\ell^2(\mathbb{Z})$ and $\ell^2(\mathbb{Z}')$.
\end{proposition}
\begin{proof}
	We need to show that
	\begin{equation}
		\label{eq:representation_identity_proof_1}
			\sum_{n\in \frac12\mathbb{Z}_{\ge0}}
			q^n(q;q)_{2n}\ssp \mathcal{F}_n(x)\mathcal{F}_n(y)
			=\mathbf{1}_{x=y}\ssp
			\frac
			{
			(-\kappa^2)(q;q)_\infty\ssp \theta_q(\kappa^{-2})}
			{\theta_{\sqrt{q}}(|\kappa|)^2}.
	\end{equation}
	We employ the so-called $q$-Mehler formula for the
	continuous $q^{-1}$-Hermite polynomials $H_n$
	\eqref{eq:continuous_q_hermite}
	with $q>1$, which is available from \cite[Theorem~2.1]{IsmailMasson1994}.
	In our notation, it reads
	\begin{equation}
		\label{eq:q_Mehler}
		\begin{split}
			&
			\sum_{m=0}^{\infty}
			\frac{H_m(\lambda\mid q^{-1})H_m(\mu\mid q^{-1})\ssp q^{m(m-1)/2}z^m}
			{(q;q)_m}
			\\&
			\hspace{120pt}
			=
			\frac{(-z\lambda\mu;q)_{\infty}(-z/(\lambda\mu);q)_{\infty}(-z\lambda/\mu;q)_{\infty}
			(-z\mu/\lambda;q)_{\infty}}
			{(z^2/q;q)_{\infty}}.
		\end{split}
	\end{equation}
	The summands in the left-hand side
	of \eqref{eq:representation_identity_proof_1}
	are expressed through the continuous $q^{-1}$-Hermite polynomials as
	\begin{equation*}
		\begin{split}
			q^n(q;q)_{2n}\ssp \mathcal{F}_n(x)\mathcal{F}_n(y)&=
			\frac{(-1)^{2x+2y}q^{x^2+y^2+(x+y)/2}
				|\kappa|^{2(x+y+1)}
			\sqrt{(1-\kappa^{-2}q^{-2x})(1-\kappa^{-2}q^{-2y})}}{\theta_{\sqrt q}(|\kappa|)^2}
			\\
			&\hspace{40pt}\times
			\frac{H_{2n}(\kappa q^x\mid q^{-1})
			H_{2n}(\kappa q^y\mid q^{-1})
			q^{2n(2n-1)/2}(-q)^{2n}}{(q;q)_{2n}}.
		\end{split}
	\end{equation*}
	Applying \eqref{eq:q_Mehler}
	with $z=-q$, $\lambda=\kappa q^x$, and $\mu=\kappa q^y$, we see that
	the sum in the left-hand side of \eqref{eq:representation_identity_proof_1}
	becomes
	\begin{equation}
		\label{eq:representation_identity_proof_2}
		\begin{split}
			&\frac{(-1)^{2x+2y}q^{x^2+y^2+(x+y)/2}
			|\kappa|^{2(x+y+1)}
			\sqrt{(1-\kappa^{-2}q^{-2x})(1-\kappa^{-2}q^{-2y})}}{\theta_{\sqrt q}(|\kappa|)^2}
			\\&\hspace{80pt}
			\times
			\frac{(\kappa^2 q^{x+y+1};q)_{\infty}(\kappa^{-2}q^{1-x-y};q)_{\infty}(q^{x-y+1};q)_{\infty}
			(q^{y-x+1};q)_{\infty}}
			{(q;q)_{\infty}}.
		\end{split}
	\end{equation}
	Thanks to the factors
	$(q^{x-y+1};q)_{\infty}
	(q^{y-x+1};q)_{\infty}$,
	if $x-y\in \mathbb{Z}$, then
	expression \eqref{eq:representation_identity_proof_2}
	vanishes unless $x=y$.

	To complete
	the proof of \Cref{prop:completeness_of_F_n_functions_app},
	it remains to simplify \eqref{eq:representation_identity_proof_2} for $y=x$
	and match it to the right-hand side of
	\eqref{eq:representation_identity_proof_1}.
	We have
	\begin{equation*}
		\begin{split}
			\eqref{eq:representation_identity_proof_2}
			&=\frac{q^{2x^2+x}
			|\kappa|^{4x+2}
			(1-\kappa^{-2}q^{-2x})}{\theta_{\sqrt q}(|\kappa|)^2}
			(\kappa^2 q^{2x+1};q)_{\infty}(\kappa^{-2}q^{1-2x};q)_{\infty}(q;q)_{\infty}
			\\&=
			\frac{q^{2x^2+x}
			|\kappa|^{4x+2}(q;q)_\infty\ssp \theta_q(\kappa^{-2}q^{-2x})}{\theta_{\sqrt q}(|\kappa|)^2}
			\\&=
			\frac{
			(-\kappa^2)(q;q)_\infty\ssp \theta_q(\kappa^{-2})}{\theta_{\sqrt q}(|\kappa|)^2},
		\end{split}
	\end{equation*}
	where in the last line we used the quasi-periodicity \eqref{eq:theta_function_and_defns}.
\end{proof}

\begin{remark}
	\label{rmk:no_symmetry_shift_in_half}
	When $x$ and $y$ have different half-parities, the product
	$(q^{x-y+1};q)_{\infty}(q^{y-x+1};q)_{\infty}$, and hence
	the whole expression
	\eqref{eq:representation_identity_proof_2}, does not
	vanish. Consequently, the functions $\mathcal{F}_n(x)$ are
	\emph{not complete} in the larger space
	$\ell^{2}(\tfrac12\mathbb{Z})$.
\end{remark}

\section{\texttt{Mathematica} code}
\label{sec:Mathematica_code_appendix}

Here we present the \texttt{Mathematica} definitions of the two-dimensional
$q$-Racah kernel~$\widetilde{K}$
\eqref{eq:K_kernel_two_dimensional},
the pre-limit barcode kernel $K_{(L)}^{\mathrm{barcode}}$
\eqref{eq:prelimit_barcode_kernel},
and the conjectural limit $\mathcal{K}^{\mathrm{barcode}}(s,t)$
\eqref{eq:conjectural_limit_of_full_barcode_kernel}.
These code snippets are involved in the numerical verification of
\Cref{conj:barcode_density_limit,conj:conjectural_limit_of_full_barcode_kernel,conj:barcode_limit_via_F}
which is described in
\Cref{sub:s7_numerics}.
All \texttt{Mathematica} formulas are written only for the first parameter zone in the hexagon \eqref{eq:hexagon_zone_1}.
Note that in code, we refer to $N$ as \texttt{NN}.

\lstset{
  basicstyle=\scriptsize\ttfamily,
  columns=fullflexible,
	tabsize=2,
  keepspaces=true,
}
\begin{lstlisting}[language=Mathematica, caption={The orthonormal $q$-Racah polynomials $f_n^t(x)$} \eqref{eq:qRacah_orthonormal_basis_f_n}, label=lst:fnt, backgroundcolor=\color{gray!10}]
w[t_, x_] := (
		(q^(1 - 2 NN - T))^-x (1 - q^(1 - S - t + 2 x) \[Kappa]^2)
		QPochhammer[q^(1 - NN - S), q, x] QPochhammer[q^(1 - NN - t), q, x]
		QPochhammer[q^(1 - S - t) \[Kappa]^2, q, x] QPochhammer[q^(1 - T) \[Kappa]^2, q, x]
) / (
		(1 - q^(1 - S - t) \[Kappa]^2) QPochhammer[q, q, x] QPochhammer[q^(1 - S - t + T), q, x]
		QPochhammer[q^(1 + NN - S) \[Kappa]^2, q, x] QPochhammer[q^(1 + NN - t) \[Kappa]^2, q, x]
)

R[n_, t_, x_] := QHypergeometricPFQ[
		{q^-n, q^(1 + n - 2 NN - T), q^-x, q^(1 - S - t + x) \[Kappa]^2},
		{q^(1 - NN - S), q^(1 - NN - t), q^(1 - T) \[Kappa]^2},
		q, q
]

h[n_, t_] := (
		(1 - q^(1 - 2 NN - T)) (q^(1 - S - t) \[Kappa]^2)^n QPochhammer[q, q, n]
		QPochhammer[q^(2 - 2 NN - T), q, -1 + NN + t] QPochhammer[q^(1 - NN + S - T), q, n]
		QPochhammer[q^(1 - NN + t - T), q, n] QPochhammer[q^(1 - 2 NN) / \[Kappa]^2, q, n]
		QPochhammer[q^(-NN + S) / \[Kappa]^2, q, -1 + NN + t]
) / (
		(1 - q^(1 + 2 n - 2 NN - T)) QPochhammer[q^(1 - NN - S), q, n] QPochhammer[q^(1 - NN - t), q, n]
		QPochhammer[q^(1 - 2 NN - T), q, n] QPochhammer[q^(1 - NN + S - T), q, -1 + NN + t]
		QPochhammer[q^(1 - 2 NN) / \[Kappa]^2, q, -1 + NN + t] QPochhammer[q^(1 - T) \[Kappa]^2, q, n]
)

f[n_, t_, x_] := R[n, t, x] Sqrt[w[t, x] / h[n, t]]
\end{lstlisting}

\begin{lstlisting}[language=Mathematica, caption={The two-dimensional kernel $\widetilde K(s,x;t,y)$ \eqref{eq:K_kernel_two_dimensional}}, label=lst:Ktilde, backgroundcolor=\color{gray!10}]
CTilde[n_, t_] := Sqrt[(q^(n - t - NN) - 1) * (1 - q^(T + NN - t - n - 1))]

Ker[s_, x_, t_, y_] :=
	Which[
		s >= t,
			Sum[
				1/Product[CTilde[n, j], {j, t, s - 1}] * f[n, s, x] * f[n, t, y],
				{n, 0, NN - 1}
			],
		s < t,
			-Sum[
				Product[CTilde[n, j], {j, s, t - 1}] * f[n, s, x] * f[n, t, y],
				{n, NN, NN + s - 1}
			]
	]
\end{lstlisting}

\begin{lstlisting}[language=Mathematica, caption={The inverse Kasteleyn matrix \eqref{eq:inverse_kasteleyn_matrix}}, label=lst:KastInv, backgroundcolor=\color{gray!10}]
wLozenge[j_] := \[Kappa] q^(j - (S + 1)/2) - 1/(\[Kappa] q^(j - (S + 1)/2))

wTilde[t_, x_] := (
		(-1)^(t + S) * q^(x * (2 * NN + T - 1)) * (1 - \[Kappa]^2 * q^(2 * x - t - S + 1))
) / (
		QPochhammer[q, q, x] * QPochhammer[q, q, T - S - t + x] *
		QPochhammer[q^(-1), q^(-1), t + NN - x - 1] *
		QPochhammer[q^(-1), q^(-1), S + NN - x - 1] *
		QPochhammer[\[Kappa]^2 * q^(x - T + 1), q, T + NN - t] *
		QPochhammer[\[Kappa]^2 * q^(x - t - S + 1), q, NN + t]
)

G[t_, x_] := (
		(-1)^(x) * \[Kappa]^(-t) * q^(x * (T + NN - t - 1) + t * (S/2 - 1/2) + t * (t + 1)/4) *
		(1 - \[Kappa]^2 * q^(2 * x - t - S + 1))
) / (
		QPochhammer[q^(-1), q^(-1), S + NN - 1 - x] *
		QPochhammer[q, q, T - S + x - t] *
		QPochhammer[\[Kappa]^2 * q^(x - T + 1), q, T + NN - t]
) / Sqrt[wTilde[t, x]]

KastInv[s_, x_, t_, y_] := G[s, x] / G[t, y] / wLozenge[x - s/2 + 1] *
  (If[x == y && s == t, 1, 0] - Ker[s, x, t, y])
\end{lstlisting}

\begin{lstlisting}[language=Mathematica, caption={The pre-limit barcode kernel $K_{(L)}^{\mathrm{barcode}}(s,t)$ \eqref{eq:prelimit_barcode_kernel}}, label=lst:KLBarcode, backgroundcolor=\color{gray!10}]
T := 8 L; NN := 4 L; S := 4 L;

Kbar[s_, t_] := KastInv[2 L + s, 3 L, 2 L + t - 1, 3 L]
\end{lstlisting}

\begin{lstlisting}[language=Mathematica, caption={The orthogonal functions $\mathcal{F}_n(x)$ \eqref{eq:F_n_functions_app}}, label=lst:F_n, backgroundcolor=\color{gray!10}]
th[q_, z_] := QPochhammer[z, q] * QPochhammer[q/z, q]

F[n_, x_] := Sqrt[1 - \[Kappa]^(-2) * q^(-2 * x)] * Sum[
		(-1)^(n + i) * q^(-i/2) / QPochhammer[q, q, n - i] /
		QPochhammer[q, q, n + i] / th[Sqrt[q], q^(x + i + 1/2) * \[Kappa]/I],
		{i, -n, n}
]

fsqnorm[n_] := (
		QPochhammer[q, q] * (-\[Kappa]^2) * th[q, \[Kappa]^(-2)] / q^n /
		QPochhammer[q, q, 2 * n] / th[Sqrt[q], \[Kappa]/I]^2
)
\end{lstlisting}

\begin{lstlisting}[language=Mathematica, caption={The partial sum in the conjectural limiting barcode kernel $\mathcal{K}^{\mathrm{barcode}}(s,t)$ \eqref{eq:conjectural_limit_of_full_barcode_kernel}}, label=lst:KLimitingBarcode, backgroundcolor=\color{gray!10}]
prefactorFromG[s_, t_] := (
		-q^(1 - s) * Sqrt[-\[Kappa]^2 / (1 - \[Kappa]^2 * q^(2 - t)) / (1 - \[Kappa]^2 * q^(1 - s))]
)

KLimitBarcodeHalf[s_, t_][M_] := (
		(-1)^(s - t) * q^((M + 1) * (s - t)) * q^((s - 1)/2) * prefactorFromG[s, t] *
		(-Sum[
				(-q^n)^(t - 1 - s) * F[n, -(t - 2)/2] * F[n, -(s - 1)/2] / fsqnorm[n],
				{n, 0, M + 1/2, 1/2}
		])
)

KLimitingBarcode[s_, t_][M_] := If[s >= t,
		KLimitBarcodeHalf[s, t][M],
		KLimitBarcodeHalf[t, s][M]
]
\end{lstlisting}

\input{SORTED_July2025_waterfall.bbl}

\medskip

\textsc{a. knizel, barnard college, columbia university, new york, ny, usa}

e-mail: \texttt{aknizel@barnard.edu}

\medskip

\textsc{l. petrov, university of virginia, charlottesville, va, usa}

e-mail: \texttt{lenia.petrov@gmail.com}

\end{document}